\documentclass[10pt,reqno]{amsart}

\usepackage{latexsym,amsfonts,amssymb,exscale,enumerate,comment}
\usepackage{amsmath,amsthm,amscd}

\usepackage{url}
\usepackage[bookmarks=true,%
    colorlinks=true,%
    linkcolor=blue,%
    citecolor=blue,%
    filecolor=blue,%
    menucolor=blue,%
    urlcolor=blue,%
    breaklinks=true]{hyperref}

\input xy
\usepackage[all]{xy}
\xyoption{line}
\xyoption{arrow}
\xyoption{color}
\SelectTips{cm}{}

\usepackage{tikz}
\usetikzlibrary{decorations.markings}
\usetikzlibrary{decorations.pathreplacing}
\tikzstyle directed=[postaction={decorate,decoration={markings,
    mark=at position #1 with {\arrow{>}}}}]

\tikzset{->-/.style={decoration={
  markings,
  mark=at position #1 with {\arrow{>}}},postaction={decorate}}}

\tikzset{middlearrow/.style={
        decoration={markings,
            mark= at position 0.5 with {\arrow{#1}} ,
        },
        postaction={decorate}
    }
}

\usepackage{graphicx}
\usepackage{color}

\newcommand{\END}{{\rm END}}

\theoremstyle{plain}
\newtheorem{theorem}{Theorem}

\newtheorem{corollary}[theorem]{Corollary}
\newtheorem{proposition}[theorem]{Proposition}
\newtheorem{lemma}[theorem]{Lemma}

\theoremstyle{definition}

\newtheorem{definition}[theorem]{Definition}

\theoremstyle{definition}
\newtheorem{remark}[theorem]{Remark}

\numberwithin{equation}{section}
\numberwithin{theorem}{section}

\newcommand{\maps}{\colon}
\newcommand{\und}[1]{\underline{#1}}

\newcommand{\refequal}[1]{\xy {\ar@{=}^{#1}
(-1,0)*{};(1,0)*{}};
\endxy}

\newcommand{\cat}[1]{\ensuremath{\mbox{\bfseries {\upshape {#1}}}}}

\newcommand{\To}{\Rightarrow}

\newcommand{\Hom}{{\rm Hom}}
\newcommand{\HOM}{{\rm HOM}}
\renewcommand{\to}{\rightarrow}

\def\Ind{{\mathrm{Ind}}}

\def\dmod{{\mathrm{-dmod}}}   
\def\stmod{{\mathrm{-\underline{mod}}}}

\def\Id{\mathrm{Id}}

\def\mf{\mathfrak}

\numberwithin{equation}{section}

\let\hat=\widehat

\let\epsilon=\varepsilon

\usepackage{bbm}

\def\N{{\mathbbm N}}

\def\Z{{\mathbbm Z}}

\def\cal#1{\mathcal{#1}}%
\def\1{\mathbbm{1}}%
\def\nn{\notag}
\newcommand{\scs}{\scriptstyle}

\def\la{\langle}
\def\ra{\rangle}

\renewcommand{\l}{\lambda}

\usepackage{bbm}

\def\cal#1{\mathcal{#1}}

\newcommand\nc{\newcommand}
\nc\rnc{\renewcommand}
\nc\Kar{\operatorname{Kar}}
\nc\End{\operatorname{End}}

\nc\Sym{\operatorname{Sym}}
\newcommand{\bigb}[1]{
\begin{tikzpicture}
\node[draw,  fill=white,rounded corners=4pt,inner sep=3pt] (X) at (0,.75) {$\scs #1$};
\end{tikzpicture}}

\allowdisplaybreaks

\nc\Omit[1]{}

\nc\sfk{\mathsf{k}}
\nc\Q{\mathbb{Q}}
\nc\IZ{\mathbb{Z}}

\nc\NH[1]{\mathsf{NH}_{#1}}
\nc\ONH{\mathsf{ONH}}
\nc\OPol{\mathsf{OPol}}
\nc\OL{\mathsf{O}\Lambda}

\nc\symL[1]{\Lambda_{#1}}
\nc\Pol[1]{\mathsf{P}_{#1}}

\nc\ten{\otimes}
\nc\Wedge{\bigwedge}
\nc\lp{\left(}
\nc\rp{\right)}
\nc\wdt[1]{\widetilde{#1}}
\nc\id{\operatorname{id}}
\nc\aand{\qquad\mbox{and}\qquad}

\newcommand{\U}{\dot{{\bf U}}}
\newcommand{\UA}{{_{\cal{A}}\dot{{\bf U}}}}
\def\Upi{{\bf U}_{q,\pi}}
\newcommand{\B}{\dot{\mathbb{B}}}

\newcommand{\oUcat}{\mf{U}}
\newcommand{\oUqp}{\mf{U}_{q,\pi}}
\newcommand{\oUbar}{\underline{\mf{U}}_{q,\pi}}
\newcommand{\oUdot}{\dot{\underline{\mf{U}}}_{q,\pi}}
\newcommand{\Ucat}{\cal{U}}
\newcommand{\UcatD}{\dot{\cal{U}}}
\renewcommand{\l}{\lambda}
\newcommand{\onel}{{\mathbf 1}_{\lambda}}

\newcommand{\onen}{{\mathbf 1}_{\l}}
\newcommand{\onenn}[1]{{\mathbf 1}_{#1}}

\newcommand{\Udotpi}{\dot{U}_{q,\pi}}

\newcommand{\AUdotpi}{_{\cal{A}}\dot{U}_{q,\pi}}

\newcommand{\E}[1]{E^{( #1 )}}
\newcommand{\F}[1]{F^{( #1 )}}

\newcommand{\bbpef}[1]{\xybox{%
  (-6,0)*{};
  (6,0)*{};
  (-4,0)*{}="t1";
  (4,0)*{}="t2";
  "t1";"t2" **\crv{(-4,-6) & (4,-6)}; ?(.15)*\dir{>} ?(.9)*\dir{>}
   ?(.5)*\dir{}+(0,-2)*{\scriptstyle{#1}};
}}
\newcommand{\bbpfe}[1]{\xybox{%
  (-6,0)*{};
  (6,0)*{};
  (-4,0)*{}="t1";
  (4,0)*{}="t2";
  "t2";"t1" **\crv{(4,-6) & (-4,-6)}; ?(.15)*\dir{>} ?(.9)*\dir{>}
  ?(.5)*\dir{}+(0,-2)*{\scriptstyle{#1}};
}}

\newcommand{\bbcfe}[1]{\xybox{%
  (-6,0)*{};
  (6,0)*{};
  (-4,0)*{}="t1";
  (4,0)*{}="t2";
  "t1";"t2" **\crv{(-4,6) & (4,6)}; ?(.15)*\dir{>} ?(.9)*\dir{>}
  ?(.5)*\dir{}+(0,2)*{\scriptstyle{#1}};
}}
\newcommand{\bbcef}[1]{\xybox{%
  (-6,0)*{};
  (6,0)*{};
  (-4,0)*{}="t1";
  (4,0)*{}="t2";
  "t2";"t1" **\crv{(4,6) & (-4,6)}; ?(.15)*\dir{>}
  ?(.9)*\dir{>} ?(.5)*\dir{}+(0,2)*{\scriptstyle{#1}};
}}

\newcommand{\lbub}[1]{
\xybox{
  (-3,0)*{};(3,0)*{} **\crv{(-3,4) & (3,4)}  ?(1)*\dir{>};
  (3,0)*{};(-3,0)*{} **\crv{(3,-4) & (-3,-4)} ?(.85)*{\bullet}+(0,-3.5)*{\scs {\ast+#1}};
  (-5,-5)*{}; (5,5)*{};
}}
\newcommand{\rbub}[1]{
\xybox{
  (-3,0)*{};(3,0)*{} **\crv{(-3,4) & (3,4)} ?(.0)*\dir{<};
  (3,0)*{};(-3,0)*{} **\crv{(3,-4) & (-3,-4)} ?(.1)*{\bullet}+(0,-3.5)*{\scs {\ast+#1}};
  (-5,-5)*{}; (5,5)*{};
}}
\newcommand{\lbbub}[1]{
\xybox{
  (-3,0)*{};(3,0)*{} **\crv{(-3,4) & (3,4)} ?(1)*\dir{>};
  (3,0)*{};(-3,0)*{} **\crv{(3,-4) & (-3,-4)} ?(.9)*{\bullet}+(0,-3.5)*{\scs {#1}};
  (-5,-5)*{}; (5,5)*{};
}}
\newcommand{\rbbub}[1]{
\xybox{
  (-3,0)*{};(3,0)*{} **\crv{(-3,4) & (3,4)} ?(0)*\dir{<};
  (3,0)*{};(-3,0)*{} **\crv{(3,-4) & (-3,-4)} ?(.1)*{\bullet}+(0,-3.5)*{\scs {#1}};
  (-5,-5)*{}; (5,5)*{};
}}

\newcommand{\smccbub}[1]{
\xybox{%
 (-5,0)*{};
  (5,0)*{};
  (-2.5,0)*{}="t1";
  (2.5,0)*{}="t2";
  "t2";"t1" **\crv{(2.5,3.3) & (-2.5,3.3)};   ?(1)*\dir{>}; ?(.3)*\dir{};
  "t2";"t1" **\crv{(2.5,-3.3) & (-2.5,-3.3)};
  (0,0)*{  {#1}}
}}

 \newcommand{\smcbub}[1]{
\xybox{%
 (-5,0)*{};
  (5,0)*{};
  (2.5,0)*{}="t1";
  (-2.5,0)*{}="t2";
  "t2";"t1" **\crv{(-2.5,3.3) & (2.5,3.3)};   ?(1)*\dir{>}; ?(.3)*\dir{};
  "t2";"t1" **\crv{(-2.5,-3.3) & (2.5,-3.3)};
  (0,0)*{  {#1}}
}}

\title{
DG structures on odd categorified quantum $sl(2)$
}

\begin{document}

\setcounter{tocdepth}{1}

\author{Ilknur Egilmez}
\email{egilmez@usc.edu}
\address{Department of Mathematics\\ University of Southern California \\ Los Angeles, CA}

\author{Aaron D. Lauda}
\email{lauda@usc.edu}
\address{Department of Mathematics\\ University of Southern California \\ Los Angeles, CA}

\date{September 10th, 2020}

\begin{abstract}
We equip Ellis and Brundan's version of the odd categorified quantum group for $sl(2)$ with a differential giving it the structure of a graded dg-2-supercategory.  The presence of the super grading gives rise to two possible decategorifications of the associated dg-2-category.  One version gives rise to a categorification of quantum $sl(2)$ at a fourth root of unity, while the other version produces a subalgebra of quantum $gl(1|1)$ defined over the integers.  Both of these algebras appear in connection with quantum algebraic approaches to the Alexander polynomial.
\end{abstract}

\maketitle

\tableofcontents

%
\section{Introduction}
%

%
\subsection{Motivations from link homology theory}
%

Khovanov homology, categorifying a certain normalization of the Jones polynomial~\cite{Kh1,Kh2}, is the simplest of a family of link homology theories associated to quantum groups and their representations. Surrounding Khovanov homology is an intricate system of related combinatorial and geometric ideas. Everything from extended 2-dimensional TQFTs~\cite{Kh2,LP3,CMW}, planar algebras~\cite{BN1,BN2},  category $\mathcal{O}$~\cite{Strop1,Strop2,BrSt3,BFK}, coherent sheaves on  quiver varieties~\cite{CK01}, matrix factorizations~\cite{KhR,KhR2}, homological mirror symmetry~\cite{SeSm}, arc algebras~\cite{Kh2,ChK,Strop1,BrSt1
}, Springer varieties~\cite{KhSp,Strop1,SW}, stable homotopy theory~\cite{LS,LS2,LS3}, and 5-dimensional gauge theories~\cite{GSV,Witten,Witten2} appear in descriptions of Khovanov homology, among many other constructions.

Given that Khovanov homology provides a nexus bridging the sophisticated structures described above, it is  surprising to discover that there exists a distinct categorification of the Jones polynomial.
 Ozsv\'{a}th, Rasmussen, Szab\'{o} found an {\em odd} analogue of Khovanov homology~\cite{ORS} that agrees with the original Khovanov homology when coefficients are taken modulo 2.  Both of these theories categorify the Jones polynomial, and results of Shumakovitch~\cite{Shum} show that these categorified link invariants are not equivalent.


The discovery of odd Khovanov homology was motivated by the existence of a spectral sequence from
ordinary Khovanov homology to the Heegaard Floer homology of the double branch cover \cite{OS-branched} with $\Z_2$ coefficients. Odd Khovanov homology was defined in an attempt to extend this spectral sequence to $\Z$ coefficients, rather than $\Z_2$.  Indeed, in \cite{ORS} they conjecture that for a link $K$ in $S^3$,  there is a spectral sequence whose $E^2$ term is the reduced odd Khovanov homology ${\rm Khr}(K)$ of $K$ and whose $E^{\infty}$ term is  the Heegaard-Floer homology $\hat{HF}(-\Sigma(K))$ of the branched double cover $\Sigma(K)$ with the orientation reversed (with coefficients in $\Z$).
\[
 \xy
  (-20,7)*+{{\rm Khr}(K)}="1";
  (-20,-7)*+{{\rm OKhr}(K)}="3";
    (20,0)*+{\hat{HF}(-\Sigma(K))}="2";
    {\ar@{~>}^{\Z/2} "1";"2" };
    {\ar@{~>}_{\Z ?} "3";"2" };
 \endxy
\]
A related version of this conjecture was proven in the context of instanton homology in
\cite{Sca
}.

There are now a number of spectral sequences connecting variants of Khovanov homology to variants of Floer homology
\cite{
Ras2,
Szabo1, 
Bloom,
KM, 
Roberts, 
Hendricks, 
Beier, 
Baldwin1, 
Baldwin2
}.
For even Khovanov homology there are many interesting connections with knot Floer-homology $\widehat{HFK}(K)$.  This is a bigraded homology for knots and links
\[
 \widehat{HFK}(K) = \bigoplus_{m,a\in\Z} \widehat{HFK}_m(K,a)
\]
where $m$ is called the Maslov (or homological) grading and $a$ is the Alexander grading.  The graded Euler characteristic of $\widehat{HFK}(K)$ is the  Alexander polynomial
\[
\sum_{m,a\in \Z} (-1)^m t^a \cdot {\rm rank}_{\Z}\left(
\widehat{HFK}_m(K,a)
\right)
= \Delta_K(t).
\]

Many of the spectral sequences listed above arise via a collapse of the bigraded homology groups to a single $\delta$-grading.  For Khovanov homology the $\delta$-grading is given by $\delta = h-q/2$,
where $q$ denotes the quantum grading and $h$ the homological.  On $\widehat{HFK}$ the $\delta$-grading is $\delta = a - m$.
Rasmussen conjectured a spectral sequence between the singly $\delta$-graded Khovanov homology $Kh_{\delta}(K)$ and the $\delta$-graded knot Floer homology $\widehat{HFK}_{\delta}(K)$~\cite{Ras2}.  Under the collapse of grading the graded Euler characteristic becomes an integer rather than a polynomial.  It is interesting to note that if we set $q = \sqrt{-1}$ in the Euler characteristic formula
\[
 \sum_{i,j} (-1)^i q^j \mathrm{rk}(\mathrm{Kh}_{i,j}) |_{q=\sqrt{-1}}
 =
 \sum_{i,j} (-1)^{i-j/2} \mathrm{rk}({\mathrm{Kh}_{i,j}})
\]
we recover the Euler characteristic of the $\delta$-graded Khovanov homology theory.  Similarly, in $\widehat{HFK}$ where $\delta = a - m$, so that the parameters are related by $q^2 = t$, we see that $q=\sqrt{-1}$ corresponds to $t=-1$, so the Euler characteristic specializes to
 \[
\sum_{i,a\in \Z} (-1)^{i + a} \cdot {\rm rank}_{\Z}\left(
\widehat{HFK}_i(K,a)
\right)
= \Delta_K(-1).
\]

The $t=-1$ evaluation of Alexander polynomial is equal to the knot determinant $\det(K)$.  This invariant has another categorification via the Heegaard-Floer 3-manifold homology of the branched double cover of $K$,
\[
\chi\left( \widehat{HF}(\Sigma(K)) \right) = |H^2(\Sigma(K), \Z)| = \det(K) =| \Delta_K(-1)|
\]
see \cite[Section 3]{OS-branched}.    This variant of Heegaard-Floer homology is the target of the conjectured spectral sequence from odd Khovanov homology discussed above.

\subsection{Quantum algebra and a zoo of quantum invariants}
These connection between varients of Heegaard-Floer homology and even/odd Khovanov homology are somewhat striking given that these invariants are defined in very different ways.   However, quantum algebra sheds some light as to why such a connection is less surprising.  It is well known that the Jones polynomial can be interpreted as a quantum invariant associated to the quantum group for $\mf{sl}_2$ and its two dimensional representation.  Varying the semisimple Lie algebra $\mf{g}$ and the irreducible representations coloring the strands of a link, one arrives at a whole family of quantum invariants.

The Alexander-Conway function $\nabla_L(t_1,\dots, t_k)$ for a $k$ component link $L$ is a rational function in variables $t_1, \dots, t_k$.  Similarly, the Alexander polynomial $\Delta_{L}(t_1, \dots, t_k)$ is a Laurent polynomial in variables $t_1^{\frac{1}{2}}, \dots, t_1^{\frac{1}{2}}$.  They are related by
\[
\nabla_L(t_1,\dots, t_k) = \Delta_{L}(t_1^2, \dots t_k^2) \quad \text{if $k>1$, and } \quad
\nabla_L(t) = \frac{\Delta_{L}(t^2)}{t-t^{-1}}.
\]
The Alexander-Conway polynomial can be formulated as a (non-semisimple) quantum invariant in several ways.  One formulation realizes $\nabla_L$  using the quantum group associated to the super Lie algebra $\mf{gl}(1|1)$~\cite{RS93}.  Murakami gave a construction using quantum $\mf{sl}_2$ with the quantum parameter specialized to a fourth root of unity~\cite{Mur92,Mur93}. Kauffman and Saleur give a construction based on quantum $\mf{sl}(1|1)$.

A comparison and review of  the ${\bf U}_{\sqrt{-1}}(\mf{sl}_2)$ and ${\bf U}_q(\mf{gl}(1|1))$ Reshetikhin-Turaev functors are studied in \cite{Vir06}.  In this work, Viro shows that there is a `$q$-less subalgebra' ${\bf U}^1$ of ${\bf U}_q(\mf{gl}(1|1))$ that is responsible for producing the Reshetikhin-Turaev functor that is closely related to the one coming from ${\bf U}_{\sqrt{-1}}(\mf{sl}_2)$.  Similarly, an algebra that can be defined over $\Z$ also appears in the Kauffman-Saleur ${\bf U}_q(\mf{sl}(1|1))$ construction of the Alexander-Conway polynomial $\nabla_K$ via a specialization ($\mathit{\lambda}=1$ in their notation, see \cite[Equation (2.1)]{KaufS}), which corresponds in our notation to working with the subalgebra $\U(\mf{sl}(1|1))1_{1}$ of $\U(\mf{sl}(1|1))$, see section~\ref{sec:qless}.   The quantum parameter is not needed in the definition of this algebra, it only arises in the coalgebra structure when one acts on tensor product representations.

Connections between the Alexander invariant and the Jones polynomial then arise via an observation by Kauffman and Saleur that the
$R$-matrix for braiding the fundamental representations of $\mf{sl}_2$ and $\mf{sl}(1|1)$ agree when evaluated at $q=\sqrt{-1}$.  This  implies an identification of quantum invariants
\begin{equation}
 J_K(q)|_{q=\sqrt{-1}} = \nabla_K(t)_{t=\sqrt{-1}} = \Delta_K(t)_{t=-1}.
\end{equation}
Our aim in this article is to lay the groundwork for a higher representation theoretic categorification of the knot determinant $| \Delta_K(-1)|$ by categorifying the quantum algebras used to define it.  Our approach provides a new perspective on connections between these different approaches via the theory of covering Kac-Moody algebras.

%
\subsection{The oddification program}
%
The so called `oddification'  program~\cite{LauR} in higher representation theory grew out of an attempt to provide a representation theoretic explanation for a number of phenomena observed in connection with odd Khovanov homology.   The idea is that Khovanov homology shares many connections throughout mathematics and theoretical physics, suggesting that many of the other fundamental structures connected with Khovanov homology may also have odd analogs.
The oddification program looks for {\em odd} analogs of structures that are typically non-commutative, having the same graded ranks as traditional objects and becoming isomorphic when coefficients are reduced modulo two.  Often the odd world provides the same combinatorial relationships in a non-commutative setting.

The nilHecke algebra plays a central role in the theory of categorified quantum groups, giving rise to an integral categorification of the negative half of ${\bf U}_q(\mf{sl}_2)$~\cite{Lau1,KL3,Rou2}.  An oddification of this algebra was defined in \cite{EKL} which can be viewed as an algebra of operators on a skew polynomial ring.  The invariants under this action define an odd version of the ring of symmetric functions~\cite{EK,EKL}.  The odd nilHecke algebra also gives rise to ``odd" noncommutative analogs of the cohomology of Grassmannians and Springer varieties~\cite{LauR,EKL}.  It also fits into a 2-categorical structure~\cite{Lau-odd,BE2} giving an odd analog of the categorification of the entire quantum group ${\bf U}_q(\mf{sl}_2)$.
In each of these cases, the structures possess combinatorics quite similar to those of their even counterparts.  When coefficients are reduced modulo two the theories become identical, but the odd analogues possess an inherent non-commutativity making them distinct from the classical theory.

The odd nilHecke algebra appears to be connected to a number of important objects in traditional representation theory.
It was independently introduced by Kang, Kashiwara and Tsuchioka~\cite{KKT}
starting from the different perspective of trying to develop super analogues of KLR algebras.  Their quiver Hecke superalgebras become isomorphic to affine Hecke-Clifford superalgebras or affine Sergeev superalgebras after a suitable completion, and the $\mf{sl}_2$ case of their construction is isomorphic to the odd nilHecke algebra.   Cyclotomic quotients of quiver Hecke superalgebras supercategorify certain irreducible representations of Kac-Moody algebras~\cite{KKO,KKO2}. A closely related spin Hecke algebra associated to the affine Hecke-Clifford superalgebra appeared in earlier work of Wang~\cite{Wang} and many of the essential features of the odd nilHecke algebra including skew-polynomials appears much earlier in this and related works on spin symmetric groups~\cite{KW1,KW2,KW4}.

%
\subsection{Covering Kac-Moody algebras}
%

Clark, Hill, and Wang showed that the odd nilHecke algebra and its generalizations fit into a framework they called {\em covering Kac-Moody algebras}~\cite{HillWang,ClarkWang, CHW,CHW2}.  Their idea was to decategorify the supergrading on the odd nilHecke algebra by introducing a parameter $\pi$ with $\pi^2=1$.  The covering Kac-Moody algebra is then defined over $\Q(q)[\pi]/(\pi^2-1)$ for certain very specific families of Kac-Moody Lie algebras.
The specialization to $\pi=1$ gives the quantum enveloping algebra of a Kac-Moody algebra and the specialization to $\pi=-1$ gives a quantum enveloping algebra of a Kac-Moody superalgebra. This idea led   to a novel bar involution $\overline{q} =\pi q^{-1}$ allowing the first construction of canonical bases for Lie superalgebras~\cite{CHW2,ClarkWang}.
In the simplest case, the covering algebra $\Upi$ can be seen as a simultaneous generalization of the modifed quantum group $\U(\mf{sl}_2)$ and the modified quantum Lie superalgebra $\U(\mathfrak{osp}(1|2))$.  This relationship is illustrated below.
\[\xy
  (0,10)*+{\U_{q,\pi}}="t";%
  (-15,-5)*+{\U(\mathfrak{sl}_2)}="bl";
  (15,-5)*+{\U(\mathfrak{osp}(1|2))}="br";
  {\ar_{\pi \to 1} "t";"bl"};
  {\ar^{\pi \to -1} "t";"br"};
 \endxy\]

Covering Kac-Moody algebras are not an $\mf{sl}_n$ phenomenon.  In finite type, the covering Kac-Moody algebras ${\bf U}_{q,\pi}(\mf{g})$ can be defined connecting the superalgebra of the anisotropic Lie superalgebra $\mf{g}=\mf{osp}(1|2n)$ with the quantum Kac-Moody algebra $\mf{g}=\mf{so}(2n+1)$ obtained by forgetting the parity in the root datum~\cite{CHW,HillWang}.  In particular, the only finite type family of covering Kac-Moody algebras ${\bf U}_{q,\pi}(\mf{g})$ have a $\pi=1$ specialization equal to the quantum eveloping algebra ${\bf U}_q(\mf{so}(2n+1))$ and the $\pi=-1$ specialization the quantum superalgebra ${\bf U}_q(\mf{osp}(1|2n)$.  The connection to $\mf{sl}_2$ only arises because of the Lie algebra coincidence $\mf{sl}_2 \cong \mf{so}(3)$.

The algebra/superalgebra pairs connected by covering theory are closely connected by the theory of {\em twistors} developed by Clark, Fan ,Li, Wang~\cite{FanLi,Clark-III}.   Denote by $\mathbf{t}$ a square root of $-1$, and let $\U[\mathbf{t}]$ denote the algebra $\U_{q,\pi}$ with scalars extended by $\mathbf{t}$.  Then the twistor associated to a covering algebra $\U_{q,\pi}(\mf{g})$ gives an isomorphism
\begin{equation} \label{eq:twistor-iso}
 \dot{\Psi} \maps \U[\mathbf{t}]|_{\pi=-1} \longrightarrow \U[\mathbf{t}]|_{\pi=1}
\end{equation}
sending $\pi \mapsto -\pi$ and thereby switching between a quantum group and its super analog.  This map sends $q\mapsto \mathbf{t}^{-1}q$. Hence, $\U[\mathbf{t}]_{\pi=1}$ and $\U[\mathbf{t}]_{\pi=-1}$ can be regarded as two different rational forms of a common algebra $\U[\mathbf{t}]$.  These two rational forms each admit their own distinct integral forms.

The twistor isomorphism \eqref{eq:twistor-iso} has implications for the corresponding quantum link invariants.  Blumen showed that $\mf{osp}(1|2n)$ and $\mf{so}(2n+1)$ invariants colored by the standard $(2n+1)$-dimensional representations agree up to a substitution of variable~\cite{Bl}. To a knot or link $K$, Clark greatly extended this observation by defining covering colored knot invariants $J_K^{\l}(q,\mathbf{t})$ associated to ${\bf U}_{q,\pi}(\mf{g})$ and a dominant integral weight $\l \in X^+$.  These knot invariants take values in a larger field $\Q(q, \mathbf{t})^{\tau}$ with $\tau^2 = \pi$.  They have the property of simultaneously generalizing the colored $\mf{so}(2n+1)$ quantum invariant and the $\mf{osp}(1|2n)$ super quantum invariant.  If we define
${}_{\mf{so}}J_k^{\l}(q) := J_K^{\l}(q,1)$ and ${}_{\mf{osp}}J_k^{\l}(q) := J_K^{\l}(q,\mathbf{t})$ then Clark shows~\cite[Theorem 4.24]{Clark-knot} that the twistor isomorphism \eqref{eq:twistor-iso} gives rise to an identification of quantum knot invariants
\begin{equation} \label{eq:twistor-knots}
 _{\mf{osp}}J_k^{\l}(q) = \alpha(\l,K) \; {}_{\mf{so}}J_k^{\l}(\mathbf{t}^{-1}q)
\end{equation}
for some scalar $\alpha(\l,K)$ depending on the dominant weight $\l$ and the link $K$.  In the case when $n=1$ this gives the surprising observation that the colored Jones polynomial can be obtained from the super representation theory of $\mf{osp}(1|2)$  with appropriate scalars.

Here we  show that the covering algebra $\U_{q,\pi}$ for $n=1$ specializes at $(q,\pi)=(\sqrt{-1},1)$ to the small quantum group for $\mf{sl}_2$ (at a fourth root of unity) and at parameters $(q,\pi)=(-1,-1)$ to a ``$q$-less subalgebra" of modified $\mf{sl}(1|1)$, see Sections~\ref{sec:small} and \ref{sec:qless}.
The quantum knot invariant twistor isomorphism ~\eqref{eq:twistor-knots} at $n=1$
specializes at $q=-1$ to a connection between the $\mf{osp}(1|2)$ invariant at parameter $q=-1$  and the $\mf{sl}_2$-invariant at $q=\mathbf{t}^{-1}(-1) =\mathbf{t}$ which is a fourth root of unity.
Hence, the connection between a $q$-less subalgebra of quantum $\mf{sl}(1|1)$ and $\mf{sl}_2$ at a fourth root of unity may be a special case of a twistor arising from the covering Kac-Moody theory.

%
\subsection{Categorification}
%

The existence of a canonical basis for the covering algebra $\U_{q,\pi}$ led Clark and Wang to conjecture the existence of a categorification of this algebra~\cite{ClarkWang}. The conjecture was proven in \cite{Lau-odd} who defined a $\Z\times \Z_2$-graded categorification $\mf{U}_{q,\pi}$ of $\U_{q,\pi}$.
Later, Brundan and Ellis gave a simplified treatment~\cite{BE2} using the theory of monoidal supercategories~\cite{BE1}.  This work provided a drastic simplification that makes the present work possible.

Thus far, the odd categorification $\oUqp$ of quantum $\mf{sl}_2$ has yet to be applied to give a higher representation theoretic interpretation of odd Khovanov homology.  However, it is interesting to note the strong agreement between the existence of covering Kac-Moody algebras for $\mf{so}(2n+1)$ and the existence of an ``odd link homology" for the same algebras predicted by the string theoretic approach to link homology constructed by Mikhaylov and Witten using D3-branes with boundary on fivebrane~\cite{Witten-odd}.

Given the expected connections to odd link homology, the conjectural  spectral sequences connecting odd Khovanov homology and knot Floer homology motivates the investigation of 2-categorical differentials on the odd categorified quantum group.
%
In particular, we categorify both specializations of the covering algebra at $(q,\pi)= (\sqrt{-1},1)$ and $(-1,-1)$ corresponding $\mf{sl}_2$ at a fourth root of unity and a subalgebra of quantum $\mf{sl}(1|1)$, see Corollary~\ref{cor:main}.   This is not as straightforward as one might hope.  In both algebras there are relations of the form $E^2 = F^2 = 0$ and such relations are known to be nontrivial to categorify.

If the identity morphism of a generator $E$ in a category is represented diagrammatically by a vertical arrow, then two vertical strands represents the object $EE$.   Khovanov was the first to identify the representation theoretic importance of dg-structures with a diagrammatic relation defining the differential of a crossing to be two vertical strands.  Such structures appeared in work of Lipshitz, Ozsvath, Thurston \cite{LOT} providing a combinatorial construction of Heegaard-Floer homology.  Khovanov showed that such a relation could be used to produce the nilpotent relation $E^2=0$ needed for a categorification of the positive part of $\mf{gl}(1|1)$~\cite{Kh-gl}.  This led to a categorification of the positive part of $\mf{gl}(m|1)$\cite{KhS-glm}.

Since Khovanov's initial observations, there have been various proposals to categorifications connected with $\mf{gl}(1|1)$ appearing in the literature.  In \cite{EPV} the  tangle Floer dg algebra is identified with a tensor product of ${\bf U}_q(\mf{gl}(1|1))$  representations and dg-bimodules were defined giving the action of quantum group generators $E$ and $F$. Further, Ozvath and Szabo's new  bordered Heegaard-Floer homology~\cite{OS-Kauff, OS-bknot} can be seen as a categorification of $\mf{gl}(1|1)$ representations via the work of Manion~\cite{Manion}.
Motivated by contact geometry, Tian defined a categofication of ${\bf U}_q(\mf{sl}(1|1))$ using triangulated categories arising from the contact category of the disc with points on the boundary~\cite{Yin1,Yin2,Yin3}.    An approach to categorifying tensor powers of the vector representation of ${\bf U}_q(\mf{gl}(1|1))$ based on super Schur-Weyl duality is given in \cite{Sart}, which is related to the bordered theory in ~\cite{Manion-KhS}.

Here we extend Khovanov's observation in order to categorify  the specializations of the covering algebra at $q^2=-\pi$.  To do this we define new dg-structures on the 2-category $\mf{U}_{q,\pi}$.

%
\subsection{Differential graded structures on categorified quantum group}
%

Derivations on the even categorification $\cal{U}(\mf{sl}_2)$  were studied by Elias and Qi~\cite{EQ2}. They were interested in
categorifying the small quantum group for $\mf{sl}_2$ at a (prime) root of unity.  Their approach made use of the theory of Hopfological algebra initiated by Khovanov~\cite{Kh-hopf} and developed by Qi~\cite{Qi}.  The main idea in Hopfological algebra is to equip a given categorification with the structure of a p-dg algebra. This is like a dg-algebra, except that $d^p=0$ rather than $d^2=0$.

Within the framework of Hopfological algebra, there have been a number of investigations into categorifications at a prime root of unity.  A p-dg analog of the nilHecke algebra was studied in \cite{KhQ}.
In \cite{EQ2} Elias and Qi categorify the small quantum group for $\mf{sl}_2$ at a (prime) root of unity by equipping the 2-category $\Ucat$ with a $p$-differential giving it the structure of a p-dg-2-category.
Using thick calculus from \cite{KLMS}, in Elias and Qi categorify an idempotented form of quantum $\mf{sl}_2$ and some of its simple representations at a prime root of unity~\cite{EQ1}.   This involves equipping the Karoubi envelope $\UcatD$ of the 2-category $\Ucat$ with a p-dg structure.
Related categorifications studied were studied in~\cite{QiSussan}.  All of these approaches require $p$ to be a prime root of unity and the base field to have characteristic $p$.

Much less in known about honest dg-structures, or categorification at a root of unity working  over an arbitrary field (see \cite{LQi} for the current state of the art).   In particular, it was shown in \cite{EQ2} that there are no nontrivial differentials in characteristic zero on the original categorification $\Ucat(\mf{sl}_2)$.
The only clue we have is the work of Ellis and Qi that equips the odd nilHecke algebra with an honest dg-algebra structure~\cite{EllisQi} .   Their work gives a categorification of the positive part of ${\bf U}_q(\mf{sl}_2)$ with $q$ specialized to a fourth root of unity.  There are a couple of points here worth highlighting.  First, they work with the odd nilHecke algebra defined over an arbitrary field or $\Z$ (no need to work in characteristic $p$).  Second, the fourth root of unity doesn't come from considering a funny version of chain complexes with $d^4=0$; they use ordinary dg-algebras.  However, the differential they define on the odd nilHecke algebra is not bidegree zero.  Rather it has $\Z \times \Z_2$ -degree $(2,\bar{1})$ leading to so called mixed complexes, or `half graded' chain complexes of vector spaces.

The effect of having mixed complexes is a collapse of the $\Z\times \Z_2$-bigrading, analogous to the $\delta$-grading from link homology theory. At
the level of the Grothendieck ring of the derived category of dg-modules, this has the effect of imposing the relation $1+q^2\pi=0$ in the ground ring $\Z[q,q^{-1},\pi]/(\pi^2-1)$.  When $\pi=1$, this gives the Grothendieck ring the structure of $\Z[\sqrt{-1}]$-algebra.   So the fourth root of unity comes from the bidegree of the differential, {\em not} from the theory of $p$-dg algebras.  This is discussed in greater detail in section~\ref{sec:gaussian}.

Ellis and Qi suggested that their work on the differential graded odd nilHecke algebra should extend to the odd categorified quantum group $\mf{U}(\mf{sl}_2)$ to provide a characteristic zero lift of the differentials defined on the original categorification $\cal{U}(\mf{sl}_2)$ that were studied in finite characeteristic in \cite{EQ2}.  Here we prove this conjecture by defining a family of differentials on the odd 2-supercategory $\mf{U}$, see Proposition~\ref{prop:differential_on_oUcat}.

%
\subsection{Main Results}
%

In Proposition~\ref{prop:differential_on_oUcat} we classify 2-categorical differentials on the odd 2-category $\mf{U}_{q,\pi}$.
Our classification depends on the so-called nondegeneracy conjecture stating that certain spanning sets form a basis for the 2-homs in $\oUqp$.  However, our  results are independent of this conjecture as we define explicit differentials giving the desired categorifications.
Following similar arguments from \cite{EQ2}, we show that the odd 2-category $\oUqp$ is dg-Morita equivalent to a positivly graded dg-algebra enabling us to compute the Grothendieck ring of the dg-2-supercategory $(\oUqp,\partial)$ using the theory of fantastic filtrations developed by Elias and Qi~\cite{EQ1}. As explained in section~\ref{sec:gaussian}, we have freedom in how we treat the $\Z_2$-grading in the Grothendieck group.  In particular, the Grothendieck group is naturally a $\Z[q,q^{-1},\pi]/(\pi^2-1, 1+q^2\pi)$ module with $[M \Pi] = \pi [M]$.  We show in Corollary~\ref{cor:main} that taking $\pi=1$ specialization results in a categorification of $\U(\mf{sl}_2)$ at a fourth root of unity.  While taking the $\pi=-1$ specialization eliminates $q$ entirely and we are left with a $\Z$-module closely related to $\mf{gl}(1|1)$.  In particular, we have relations $E^2=F^2=0$ and a super commutator relation for $E$ and $F$.  In this way, ${\bf U}_{\sqrt{-1}}(\mf{sl}_2)$   together with a $q$-less version of $\mf{gl}(1|1)$ appear naturally via different decategorifications of the same 2-category $\mf{U}_{q,\pi}$.

The significants of a bidgree $(2,\bar{1})$, or $\delta$-grading preserving, differential connecting the odd 2-category $\oUqp$ with quantum algebras connected to the Alexander polynomial is the 
further evidence it provides that odd categorified quantum groups may supply a higher representation theoretic bridge between  odd Khovanov homology and $\widehat{HF}(\Sigma(K))$.   An odd categorified quantum groups construction of odd Khovanov homology should have interesting interactions with the differential defined here, inducing a spectral sequence associated to these new differentials, and providing a higher representation theoretic categorification of the knot determinant $\det(K) =| \Delta_K(-1)|$.

\subsection*{Acknowledgements}
The authors are very grateful to You Qi for patiently explaining the details of his previous work
and to Andy Manion for explaining his perspective on quantum algebraic aspects of Heegaard-Floer homology.
We would also like to thank Jon Brundan and Joshua Sussan for comments on an earlier draft of this article.  Both authors were partially supported by the NSF grants DMS-1255334 and DMS-1664240.

%
\section{Super dg theory }
%

Here we consider $\Z\times \Z_2$-graded dg categories.   This is a modest generalization of the standard theory of dg-categories, since a $\Z$-graded dg-category induces a $\Z_2$-graded one by collapsing the grading modulo 2.  However, we note that the $\Z_2$ grading on 2-morphisms in the 2-category $\mf{U}$ defined in section~\ref{sec:oddU} are {\bf not} the mod 2 reductions of the quantum $\Z$-grading.  It is easy to see this from the bigrading on caps and cups.  We consider differentials with respect to the $\Z_2$ (or super) grading. If the differential also has a nontrivial $\Z$-grading (as is the case with the differential on $\mf{U}$) this can produce interesting effects  on the Grothendieck ring.  In particular, if the differential has bidegree $(2, \bar{1})$ we are led to the notion of `half graded' complexes whose Grothendieck ring corresponds to the Gaussian integers, see section~\ref{sec:gaussian}.

The natural context for discussing $\Z_2$-graded dg categories is the super category formalism developed by Brundan and Ellis~\cite{BE1,BE2} that we review in section~\ref{sec:supercategories}.

%
\subsection{2-supercategories} \label{sec:supercategories}
%
Let $\Bbbk$ be a field with characteristic not equal to 2.   A \emph{superspace} is a $\Z_2$-graded vector space
$
V = V_{ \bar{0}} \oplus V_{ \bar{1}}.
$
For a homogeneous element $v \in V$, write $|v|$ for the parity of $v$.

Let \cat{SVect} denote the category of superspaces and all linear maps. Note that homs $\Hom_{\cat{SVect}}(V,W)$ has the structure of a superspace since and linear map $f\maps V \to W$ between superspaces decomposes uniquely into an even and odd map.
The usual tensor product of $\Bbbk$-vector spaces is again a superspace with
$
(V\otimes W)_{ \bar{0}} = V_{ \bar{0}} \otimes W_{ \bar{0}}\oplus V_{ \bar{1}} \otimes W_{ \bar{1}}
$
and
$
(V\otimes W)_{ \bar{1}} = V_{ \bar{0}} \otimes W_{ \bar{1}}\oplus V_{ \bar{1}} \otimes W_{ \bar{0}}.
$
Likewise, the tensor product $f\otimes g$ of two linear maps between superspaces is defined by
\begin{equation}
  (f\otimes g) (v \otimes w) := (-1)^{|g||v|} f(v) \otimes g(w).
\end{equation}
Note that this tensor product does \emph{not} define a tensor product on \cat{SVect}, as the usual interchange law between tensor product and composition has a sign in the presence of odd maps
\begin{equation}
  (f\otimes g) \circ (h \otimes k) = (-1)^{|g||h|} (f \circ h) \otimes (g\circ k).
\end{equation}
This failure of the interchange law depending on parity is the primary structure differentiating super monoidal categories from their non-super analogs.

If we set $\underline{\cat{SVect}}$ to be the subcategory consisting of only even maps, then the tensor product equips $\underline{\cat{SVect}}$ with a monoidal structure.  The map $u \otimes v \mapsto (-1)^{|u||v|} v \otimes u$ makes  $\underline{\cat{SVect}}$ into a symmetric monoidal category.
We now define supercategories, superfunctors, and supernatural transformations by enriching categories over the symmetric monoidal category $\underline{\cat{SVect}}$.   See \cite{kel1} for a review of the enriched category theory.

\begin{definition}
A \emph{supercategory} $\cal{A}$ is a category enriched in $\underline{\cat{SVect}}$.  A superfunctor $F \maps \cal{A} \to \cal{B}$ between supercategories is an $\underline{\cat{SVect}}$-enriched functor.
\end{definition}

Unpacking this definition, the hom spaces in a supercategory are superspaces
\[
 \HOM_{\cal{A}}(X,Y) = \Hom_{\cal{A}}^{\bar{0}}(X,Y) \oplus \Hom_{\cal{A}}^{\bar{1}}(X,Y)
\]
and composition is given by an even linear map.
Let $\cat{SCat}$ denote the category of all (small) supercategories, with morphisms given by superfunctors.  This category admits a monoidal structure making it a symmetric monoidal category~\cite[Definition 1.2]{BE2}.

\begin{definition}
A 2-supercategory is a category enriched in \cat{SCat}.  These means that for each pair of objects we have a supercategory of morphisms, with composition given by a superfunctor.
\end{definition}

For our purpose, it suffices to consider a 2-supercategory to be an extension of the definition of a 2-category to a context where the interchange law relating horizontal and vertical composition is replaced by the \emph{super interchange law}
\[
\xy 0;/r.19pc/:
 (-6,10);(-6,-10); **\dir{-} ?(.75)*\dir{ };
  (6,10);(6,-10); **\dir{-} ?(.75)*\dir{ };
 (-6,4)*{\bigb{g}};
 (6,-4)*{\bigb{f}};
 (10,3)*{ \scs \l};
 (0,3)*{\scs \mu};
 (-10,3)*{\scs \nu};
 (-6,-13)*{Y};
 (6,-13)*{X};
  (-6,13)*{Y'};
 (6,13)*{X'};
 \endxy
 \;\; = \;\;
 \xy 0;/r.19pc/:
 (-6,10);(-6,-10); **\dir{-} ?(.75)*\dir{ };
  (6,10);(6,-10); **\dir{-} ?(.75)*\dir{ };
 (-6,0)*{\bigb{g}};
 (6,0)*{\bigb{f}};
 (12,3)*{ \scs \l};
 (0,3)*{\scs \mu};
 (-12,3)*{\scs \nu};
  (-6,-13)*{Y};
 (6,-13)*{X};
  (-6,13)*{Y'};
 (6,13)*{X'};
 \endxy
 \;\; = \;\;
 (-1)^{|f||g|} \;\;
 \xy 0;/r.19pc/:
 (-6,10);(-6,-10); **\dir{-} ?(.75)*\dir{ };
  (6,10);(6,-10); **\dir{-} ?(.75)*\dir{ };
 (-6,-4)*{\bigb{g}};
 (6,4)*{\bigb{f}};
 (12,3)*{ \scs \l};
 (0,3)*{\scs \mu};
 (-12,3)*{\scs \nu};
  (-6,-13)*{Y};
 (6,-13)*{X};
  (-6,13)*{Y'};
 (6,13)*{X'};
 \endxy
\]
Effectively this means that when exchanging heights of morphisms we must take into account their parity.

\subsection{DG-superalgebras}

In this section we collect some facts about differential graded algebras in the  super setting.  Following \cite{EllisQi} we grade our dg algebras by $\Z/ 2\Z$.  Traditional dg algebras inherit a $\Z_2$ grading by collapsing the $\Z$-grading mod 2.  However, in our setting we will have both a $\Z$-grading and $\Z_2$-grading that is not the mod 2 reduction of the $\Z$ grading.

A  \emph{dg-superalgebra} $(A,\partial_A)$ is a superalgebra $A=A_{\bar{0}} \oplus A_{\bar{1}}$ and an odd parity $\bar{1}$ $\Bbbk$-linear map $\partial=\partial_A\maps A \to A$ satisfying $\partial^2=0$ and for any homogeneous $a,b\in A$
\begin{equation} \label{eq:super-L}
 \partial(ab) = \partial_A(ab) = \partial_A(a)b +(-1)^{|a|}a\partial_A(b).
\end{equation}
A left \emph{dg-supermodule} $(M,\partial_M)$ is a supermodule $M=M_{\bar{0}} \oplus M_{\bar{1}}$ equipped with an odd parity $\Bbbk$-linear map $\partial_M\maps M \to M$ such that for any homogeneous elements $a\in A$, $m \in M$ we have
\[
 \partial_M(am) = \partial_A(a)m +(-1)^{|a|}a\partial_M(m).
\]
If $A$ and $B$ are dg-superalgebras, then a \emph{dg $(A,B)$-superbimodule} is a superspace equipped with a differential and commuting left dg $A$-supermodule and right dg $B$-supermodule structure.  Given a dg $(A,B)$-superbimodule $M$, define the \emph{parity shift} superbimodule $\Pi M$ to have the same underlying vector space as $M$ viewed as a superspace with the oppposite $\Z_2$-grading
\[
(\Pi M)_{\bar{0}} = M_{\bar{1}}, \qquad (\Pi M)_{\bar{1}} = M_{\bar{0}}.
\]
The superbimodule structure on $\Pi M$ is defined by $a \cdot m \cdot b := (-1)^{|a|}avb$.  With this definition the identity function on the underlying vector space defines an odd dg-superbimodule isomorphism $\zeta_M \maps \Pi M \to M$.  For a morphism of superbimodules $f \maps M \to N$ we define $\Pi f \maps \Pi M \to \Pi N$ by the map $(-1)^{|f|}f$.  Then there is an isomorphism $\xi_M \maps \Pi^2 M \to M$ given by minus the identity.  The category of $(A,B)$-superbimodules equipped with the superfunctor $\Pi$ and odd supernatural isomorphism $\zeta \maps \Pi \to \Id$ equips the category of dg $(A,B)$-superbimodules with the structure of a $\Pi$-supercategory in the sense of \cite[Definition 1.7]{BE1}.

\subsection{Graded dg-superalgebras}
A \emph{graded superspace} is super vector space $V$ equipped with an additional $\Z$-grading
\[
V = \bigoplus_{n\in \Z} V_n = \bigoplus_{n\in \Z} V_{n,\bar{0}} \oplus V_{n,\bar{1}}
\]
that is independent of the $\Z_2$ parity grading.
A \emph{graded} dg-superalgebra is a graded superalgebra with a $\Z\times \Z_2$ degree $(2,\bar{1})$ $\Bbbk$-linear differential  $\partial=\partial_A\maps A \to A$ satisfying \eqref{eq:super-L}.  Graded dg-superbimodules are defined analogously.  Given a graded dg $(A,B)$-superbimodule $M$, define the \emph{$\Z$-grading shift} dg-superbimodule $Q M$ by shifting the $\Z$-grading of the underlying vector space
$
(Q M)_{n} = M_{n-1}
$
and leaving the parity grading untouched.  For graded dg-superalgebras $A$ and $B$ the category of graded dg $(A,B)$ superbimodules is a graded $(Q,\Pi)$-supercategory in the sense of \cite[Definition 6.4]{BE1} with grading and parity shift functors $Q$ and $\Pi$.
\medskip

Given a (graded) superalgebra $A$, denote by $\cal{C}(A)$ the homotopy category of (graded) dg-supermodules given by quotienting maps of dg-supermodules by null-homotopies.  Likewise, we denote by $\cal{D}(A)$ the derived category of (graded) dg-supermodules.  Both $\cal{C}(A)$ and $\cal{D}(A)$ are triangulated categories. In the super setting that we are working in, the translation functor $[1]$ acts by the parity shift :
\[
(M[1])^k := \Pi M = M^{k + \bar{1}} , \qquad \partial_{M[1]} := \Pi \partial_M:= -\partial_M.
\]

\subsection{DG-supercategories}
For standard results on  dg-categories see~\cite{Keller}.
\begin{definition}
A supercategory $\cal{A}$ is called a {\em dg-supercategory} if the morphism spaces between any two objects $X,Y \in \cal{A}$ are equipped with a degree $\bar{1}$ differential $\partial$
\[
\partial \maps \Hom_{\cal{A}}^{\bar{x}}(X,Y) \longrightarrow \Hom_{\cal{A}}^{\bar{x} +\bar{1}}(X,Y),
\]
which acts via the Leibnitz rule
\begin{align}
 \partial(g\circ f) = \partial(g)\circ f + (-1)^{|g|}g \circ \partial(f) \nn
\end{align}
on composable pairs morphisms $f$ and $g$.
\end{definition}

Given a dg-superalgebra   $A$, consider the dg-enhanced supermodule category $A_{\partial}\dmod$ by defining the HOM-complex between two dg modules $M$ and $N$ to be
\[
\HOM_A(M,N) = \Hom^{\bar{0}}_A(M,N) \oplus \Hom^{\bar{1}}_A(M,N).
\]
The differential $\partial$ acts on a homogenous map $f\in \HOM_A(M,N)$ as
\[
 \partial(f) := \partial_N \circ f - (-1)^{|f|}f \circ \partial_M.
\]
If we take $A=\Bbbk$ with trivial differential differential then $\Bbbk_{\partial}\dmod$ is just the dg-category of chain complexes of super vector spaces.

The morphism space in the homotopy category $\cal{C}(A)$ is just the degree-zero part of the resulting cohomology
\[
 \Hom_{\cal{C}(A)}( M,N) = H^0(\HOM_A(M,N)).
\]
 The morphism spaces $\Hom_{\cal{D}(A)}(M,N)$ in the derived category are computed by replacing $M$ by a cofibrant replacement $P_M$.  Recall that a dg module $P$ over $A$ is called \emph{cofibrant} if for any surjective quasi-isomorphism $f \maps M \to N$ and any morphism $g \maps P \to N$, there exists a morphism $\bar{g}\maps P \to M$ making the diagram
\[
 \xy
 (0,15)*+{P}="P";
(0,0)*+{N}="N";
(-15,0)*+{M}="M";
    {\ar^g "P";"N"};
    {\ar_f "M";"N"};
    {\ar@{-->}^{\bar{g}} "P";"M"};
 \endxy
\]
commute. Every dg-supermodule $M$ admits a cofibrant replacement $P_M$, unique up to quasi isomorphism see \cite[Proposition 2.3]{EllisQi} and there are natural isomorphisms
\begin{equation} \label{eq:cofib}
  \Hom_{\cal{D}(A)}(M,N) \cong \Hom_{\cal{C}(A)}(P_M,N) = H^0(\HOM_A(P_M,N)).
\end{equation}

\begin{definition} \label{def:rep}
A left (respectively right) {\em dg-supermodule} $\cal{M}$ over a dg-supercategory $\cal{A}$ is a superfunctor
\begin{equation}
  \cal{M} \maps \cal{A} \to \Bbbk_{\partial}\dmod, \qquad
  ( \text{resp. } \cal{M} \maps \cal{A}^{{\rm op}} \to \Bbbk_{\partial}\dmod),
\end{equation}
that commutes with the $\partial$-actions on $\cal{A}$ and $\Bbbk_{\partial}\dmod$.  A dg-supermodule is called \emph{representable} if $\cal{M} = \HOM_{\cal{A}}(X,-)$ for some object $X$  of $\cal{A}$.
\end{definition}

\subsection{ DG 2-supercategories} \label{sec:super-dg-2-cat}

\begin{definition}
A (strict) dg 2-supercategory $(\mf{U},\partial)$ consists of a 2-category $\mf{U}$, together with a differential on 2-morphisms satisfying the super Leibnitz rule for both horizontal and vertical composition.
\end{definition}

More explicitly,  a dg-2-supercategory consists of the following data.

\begin{enumerate}
  \item A set of objects $I={\l, \mu , \dots}$, and for an $\l,\mu \in I$ we have
\[
 {}_{\mu}\mf{U}_{\l} := \Hom_{\mf{U}}(\l,\mu)
\]
is a dg-supercategory.  In particular,   vertical composition of 2-morphisms obeys the dg-supercategory Leibnitz rule for morphisms.

\item For any pair of 1-morphisms
${}_{\mu}E_{\l}$,
${}_{\mu}E'_{\l}$ in the same Hom space, the space of 2-morphisms
\[
 \HOM_{{}_{\mu}\mf{U}_{\l}} ({}_{\mu}E_{\l}, {}_{\mu}E'_{\l} )
\]
is a chain complex of superspaces.

\item The vertical composition of 2-morphisms satisfied the Leibnitz rule.  That is, for any pair of objects $\l, \mu, \in I$, then
\begin{align*}
\HOM_{{}_{\mu}\mf{U}_{\l}}({}_{\mu}E_{\l}, {}_{\mu}F_{\l})
\times
\HOM_{{}_{\mu}\mf{U}_{\l}}({}_{\mu}E_{\l}, {}_{\mu}F_{\l})
&\longrightarrow
\HOM_{{}_{\mu}\mf{U}_{\l}}({}_{\mu}E_{\l}, {}_{\mu}F_{\l}) \\
(g,f) & \mapsto (g\circ f)
\end{align*}
satisfies
\[
\partial(g \circ f) = \partial(g) \circ f+(-1)^{|g|} g \circ \partial(f).
\]

\item The horizontal composition of 2-morphisms satisfied the Leibnitz rule.  That is, for any triple of objects $\l, \mu, \nu \in I$, then
\begin{align*}
\HOM_{{}_{\nu}\mf{U}_{\mu}}({}_{\nu}F_{\mu}, {}_{\nu}F_{\mu})
\times
\HOM_{{}_{\mu}\mf{U}_{\l}}({}_{\mu}E_{\l}, {}_{\mu}E'_{\l})
&\longrightarrow
\HOM_{{}_{\nu}\mf{U}_{\l}}({}_{\mu}FE_{\l}, {}_{\mu}F'E'_{\l}) \\
(h,f) & \mapsto (hf)
\end{align*}
satisfies
\[
\partial(hf) = \partial(h)f+(-1)^{|h|}h\partial(f).
\]
\end{enumerate}

\subsection{Gradings}

The notion of a dg 2-supercategory can be extended to the notion of a \emph{graded} dg 2-supercategory in a straightforward way.
Denote by $\underline{\cat{GSVect}}$ the symmetric monoidal category of super vector spaces and degree preserving linear maps.  A \emph{graded supercategory} is then a category enriched over $\underline{\cat{GSVect}}$.  A \emph{graded 2-supercategory} is then a category enriched over the monoidal category of all small graded supercategories.  A \emph{graded super-dg-2-category} is a graded 2-supercategory equipped with a differential $\partial$ of degree 2 and parity $\bar{1}$ satisfying the super Leibnitz rule as in Section~\ref{sec:super-dg-2-cat}.

\subsection{$(Q,\Pi)$-envelopes of dg-2-supercategories} \label{sec:QPi-envelope}

A graded $(Q, \Pi)$-supercategory is a graded supercategory $\cal{A}$ together with superfunctors $Q,Q^{-1},\Pi\maps \cal{A} \to \cal{A}$, an odd supernatural isomorphism $\zeta \maps \Pi \To I$ that is homogeneous of degree 0, and even supernatural isomorphisms $\sigma \maps Q \To I$, $\bar{\sigma} \maps Q^{-1} \To I$ that are homogeneous of degrees 1 and -1, respectively.  This data makes $Q$ and $Q^{-1}$ mutually inverse graded superequivalences and $\Pi$ a self-inverse graded superequivalence.  A graded $(Q,\Pi)$-2-supercategory can be defined similarly, where each hom category has the structure of a graded $(Q,\Pi)$-supercategory, see \cite[Definition 6.5]{BE1} for a precise definition.

We have already seen an example of a graded $(Q,\Pi)$-supercategory coming from the category of graded supermodules over a dg-superalgebra.  Given a graded supercategory it always possible to enhance it to a graded $(Q,\Pi)$-supercategory by taking its $(Q,\Pi)$-envelope.

\begin{definition}[\cite{BE2} Definition 1.6]
Given a graded 2-supercategory $\mf{U}$, its $(Q,\Pi)$-envelope $\mf{U}_{q,\pi}$ is the graded 2-supercategory with the same objects as $\mf{U}$, 1-morphisms defined from
\[
\Hom_{\mf{U}_{q,\pi}}(\l,\,u) := \left\{
 Q^m\Pi^a F \mid \text{for all $F \in \Hom_{\mf{U}}(\l,\mu)$ with $m\in \Z$ and $a \in \Z/2\Z$}
\right\}
\]
with composition law $(Q^n\Pi^bG)(Q^m\Pi^aF):= Q^{m+n}\Pi^{a+b}(GF)$.  The 2-morphisms are defined by
\[
\Hom_{\mf{U}_{q,\pi}}(Q^m\Pi^aF, Q^n\Pi^bG) :=
\left\{
x_{m,a}^{n,b} \mid \text{for all $x \in \Hom_{\mf{U}}$(F,G)}
\right\}
\]
viewed as a superspace with addition given by $x_{m,a}^{n,b} + y_{m,a}^{n,b} := (x+y)^{n,b}_{m,a}$ and scalar multiplication given by $c(x_{m,a}^{n,b}) := (cx)_{m,a}^{n,b}$. The degrees are given by $\deg( x_{m,a}^{n,b})=\deg(x)+n-m$, $|x_{m,a}^{n,b}|=|x|+a+b$.
The horizontal composition  is given by
\begin{equation} \label{eq:hor-rule}
  y_{m,c}^{n,d}\cdot x_{k,a}^{l,b}
:= (-1)^{c|x|+b|y|+ac+bc} (y \cdot x)_{k+m,a+c}^{l+n,b+d},
\end{equation}
and the vertical composition by
\begin{equation} \label{eq:vert-rule}
 y_{m,b}^{n,b} \circ x_{\ell,a}^{m,b} = (y\circ x)_{\ell,a}^{n,c}.
\end{equation}
\end{definition}

\begin{lemma}
Let $(\mf{U},\partial)$ denote a graded dg-2-supercategory. Then the $(Q,\Pi)$-envelope $\mf{U}_{q,\pi}$ of $\mf{U}$ admits a super dg structure defined by
\begin{align*}
 \partial \maps \Hom_{\mf{U}_{q,\pi}}(Q^m\Pi^a F, Q^n \Pi^b G) &\longrightarrow \Hom_{\mf{U}_{q,\pi}}(Q^m\Pi^a F, Q^{n} \Pi^{b} G) \\
x_{m,a}^{n,b} &\;\; \mapsto \;\;   (-1)^{b}\left(\partial(x)\right)_{m,a}^{n,b}
\end{align*}
\end{lemma}

\begin{proof}
We verify that this definition satisfies the super Leibnitz rule with respect to the horizontal composition rule \eqref{eq:hor-rule} in the $(Q,\Pi)$-envelope
\begin{align*}
&  \partial \left( y_{m,c}^{n,d}\cdot x_{k,a}^{l,b} \right)
\\
& \quad
\; :=\;  (-1)^{c|x|+b|y|+ac+bc }
\partial\left((y\cdot x)_{k+m,a+c}^{l+n,b+d}\right) \\
& \quad
\; :=\;  (-1)^{c|x|+b|y|+ac+bc  +b+d  }
\left( \partial(y\cdot x)\right)_{k+m,a+c}^{l+n,b+d}
\\
& \quad
\; =\;  (-1)^{c|x|+b|y|+ac+bc+b+d }
\left(\partial(y)\cdot x \right)_{k+m,a+c}^{l+n,b+d}
+
(-1)^{c|x|+b|y|+ac+bc+b+d }(-1)^{|y|}
\left(y \cdot \partial(x) \right)_{k+m,a+c}^{l+n,b+d}\\
& \quad
\; =\;  (-1)^{c|x|+b|\partial(y)|+ac+bc +d }
\left(\partial(y)\cdot x \right)_{k+m,a+c}^{l+n,b+d}
+
(-1)^{c|\partial(x)|+b|y|+ac+bc+b}(-1)^{|y|+c+d}
\left(y \cdot \partial(x) \right)_{k+m,a+c}^{l+n,b+d}
\\
& \quad
\; =\;   \partial \left( y_{m,c}^{n,d}\right) \cdot x_{k,a}^{l,b}
+
 (-1)^{|y_{m,c}^{n,d}|}
y_{m,c}^{n,d} \cdot \partial \left(x_{k,a}^{l,b}\right)
\end{align*}
and the vertical composition rule \eqref{eq:vert-rule}
\begin{align*}
  &\partial\left( y_{m,b}^{n,c} \circ x_{\ell,a}^{m,b}\right)
  = \partial\left((y\circ x)_{\ell,a}^{n,c}\right)
  \\
 & = (-1)^c\left(\partial(y\circ x)\right)_{\ell,a}^{n,c}
 \\
 & = (-1)^c\left(\partial(y)\circ x \right)_{\ell,a}^{n,c}  +(-1)^{c+|y|} \left(y \circ \partial(y) \right)_{\ell,a}^{n,c}
 \\
 & =
(-1)^c (\partial(y))_{m,b}^{n,c} \circ x_{\ell,a}^{m,b}
+
(-1)^{|y|+c}y_{m,b}^{n,c} \circ   (\partial(x))_{\ell,a}^{m,b}
\\
&=
(-1)^c (\partial(y))_{m,b}^{n,c} \circ x_{\ell,a}^{m,b}
+
(-1)^{|y|+b+c}y_{m,b}^{n,c} \circ (-1)^b (\partial(x))_{\ell,a}^{m,b}
\\
&= \partial(y_{m,b}^{n,c}) \circ x_{\ell,a}^{m,b}
    +(-1)^{|y_{m,b}^{n,c}|} y_{m,b}^{n,c} \circ \partial(x_{\ell,a}^{m,b}).
\end{align*}
\end{proof}

\begin{definition} \label{def:Uund}
The \emph{underlying 2-category} $\underline{\mf{U}}_{q,\pi}$ of the $(Q,\Pi)$ envelope $\mf{U}_{q,\pi}$ of a 2-supercategory $\mf{U}$ is obtained by restricting $\mf{U}_{q,\pi}$ to degree zero 2-morphisms.
\end{definition}

The 2-category $\underline{\mf{U}}_{q,\pi}$  carries the structure of a $(Q,\Pi)$-2-category in the sense of \cite[Definition 6.14]{BE1}.
If $(\oUcat, \partial)$ is a graded dg  2-supercategory then the differential $\partial$ restricts to a degree zero map $\underline{\partial}$ on $\oUbar$ giving it that structure of a dg 2-category $(\oUbar,\underline{\partial})$.

\section{Hopfological algebra}

One of the primary reasons that triangulated categories are prevalent in categorification is the need to accommodate minus signs in the Grothendieck ring.  For positive algebraic structures, typically additive categories suffice with basis elements corresponding to indecomposable objects in the categorification.   Quantum groups with their canonical basis are an excellent example of this phenomenon.  However, as we expand categorification to include non-positive structures like the Jones polynomial, minus signs are lifted via the shift functor $[1]$ for some triangulated category, with the shift functor $[1]$ inducing the map of multiplication by $-1$ at the level of the Grothendieck group.

In his proposal for categorification at roots of unity, Khovanov showed that the traditional world of dg-categories, together with their homotopy and derived categories of modules, fits into a framework of Hopfological algebra.  For our purposes, Hopfological algebra will provide a valuable perspective on the possible decategorifications of graded dg-2-supercategories.  We quickly review the relevant details of Hopfological algebra needed for these purposes.   For a more detailed review see \cite{Kh-hopf, Qi}.

\subsection{Basic setup}
Let $H$  be a finite-dimensional Hopf algebra.  Then $H$ is also a Frobenius algebra and every injective $H$-module is automatically projective.  Define the stable category $H\stmod$ as the quotient of the category $H{\rm -mod}$  by the ideal of morphisms that factor through a projective (equivalently injective) module.  The category $H\stmod$ is triangulated
, see for example \cite{Ha}.

The shift functor for the triangulated structure on $H\stmod$ is defined by the cokernel of an inclusion of $M$ as a submodule into an injective (projective) module $I$.  We can fix this inclusion by noting that for any $H$-module $M$, the tensor product $H \otimes M$ with a free module is a free module, and the tensor product $P \otimes M$ with a projective module is always projective \cite[Proposition 2]{Kh-hopf}. A left integral $\Lambda$ for a Hopf algebra $H$ is an element $\Lambda \in H$ satisfying
\[
h \Lambda = \varepsilon(h) \Lambda.
\]
Using the left integral, any $H$-module $M$ admits a canonical embedding into an injective module via $M \mapsto H \otimes M$ sending $m \mapsto \Lambda \otimes m$.
This allows us to define a shift functor on the category of stable $H$-modules via
\begin{align}
  T \maps H\stmod & \to \;\;\;H\stmod \\
   M\;\;  &\mapsto (H/(H\Lambda)) \otimes M.  \nn
\end{align}

We now define the basic objects of interest in the theory of Hopfological algebra that generalize dg-algebras and their modules.  The reader may find Figure~\ref{fig:comparison} helpful for tracking the analogy.
An $H$-module algebra $B$ is an algebra equipped with an action of $H$ by algebra automorphisms.
A left $H$-comodule algebra is an associative $\Bbbk$-algebra $A$ equipped with a map
\[
 \Delta_A \maps A \to H \otimes A
\]
making $A$ an $H$-comodule and such that $\Delta_A$ is a map of algebras.

There is a natural construction to form a left $H$-comodule algebra from a right $H$-module algebra by forming the smash product algebra  $A := H\#B$.  As a $\Bbbk$-vector space $A$ is just $H \otimes B$, with multiplication given by
\[
(h \otimes b)(\ell \otimes c) = \sum h \ell_{(1)} \otimes (b \cdot \ell_{(2)})c,
\]
where we use Sweedler notation for the coproduct $\Delta(\ell) = \sum_{(\ell)} \ell_{(1)} \otimes \ell_{(2)} \in H \otimes H$.  The left $H$-comodule structure on $A = H\#B$ is given by $\Delta_A(h \otimes b) = \Delta(h)\otimes b$.
Let $A{\rm -mod}$ denote the category of left $A$-modules and define $A_H\stmod$ to be the quotient of $A{\rm -mod}$ by the ideal of morphisms that factor through an $A$-module of the form $H\otimes N$.  The category $A_H\stmod$ is triangulated~\cite[Theorem 1]{Kh-hopf} with shift functor inherited from $H\stmod$ defined by sending an object $M$ in $A_H\stmod$ to the module
\begin{equation}
T(M) := (H/(k\Lambda)) \otimes M.
\end{equation}

Since $H$ is a subalgebra of $A=H\#B$, we can restrict an $A$-module to an $H$-module, which descends to an exact functor $A_H\stmod$ to $H\stmod$. %
In the context of the $H$-comodule algebra $A=H\#B$ we write $\cal{C}(B,H) = A_H\stmod$.
Define a morphism $f \maps M \to N$ in $A_H\stmod$ to be a {\em quasi-isomorphism} if it restricts to an isomorphism in $H\stmod$.
Denote by $\cal{D}(B,H)$ the localization of $A$ with respect to quasi-isomorphisms.  It is shown in \cite[Corollary 2]{Kh-hopf} and \cite[Corollary 7.15]{Qi} that $\cal{D}(B,H)$ is a triangulated category whose Grothendieck group is a module over $K(H\stmod)$.


\subsection{DG-algebras from the Hopfological perspective}

The standard theory of dg-algebras and their modules is equivalent to the Hopfological algebra of the $\Z$-graded Hopf superalgebra $H=\Bbbk[D]/D^2$ in the category of super vector spaces.  Here $\deg(D)=\bar{1}$ and
\begin{alignat}{2}
  \Delta(1) &= 1 \otimes 1, \qquad    &&\varepsilon(1)=1 \\
  \Delta(D) &= 1 \otimes D + D \otimes 1 , \qquad  &&\varepsilon(D)=0.
\end{alignat}
For the Hopf superalgebra $\Bbbk[D]/D^2$ the left integral is spanned by $\Lambda = D$.

For a graded $\Bbbk$-superalgebra $B$ to admit an $H$-module structure this is equivalent to $B$ having a degree $\bar{1}$ map $\partial \maps B \to B$ satisfying
\[
 \partial(ab) = \partial(a)b + (-1)^{|a|}a\partial(b), \qquad \partial^2(a)=0,
\]
for all $a,b \in B$.   Hence, an $H$-module algebra is the same thing as a dg-algebra.   In a similar way, if we set $A := B\#H$ then an $A$ module is the same thing as a $B$-dg-module.  Further, one can show that $\cal{C}(B,H) =A_H\stmod$ is equivalent to the homotopy category $\cal{C}(B)$ of $B$-dg modules and that $\cal{D}(B,H)$ is equivalent to the derived category $\cal{D}(B)$ of $B$-dg-modules.

\begin{figure} [h]\label{fig:comparison}
\begin{center}
 \begin{tabular}{|c|c|}
   \hline
   {\bf DG-algebras} & {\bf Hopfological algebra}
   \\  \hline
    DG-algebra $B$ & $H$-module algebra
    \\  \hline
   DG-module & $A:=B\#H$-module
    \\ \hline
   $B{\rm -dgmod}$ & $A{\rm -mod}$
    \\ \hline
    Homotopy category of $B{\rm -dgmod}$ & $A_{H}\stmod$
    \\ \hline
    Derived category of $B$-dg-modules & $\cal{D}(B,H)$
    \\ \hline
 \end{tabular}
\end{center}
\end{figure}

\subsection{Decategorification from the Hopfological perspective}

To have an interesting notion of Grothendieck group for the triangulated categories $A_H\stmod$ it is important that we restrict the classes of modules under consideration to avoid pathologies that can arise.  In the context of Hopfological algebra the correct notion is that of \emph{compact hopfological modules} from \cite[Section 7.2]{Qi}.  Denote by $\cal{D}^{c}(A,H)$ the strictly full subcategory of compact hopfological modules in $\cal{D}(A,H)$.
\begin{definition}[\cite{Qi}]
Let $B$ be an $H$-module algebra over a finite dimensional Hopf algebra $H$ over a base field $\Bbbk$.  Define the Grothendieck group $K_0(\cal{D}^{c}(B,H))$ to be the abelian group generated by symbols of isomorphism classes of objects in $\cal{D}^{c}(B,H)$, modulo the relation
\[
 [Y] = [X] + [Z],
\]
whenever there is a distinguished triangle inside $\cal{D}^{c}(B,H)$ of the form
\[
X \longrightarrow Y \longrightarrow Z \longrightarrow T(X).
\]
\end{definition}

Both the Grothendieck rings of categories $\cal{C}(B,H)$ and $\cal{D}(B,H)$ are left modules over the Grothendieck ring $K_0(H\stmod)$ (see \cite[Corollary 1 and 2]{Kh-hopf}).  Hence, the ground ring for decategorification provided by the theory of Hopfological algebra associated to the Hopf algebra $H$ is determined by  $K_0(H\stmod)$. Note this group has a ring structure because $H\stmod$ has an exact tensor product.  When $H$ is quasi-triangular then $K(H\stmod)$ is commutative, so that we do not need to distinguish between left and right modules~\cite[Remark 7.17]{Qi}.

\subsubsection{Ground ring for Grothendieck group from the Hopfological perspective}

In the special case when $A=\Bbbk$, the Grothendieck group for $\cal{D}(\Bbbk,H)$ is the same as $H\stmod$ since $H$ acts trivially on $\Bbbk$ \cite[Corollary 9.11]{Qi}.
Since $K_0(A_H\stmod)$ is a module over $K_0(H\stmod) \cong K_0(\cal{D}(\Bbbk,H))$, the Grothendieck ring of $\cal{D}(\Bbbk,H)$ determines the ground ring for the Grothendieck group of $A_H\stmod$.   In the language of dg-algebras, this just says that $K_0$ of the derived category of chain complexes of vector spaces determines the ground ring for $K_0$ of the category of dg-modules.

Consider the category of complexes of $\Bbbk$-vector spaces.  Considering the homological degree modulo two gives rise to a $\Z_2$ grading for the dg homotopy category of (ungraded) chain complexes $\cal{D}(\Bbbk)$ of vector spaces where the differential has degree $\deg(d) = \bar{1}$.
Assuming $\Bbbk=\Z$ or a field, it follows that any complex in $\cal{D}(\Bbbk)$ is isomorphic to a direct sum of indecomposable chain complexes of the following form:
\begin{itemize}
   \item a single copy of $\Bbbk$ in any bidegree;
   \item a copy of $S = \left( 0 \to  \Bbbk  \stackrel{d}{\longrightarrow}  \Bbbk \Pi \to 0 \right)$ where we include the parity shift of $\Pi$ on the right hand side to accommodate the degree of the differential.
\end{itemize}
Then the Grothendieck group is generated as a $\Z[\pi]/(\pi^2-1)$-module by the symbol $[\Bbbk]$ with  $[\Bbbk\Pi]=\pi[\Bbbk]$.    If the differential $d$ in the complex $S$ is given by multiplication by a unit in $\Bbbk$, then $S$ is contractible and therefore isomorphic to $0$ in $K_0(\cal{D}(\Bbbk))$. The contractibility of $S$ imposes the additional relation
\begin{equation}
  (1+\pi)[\Bbbk] = 0.
\end{equation}
The classication of objects in $\cal{D}(\Bbbk)$ implies that this is the only relation, and it forces the symbol of $S$ to be zero even when $d$ is not multiplication by an invertible element.   Hence, $\pi=-1$ and
\begin{equation}
  K_0(\cal{D}(\Bbbk)) \cong \Z[ \pi]/(1+\pi)  = \Z.
\end{equation}

The homological shift $\Bbbk[1]$ is given by the cokernel of the inclusion into $H\otimes \Bbbk$ with $k \mapsto \Lambda \Bbbk = D \otimes k$.  The injective envelope $H \otimes \Bbbk$ is two dimensional as a vector space spanned by the identity and $D$. We can represent $H \otimes \Bbbk$ by the complex
\[
 \xy
  (-10,0)*+{\Bbbk \Pi}="tl";
  (10,0)*+{\Bbbk}="tr";
  {\ar^{D} "tl"; "tr"};
 \endxy
\]
where $\Bbbk$ includes into the right most term via the map $D \otimes 1$.  Hence, the cokernel of this inclusion gives that $\Bbbk[1] = \Bbbk \Pi$.  So we have recovered from the hopfological perspective the fact that the shift $[1]$ is just the parity shift $\Pi$ and at the level of the Grothendieck group we have
\[
\big[ (\Bbbk[1]) \big] = \big[\Bbbk \Pi \big] = \pi\big[\Bbbk\big] = - \big[\Bbbk\big].
\]
We carefully reviewed the usual dg-case to set the stage for our treatment in the `mixed complex' setting.

\subsection{Gaussian integers}  \label{sec:gaussian}
The following section is an extension of the discussion in \cite[Section 2.2.4]{EllisQi} that was explained to us by You Qi.
Consider the category of $\Z \times \Z_2$-graded modules.  We denote by $\la 1\ra$ a shift of the quantum (or $\Z$-grading), and by $\Pi$ the parity shift functor.  Define a differential between such modules to be a map of bidegree $(2,\bar{1})$ that squares to 0.
The main difference between this case and the previous is that our Hopf algebra input into Hopfological algebra is now the Hopf superaglebra $H=\Bbbk[D]/D^2$
where $D$ has mixed degree $(2,\bar{1})$.
A chain complex is a $\Bbbk$-module equipped with such a differential.   Following \cite{EllisQi} we call such complexes {\em half-graded complexes} for reasons that will become clear.
Denote the corresponding homotopy category by $\cal{C}(\Bbbk)$ and the derived category by $\cal{D}(\Bbbk)$.

 Any category of $\Z \times \Z_2$ graded dg-modules with differentials of bidegree $(2,\bar{1})$ will have a Grothendieck ring that is a module over $K_0(\cal{D}(\Bbbk))$, so this Grothendieck ring controls the ground ring that appears in categorification via half-graded complexes.
Assuming $\Bbbk=\Z$ or a field, it follows that any complex in $\cal{D}(\Bbbk)$ is isomorphic to a direct sum of indecomposable chain complexes of the following form:
\begin{itemize}
   \item a single copy of $\Bbbk$ in any bidegree;
   \item a copy of $S = \left( 0 \to \Pi^b\Bbbk\la a\ra \stackrel{d}{\longrightarrow} \Pi^{b+1}\Bbbk\la a+2\ra \to 0 \right)$ with the first term in any bidegree $(a,b)$ and the right most copy in bidegree $(a+2,b+\bar{1})$.
\end{itemize}
Then the Grothendieck group is generated as a $\Z[q,q^{-1},\pi]/(\pi^2-1)$-module by the symbol $[\Bbbk]$ with $[\Bbbk\la 1 \ra] = q[\Bbbk]$ and $[\Bbbk\Pi]=\pi[\Bbbk]$.    If the differential $d$ in the complex $S$ is given by multiplication by a unit in $\Bbbk$, then $S$ is contractible and therefore isomorphic to $0$ in $K_0(\cal{D}(\Bbbk))$.  For simplicity take $a=b=0$ in $S$, the contractibility of $S$ imposes the additional relation
\begin{equation}
  (1+q^2\pi)[\Bbbk] = 0 .
\end{equation}
The classication of objects in $\cal{D}(\Bbbk)$ implies that this is the only relation, and it forces the symbol of $S$ to be zero even when $d$ is not multiplication by an invertible element.   Hence,
\begin{equation}
  K_0(\cal{D}(\Bbbk)) \cong \Z[q,q^{-1},\pi]/(\pi^2-1, 1+q^2\pi).
\end{equation}

The homological shift is now given by the inclusion of $\Bbbk$ into  $H \otimes \Bbbk$ via $D \otimes 1$
\[
 \xy
  (-10,0)*+{\Bbbk \la -2\ra\Pi}="tl";
  (10,0)*+{\Bbbk}="tr";
  {\ar^-{D} "tl"; "tr"};
 \endxy
\]
so that $\Bbbk[1]:= \Bbbk \la -2\ra \Pi$ and at the level of the Grothendieck group we have
\[
 \big[ \Bbbk[1] \big] = \big[\Bbbk \la -2 \ra \Pi\big] = \big[\Bbbk\big]q^{-2}\pi = - \big[\Bbbk\big]
\]
since $1+q^2\pi =0$.  Hence, the homological shift is multiplication by $-1$ on $K_0$.

If we specialize $\pi=-1$, then the equation imposed by the contractible complex implies that $q^2=1$, so the ground ring for reduces to $\Z$.  If we specialize $\pi =1$ then we have the relation $q^2=-1$ and we get that $q$ must be a fourth root of unity.  Hence, we have the following result.

\begin{proposition} \label{prop:K0-alg}
Given a $\Z \times \Z_2$ graded algebra equipped with a differential $d$ of bidegree $(2,\bar{1})$.  Then the Grothendieck group associated with the category of $\Z \times \Z_2$-graded dg-modules is a module over the ring
\[
\Z[q,q^{-1},\pi]/(\pi^2-1, 1+q^2\pi).
\]
At $\pi=-1$ this is just $\Z$ and at $\pi=1$ this is $\Z[\sqrt{-1}]$.
\end{proposition}

\section{Results on Grothendieck groups of dg-superalgebras}

\subsection{Grothendieck group of a dg-superalgebra }

Despite our protracted discussion of Hopfological algebra, the decategorification of categories of dg-supermodules is not so unlike the decategorification of normal dg-modules.   We detoured through Hopfological algebra to highlight the fact that the Grothendieck ring will have the structure of a module over the Gaussian integers $\Z[\sqrt{-1}]$.  Just as in the usual theory of dg-modules over a dg-algebra $A$, to have a sensible notion of Grothendieck group of $\cal{D}(A)$, we pass to the compact or perfect derived category $\cal{D}^c(A)$.
The category $\cal{D}^c(A)$ is a subcategory of $\cal{D}(A)$ consisting {\em compact } dg modules, that is, those dg-supermodules $M$ such that the functor $\HOM_{\cal{D}(A)}(M,-)$ commutes with infinite direct sums. This is the same as considering $\cal{D}^{c}(A,H)$ in the Hopfological setup with $H$ defined in section~\ref{sec:gaussian}.

For our purposes the connection between compact dg modules and finite-cell modules will be of particular relevance.  See for example \cite[Example 2.4]{EQ2}. A \emph{finite-cell} module over a dg-superalgebra $A$ is a dg-supermodule with a finite filtration whose subquotients are isomorphic as dg-supermodules to dg-summands of $A$.  Such finite-cell modules are always compact.

The {\em Grothendieck group} $K_0(A)$ of a (graded) dg-superalgebra $A$ is the quotient of the free abelian group on the isomorphism classes $[M]$ of compact (graded) dg-supermodules $M$ by the relation $[M]=[M_1]+[M_2]$ whenever
\[
M_1 \to M \to M_2 \to M_1[1]
\]
is an exact triangle of compact objects in $\cal{D}^c(A)$.  This is the same as $\cal{D}^c(A,H)$ for $H$ defined in section~\ref{sec:gaussian}.  The \emph{Grothendieck group $K_0(\cal{A})$} of a graded dg-supercategory $\cal{A}$ can be defined similarly, by regarding $\cal{A}$ as a graded dg-superalgebra
\[
 \mathbb{A}:= \bigoplus_{x,y \in {\rm Ob}\cal{A}} \Hom_{\cal{A}}(x,y)
\]
with orthogonal idempotents $1_x$ for each object $x \in \cal{A}$.

\subsection{Positively graded   dg-algebras}

A $\Z$-graded dg-superalgebra is called a {\em positive dg-superalgebra} (see \cite{Sch}) if it satisfies the following
\begin{enumerate}
  \item the superalgebra $A = \oplus_{i \in \Z}A^i$ is non-negatively graded,
  \item the degree zero part $A^0$ is semisimple, and
  \item the differential acts trivially on $A^0$.
\end{enumerate}
We say that $A$ is a {\em strongly positive dg-superalgebra} is it is a positive dg-superalgebra with degree zero part $A^0 \cong \Bbbk$.

The calculation of the Grothendieck ring of a positively graded dg-algebra is greatly simplified.

\begin{theorem} \label{thm:Ko-positive}
  Let $A$ be a positive dg-superalgebra, and $A^0$ be its homogeneous degree zero part. Then
  \[
  K_0(A) \cong K_0(A^0).
  \]
\end{theorem}

\begin{proof}
 This is a direct extension of the non-super result from \cite{Sch} and \cite[Corollary 2.6]{EQ2}.
\end{proof}

%
\subsection{Fantastic filtrations} \label{sec:fantastic}
%
In this section, we give a review of the fantastic filtration and recall the related theorems from \cite{EQ2}. Fantastic filtration is an essential tool in this work for determining the Grothendieck ring of the odd dg 2-category $\mf{U}$ defined in the next section.  The key issue is that if $A$ is a dg-superalgebra the direct sum decomposition of $A$-modules does not necessarily commute with the differential. However, if there exists a fantastic filtration $F^{\bullet}$ on an $A$-module $Ae$, where $e$ is an idempotent, then the direct sum decomposition of $Ae$ as $A$-modules becomes a filtered direct sum decomposition of dg-modules.

We collect several important results on fantastic filtrations from \cite[Section 5]{EQ2} that are easily adapted to the super dg-setting.
\begin{lemma} \label{lemma:orthogonal_idemp_cond.}
Let $R$ be a superring and the elements $u_i, v_i \in R$ with $|u_i|=|v_i|$, where $i \in I$ is a finite set, satisfy the following conditions:
\begin{equation} \label{eq:uv-equations}
\begin{split}
&u_iv_iu_i = u_i \\
&v_iu_iv_i = v_i \\
&v_iu_j = \delta_{i,j}
\end{split}
\end{equation}
then $e = \sum_{i} u_iv_i$ is an idempotent and we have a direct sum decomposition $Re \cong \oplus_{i}Rv_iu_i$.
\end{lemma}
Note that $u_iv_i$ is an idempotent for each $i \in I$, as $u_iv_iu_iv_i = u_iv_i$, and moreover $\{u_iv_i\}_{i \in I}$ is a set of orthogonal idempotents, as for any $i \neq j$,
$u_iv_iu_jv_j = u_jv_ju_iv_i = 0$. It follows that $e$ is an idempotent and $Re \cong \oplus_{i}Rv_iu_i$.

For a dg-algebra $A$ and any idempotent $e \in A$, the $A$-module $Ae$ is an $A_{\partial}\dmod$ summand if for any $a\in A$, we have $\partial(ae) \in Ae$ for any $be \in Ae$. By the Leibniz rule,
\begin{align*}
\partial(abe) &=
\partial(a)be + (-1)^{|a|}a\partial(b)e + (-1)^{|a|+|b|}ab\partial(e)\\
&=\partial(ab)e + (-1)^{|a|+|b|}ab\partial(e)
\end{align*}
so that $\partial(abe) \in Ae$  if $\partial(e)=0$.  The computation of the differential of an idempotent $e$ is important for determining if $Ae$ is compact in the derived category $\mathcal{D}(A)$, since $\partial(e)=0$ implies that $Ae$ is cofibrant and has a compact image in $\mathcal{D}(A)$.

The following is a straight-forward adaptation of Lemma 5.3 in \cite{EQ2}.
\begin{proposition} \label{prop: Filtration is dg}
Let $(A,\partial)$ be a dg-superalgebra, $i \in I$ a finite index set, $u_i,v_i \in A$  satisfying the hypothesis of Lemma \ref{lemma:orthogonal_idemp_cond.}.  Suppose that $e = \sum_{i}u_iv_i$, and $<$ is a total order on $I$. An $I$-indexed $A$-supermodule filtration $F^{\bullet}$ of $Ae$ is defined by
\[
F^{\leq i} := \sum_{ j\leq i} Ru_jv_j
\]
and $F^{\emptyset} := 0$, so that $F^{\leq i} / F^{<i} \cong Av_iu_i$ as $A$ modules. Then the following conditions are equivalent:
\begin{enumerate}
\item $F^{\bullet}$ is a filtration by dg-supermodules, so that $Av_iu_i$ is a dg-supermodule and the subquotient isomorphism is an isomorphism of dg-supermodules.
\item The following equations are satisfied for all $i \in I$,
\begin{align}
& v_i \partial(u_i) = 0\\
& u_i \partial(v_i) \in F^{< i}.
\end{align}
\end{enumerate}
\end{proposition}

\begin{definition}
If the filtration $F^{\bullet}$ in Proposition \ref{prop: Filtration is dg} satisfies $\partial(e) = 0$ and $\partial(v_iu_i) =0$ for all $i\in I$, then it is called a \emph{fantastic filtration} on the dg-module $Ae$.
\end{definition}

The main advantage of the fantastic filtration is that it gives a filtered direct sum decomposition of the images of idempotents as dg-modules.  By a straightforward extension of \cite[Corollary 5.8]{EQ2} the following theorem holds.
\begin{theorem}\label{thm:Fantastic_filt_eqn}
Let $A$ be a dg-superalgebra, $\{u_i,v_i\}_{i \in I}$ a finite set of elements of $A$ satifying Proposition \ref{prop: Filtration is dg}, then there is a fantastic filtration on the dg module $Ae$ if and only if there exists a total order on $I$ such that
\[
v_i \partial(u_j) = 0  \quad \text{for $j \geq i$.}
\] Moreover, in $K_0(A)$, we have the relation
\[
[Ae] = \sum_{i\in I} [Av_iu_i]. \]
\end{theorem}

\subsection{Grothendieck ring of  dg 2-supercategories }  \label{section:Groth-dg-2cat}

Recall from section~\ref{sec:QPi-envelope} that if $\mf{U}$ is a graded 2-supercategory then we denote by $\oUqp$ is $(Q,\Pi)$-envelope and by $\oUbar$ its underlying 2-category obtained by restricting to degree zero maps.  The Grothendieck group of $\mf{U}$ is defined as
\[
K_0(\mf{U}) := K_0(\oUdot)
\]
where  $\oUdot$ denotes the Karoubi envelope of $\oUbar$.
The  $(Q,\Pi)$-2-category structure (see \cite[Section 6]{BE1} and section~\ref{sec:QPi-envelope}) on $\oUbar$ makes $K_0(\mf{U})$ a  $\Z[q,q^{-1},\pi]/(\pi^2-1)$-module with $[Q^mX ]= q^m[X]$ and $[X\Pi^a]=\pi^a[X]$ for $X\in \Hom_{\oUbar}(\l,\mu)$.

For our discussion in the dg-setting it is helpful to recall that the hom category $\Hom_{\oUdot}(\l,\mu)$ is equivalent to the abelian category of finitely generated graded projective left ${}_{\mu}\mathbb{A}_{\l}$-supermodules and morphisms that preserve degree and parity, where ${}_{\mu}\mathbb{A}_{\l}$ is the graded  superalgebra defined as the direct sum
\[
{}_{\mu}\mathbb{A}_{\l} := \left(\bigoplus_{x,y }\Hom_{\mf{U}}(\onenn{\mu}x \onel, \onenn{\mu}y\onel) \right)
\]
over 1-morphisms $x,y \maps \l \to \mu$ in $\mf{U}$.
We now consider the dg-setting.

\begin{definition}
Given a graded dg 2-supercategory $(\mf{U}, \partial)$ let
\begin{equation}
  \cal{D}(\mf{U}) := \bigoplus_{\l,\mu \in {\rm Ob}(\mf{U})} \cal{D}({}_{\mu}\mf{U}_{\l})
=  \bigoplus_{\l,\mu \in {\rm Ob}(\mf{U})} \cal{D}({}_{\mu}\mathbb{A}_{\l}),
\end{equation}
where ${}_{\mu}\mathbb{A}_{\l}$ is the graded dg-superalgebra obtained by summing over all objects of the dg-supercategory ${}_{\mu}\mf{U}_{\l}$.   Define the \emph{split Grothendieck group $K_0(\mf{U},\partial)$ of  $(\mf{U}, \partial)$} by
\begin{equation}
  K_0(\mf{U},\partial) = K_0(\oUdot,\underline{\partial})
:= \bigoplus_{\l,\mu \in {\rm Ob}(\mf{U})} K_0\left(\cal{D}^c\left(_{\mu}\mf{U}_{\l}\right)\right)
= \bigoplus_{\l,\mu \in {\rm Ob}(\mf{U})} K_0\left(\cal{D}^c\left(_{\mu}\mathbb{A}_{\l}\right)\right).
\end{equation}
\end{definition}

\begin{corollary} \label{cor:Ko-super}
Let $(\mf{U},\partial)$ be a graded dg 2-supercategory where the differential $\partial$ has bidgree  $(2,\bar{1})$.  Then the Grothendieck group of $K_0(\mf{U},\partial)$ is a module over the ring
\[
\mathcal{R}:= \Z[q,q^{-1},\pi]/(\pi^2-1, 1+q^2\pi)
\]
 with $[Q^mX ]= q^m[X]$ and $[X\Pi^a]=\pi^a[X]$ for $X\in \Hom_{\mf{U}}(\l,\mu)$ for some objects $\l,\mu \in \mf{U}$.
At $\pi=-1$ this is just $\Z$ and at $\pi=1$ this is $\Z[\sqrt{-1}]$.
\end{corollary}

\begin{proof}
This is an immediate corollary of Proposition~\ref{prop:K0-alg} applied to graded dg-superalgebras $_{\mu}\mathbb{A}_{\l}$ for objects $\l,\mu \in \mf{U}$.
\end{proof}

%
\section{The odd 2-category for $sl(2)$} \label{sec:oddU}
%

 \subsection{The odd nilHecke ring} \label{sec:oddnil}
The \text{odd nilHecke algebra} $\ONH_n$ is the graded unital associative superalgebra generated by elements $x_1,\ldots,x_n$ of degree 2 and parity $\bar{1}$ and elements $\partial_1,\ldots,\partial_{n-1}$ of degree $-2$ and parity $\bar{1}$, subject to the relations
\begin{eqnarray}
& & \partial_i^2 = 0 , \quad \partial_i \partial_{i+1}\partial_i =
 \partial_{i+1}\partial_i \partial_{i+1},\\
 & &  x_i \partial_i + \partial_i  x_{i+1} =1, \quad
   \partial_i x_i + x_{i+1}\partial_i = 1,\\
& & x_i x_j + x_j x_i =0 \quad (i\neq j),  \quad
\partial_i \partial_j + \partial_j \partial_i =0 \quad (|i-j|>1),  \\
& & x_i \partial_j +\partial_j x_i = 0 \quad (i\neq j,j+1).
\end{eqnarray}

For $w\in S_n$ and a choice  of a reduced expression $w=s_{i_1}\cdots s_{i_\ell}$ in terms of simple transpositions $s_i=(i\quad i+1)$, define
\begin{equation}
\partial_w=\partial_{i_1}\cdots\partial_{i_\ell}.
\end{equation}
Note that $\partial_w$ only depends on the reduced expression up to an overall sign.  For $w_0$ the longest word in $S_a$ we fix a preferred choice of reduced expression.
\[
\partial_{w_0} = \partial_1(\partial_2\partial_1) \dots (\partial_{n-1} \dots \partial_1).
\]
The element
\begin{equation} \label{eq:idempotent}
  {\sf e}_n := (-1)^{\binom{n}{3}}\partial_{w_0}x_1^{n-1} x_2^{n-1} \dots x_n^0
\end{equation}
is an idempotent of $\ONH_n$, see \cite[Lemma 2.17]{EllisQi} or \cite{EKL}.

The superalgebra of skew polnomials $\OPol_n$ is defined by
\begin{equation}
  \OPol_n := \Z \la x_1 ,\dots, x_n \ra/ (x_i x_j + x_j x_i =0 \; \text{if $i\neq j$.})
\end{equation}
The superalgebra $\ONH_n$ acts on $\OPol_n$ with $x_i$ acting by multiplication and $\partial_i$, the $i$-th odd divided difference operator
$\partial_i \maps \OPol_n \to \OPol_n$ defined by
\begin{align}
  \partial_i(x_j) &=
  \left\{
    \begin{array}{ll}
      1 & \hbox{$j=i,i+1$} \\
      0 & \hbox{$j \neq i,i+1$,}
    \end{array}
  \right.
\\
 \partial_i (fg) &= \partial_i(f) g + (-1)^{|f|} s_i(f) \partial_i(g).
\end{align}
Under this action, $\OPol_n \cong \ONH_n {\sf e}_n$ is the unique (up to isomorphism and grading shift) indecomposable projective $\ONH_n$-supermodule.  In \cite{EKL} it was shown that there is a superalgebra isomorphism
\begin{equation}
  \ONH_n \cong {\rm END}_{{\rm O}\Lambda_n}(\OPol_n),
\end{equation}
where ${\rm O}\Lambda_n$ is the superalgebra of {\em odd symmetric polynomials}.

For a superalgebra $A$ we denote by $A^{{\rm sop}}$ the superalgebra with multiplication defined by
\[
 x^{{\rm sop}} y^{{\rm sop}} := (-1)^{|x||y|} (y x)^{{\rm sop}}.
\]
There is a super algebra antiinvolution $\omega \maps \ONH_n \to \ONH_n^{{\rm sop}}$ defined by sending
\[
\omega(x_i) = x_i^{{\rm sop}}, \qquad \omega(\partial_i) = - \partial_i^{{\rm sop}}.
\]
In $\ONH_n^{{\rm sop}}$ the relations are the same except for the relation
\begin{equation}
 - x_i^{{\rm sop}} \partial_i^{{\rm sop}} - \partial_i^{{\rm sop}}  x_{i+1}^{{\rm sop}} =1, \quad
  - \partial_i^{{\rm sop}} x_i^{{\rm sop}} - x_{i+1}^{{\rm sop}}\partial_i^{{\rm sop}} = 1.
\end{equation}

%
\subsection{The odd categorified quantum group}
%

In \cite{BE2} Ellis and Brundan give a minimal presentation of the 2-category $\oUcat$ that requires the invertibility of certain maps.  Here we give a more traditional presentation by including the additional relations on 2-morphisms that are equivalent to the invertibility of these maps.

\begin{definition}
The odd $2$-supercategory $\oUcat=\oUcat(\mf{sl}_2)$ is the $2$-supercategory
consisting of
\begin{itemize}
\item objects $\l$ for $\l \in \Z$,
\item
for a signed sequence $\underline{\epsilon}=(\epsilon_1,\epsilon_2, \dots, \epsilon_m)$, with $\epsilon_1,\dots,\epsilon_m \in \{+,-\}$, define
\[\cal{E}_{\underline{\epsilon}}:= \cal{E}_{\epsilon_1} \cal{E}_{\epsilon_2} \dots \cal{E}_{\epsilon_m}
\]
where $\cal{E}_{+}:=\cal{E}$ and $\cal{E}_{-}:= \cal{F}$.  A
$1$-morphisms from $\l$ to $\l'$ is a formal finite direct sum of strings
\[
 \cal{E}_{\underline{\epsilon}}\onel   =
  \onenn{\l'}\cal{E}_{\underline{\epsilon}}
\]
for any   signed sequence $\underline{\epsilon}$ such that $\l'= \l+ 2\sum_{j=1}^m\epsilon_j1$.

\item $2$-morphisms are generated by
\begin{align}
  \xy 0;/r.17pc/:
 (0,7);(0,-7); **\dir{-} ?(.75)*\dir{>};
 (0,0)*{\bullet};
 (7,3)*{ \scs \l};
 (-9,3)*{\scs \l+2};
 (-10,0)*{};(10,0)*{};
 \endxy &\maps \cal{E}\onel  \to \cal{E}\onel   & \quad
 &
    \xy 0;/r.17pc/:
  (0,0)*{\xybox{
    (-4,-4)*{};(4,4)*{} **\crv{(-4,-1) & (4,1)}?(1)*\dir{>} ;
    (4,-4)*{};(-4,4)*{} **\crv{(4,-1) & (-4,1)}?(1)*\dir{>};
     (9,1)*{\scs \l};
     (-10,0)*{};(10,0)*{};
     }};
  \endxy \;\;\maps \cal{E}\cal{E}\onel   \to \cal{E}\cal{E}\onel   \nn \\
  & \text{degree } (2, \bar{1})  & & \qquad \quad   \text{degree } (-2,\bar{1} )  \nn
   \\
   & & & \nn \\
     \xy 0;/r.17pc/:
    (0,-3)*{\bbpef{}};
    (8,-5)*{\scs \l};
    (-10,0)*{};(10,0)*{};
    \endxy &\maps \onel   \to \cal{F}\cal{E}\onel  &
    &
   \xy 0;/r.17pc/:
    (0,-3)*{\bbpfe{}};
    (8,-5)*{\scs \l};
    (-10,0)*{};(10,0)*{};
    \endxy \maps \onel  \to\cal{E}\cal{F}\onel   \nn \\
     &  \text{degree } (1+\l,\bar{0} )   & & \qquad \quad \text{degree } (1-\l, \overline{\l+1}) ) \nn
   \\
      & & & \nn \\
  \xy 0;/r.17pc/:
    (0,0)*{\bbcef{}};
    (8,4)*{\scs  \l};
    (-10,0)*{};(10,0)*{};
    \endxy & \maps \cal{F}\cal{E}\onel  \to\onel  &
    &
 \xy 0;/r.17pc/:
    (0,0)*{\bbcfe{}};
    (8,4)*{\scs  \l};
    (-10,0)*{};(10,0)*{};
    \endxy \maps\cal{E}\cal{F}\onel    \to\onel  \nn \\
      & \text{degree } ( (1+\l, \overline{1+\l} )    & & \qquad \quad  \text{degree } (1-\l,  \bar{0})  \nn
\end{align}
where we have indicated a $Q$-grading and parity as an ordered tuple $(x,\bar{y})$.
 Note  that the $\Z_2$ degree of the right pointing cap and cup are not the mod 2 reductions of the $\Z$-degree.
\end{itemize}
The identity $2$-morphism of the $1$-morphism $\cal{E} \onen$ is represented by an upward oriented line (likewise, the identity $2$-morphism of $\cal{F} \onen$ is represented by a downward oriented line).

Composites of the above diagrams are interpreted using the conventions for supercategories from Section~\ref{sec:supercategories}.   The rightmost region in our diagrams is usually colored by $\l$.  The fact that we are defining a 2-supercategory means that diagrams with odd parity skew commute.
The $2$-morphisms satisfy the following relations
 (see \cite{BE2} for more details).
\begin{enumerate}

\item {\bf (Odd nilHecke)}
The $\cal{E}$'s carry an action of the odd nilHecke algebra with $x_i$ corresponding to a dot and $\partial_i$ corresponding to a crossing. Using the adjoint structure this induces an action of the odd nilHecke algebra on the $\cal{F}$'s via the antiinvolution $\omega$.

\begin{equation} \label{eq_nilHecke1}
  \vcenter{\xy 0;/r.18pc/:
    (-4,-4)*{};(4,4)*{} **\crv{(-4,-1) & (4,1)}?(1)*\dir{};
    (4,-4)*{};(-4,4)*{} **\crv{(4,-1) & (-4,1)}?(1)*\dir{};
    (-4,4)*{};(4,12)*{} **\crv{(-4,7) & (4,9)}?(1)*\dir{>};
    (4,4)*{};(-4,12)*{} **\crv{(4,7) & (-4,9)}?(1)*\dir{>};
 \endxy}
 \;\;= \;\; 0, \qquad \quad
 \vcenter{
 \xy 0;/r.18pc/:
    (-4,-4)*{};(4,4)*{} **\crv{(-4,-1) & (4,1)}?(1)*\dir{>};
    (4,-4)*{};(-4,4)*{} **\crv{(4,-1) & (-4,1)}?(1)*\dir{>};
    (4,4)*{};(12,12)*{} **\crv{(4,7) & (12,9)}?(1)*\dir{>};
    (12,4)*{};(4,12)*{} **\crv{(12,7) & (4,9)}?(1)*\dir{>};
    (-4,12)*{};(4,20)*{} **\crv{(-4,15) & (4,17)}?(1)*\dir{>};
    (4,12)*{};(-4,20)*{} **\crv{(4,15) & (-4,17)}?(1)*\dir{>};
    (-4,4)*{}; (-4,12) **\dir{-};
    (12,-4)*{}; (12,4) **\dir{-};
    (12,12)*{}; (12,20) **\dir{-};
  (18,8)*{\l};
\endxy}
 \;\; =\;\;
 \vcenter{
 \xy 0;/r.18pc/:
    (4,-4)*{};(-4,4)*{} **\crv{(4,-1) & (-4,1)}?(1)*\dir{>};
    (-4,-4)*{};(4,4)*{} **\crv{(-4,-1) & (4,1)}?(1)*\dir{>};
    (-4,4)*{};(-12,12)*{} **\crv{(-4,7) & (-12,9)}?(1)*\dir{>};
    (-12,4)*{};(-4,12)*{} **\crv{(-12,7) & (-4,9)}?(1)*\dir{>};
    (4,12)*{};(-4,20)*{} **\crv{(4,15) & (-4,17)}?(1)*\dir{>};
    (-4,12)*{};(4,20)*{} **\crv{(-4,15) & (4,17)}?(1)*\dir{>};
    (4,4)*{}; (4,12) **\dir{-};
    (-12,-4)*{}; (-12,4) **\dir{-};
    (-12,12)*{}; (-12,20) **\dir{-};
  (10,8)*{\l};
\endxy}
\end{equation}

\begin{equation} \label{eq_nilHecke2}
  \xy 0;/r.18pc/:
  (3,4);(3,-4) **\dir{-}?(0)*\dir{<}+(2.3,0)*{};
  (-3,4);(-3,-4) **\dir{-}?(0)*\dir{<}+(2.3,0)*{};
  (8,2)*{\l};
 \endxy
 \quad =
\xy 0;/r.18pc/:
  (0,0)*{\xybox{
    (-4,-4)*{};(4,4)*{} **\crv{(-4,-1) & (4,0)}?(1)*\dir{>}?(.25)*{\bullet};
    (4,-4)*{};(-4,4)*{} **\crv{(4,-1) & (-4,0)}?(1)*\dir{>};
     (8,1)*{\l};
     (-10,0)*{};(10,0)*{};
     }};
  \endxy
 \;\; +
 \xy 0;/r.18pc/:
  (0,0)*{\xybox{
    (-4,-4)*{};(4,4)*{} **\crv{(-4,-1) & (4,0)}?(1)*\dir{>}?(.75)*{\bullet};
    (4,-4)*{};(-4,4)*{} **\crv{(4,-1) & (-4,0)}?(1)*\dir{>};
     (8,1)*{\l};
     (-10,0)*{};(10,0)*{};
     }};
  \endxy
 \;\; =
\xy 0;/r.18pc/:
  (0,0)*{\xybox{
    (-4,-4)*{};(4,4)*{} **\crv{(-4,-1) & (4,0)}?(1)*\dir{>};
    (4,-4)*{};(-4,4)*{} **\crv{(4,-1) & (-4,0)}?(1)*\dir{>}?(.75)*{\bullet};
     (8,1)*{\l};
     (-10,0)*{};(10,0)*{};
     }};
  \endxy
 \;\; +
  \xy 0;/r.18pc/:
  (0,0)*{\xybox{
    (-4,-4)*{};(4,4)*{} **\crv{(-4,-1) & (4,0)}?(1)*\dir{>} ;
    (4,-4)*{};(-4,4)*{} **\crv{(4,-1) & (-4,0)}?(1)*\dir{>}?(.25)*{\bullet};
     (8,1)*{\l};
     (-10,0)*{};(10,0)*{};
     }};
  \endxy
\end{equation}

\item {\bf (Right adjunction axioms)}
\begin{equation} \label{eq_biadj1}
    \xy 0;/r.17pc/:
    (-8,5)*{}="1";
    (0,5)*{}="2";
    (0,-5)*{}="2'";
    (8,-5)*{}="3";
    (-8,-10);"1" **\dir{-}?(.5)*\dir{>};
    "2";"2'" **\dir{-};
    "1";"2" **\crv{(-8,12) & (0,12)} ;
    "2'";"3" **\crv{(0,-12) & (8,-12)};
    "3"; (8,10) **\dir{-}?(1)*\dir{>};
    (15,-9)*{ \l-2};
    (-12,9)*{\l};
    (10,8)*{\scs };
    (-10,-8)*{\scs };
    \endxy
    \quad = \quad
      \xy 0;/r.17pc/:
 (0,10);(0,-10); **\dir{-} ?(1)*\dir{>}+(2.3,0)*{\scriptstyle{}}
 ?(.1)*\dir{ }+(2,0)*{\scs };
 (-6,5)*{\l};
 (8,5)*{ \l-2};
 (-10,0)*{};(10,0)*{};(-2,-8)*{\scs };
 \endxy
    \qquad \qquad
          \xy 0;/r.17pc/:
 (0,10);(0,-10); **\dir{-} ?(0)*\dir{<}+(2.3,0)*{\scriptstyle{}}
 ?(.1)*\dir{ }+(2,0)*{\scs };
 (-6,5)*{ \l};
 (8,5)*{ \l+2};
 (-10,0)*{};(10,0)*{};(-2,-8)*{\scs };
 \endxy
    \quad = \quad
    \xy 0;/r.17pc/:
    (8,5)*{}="1";
    (0,5)*{}="2";
    (0,-5)*{}="2'";
    (-8,-5)*{}="3";
    (8,-10);"1" **\dir{-} ?(0)*\dir{<};
    "2";"2'" **\dir{-} ?(.5)*\dir{<};
    "1";"2" **\crv{(8,12) & (0,12)} ;
    "2'";"3" **\crv{(0,-12) & (-8,-12)};
    "3"; (-8,10) **\dir{-} ?(.5)*\dir{<};
    (15,9)*{\l+2};
    (-12,-9)*{\l};
    (-10,8)*{\scs };
    (10,-8)*{\scs };
    \endxy
\end{equation}

\item {\bf (Parity left adjoint)}
\begin{equation} \label{eq_biadj2}
   \xy 0;/r.17pc/:
 (0,10);(0,-10); **\dir{-} ?(1)*\dir{>}+(2.3,0)*{\scriptstyle{}}
 ?(.1)*\dir{ }+(2,0)*{\scs };
 (-6,5)*{\l};
 (8,5)*{ \l-2};
 (-10,0)*{};(10,0)*{};(-2,-8)*{\scs };
 \endxy
    \quad = \quad (-1)^{\l+1}
    \xy 0;/r.17pc/:
    (8,5)*{}="1";
    (0,5)*{}="2";
    (0,-5)*{}="2'";
    (-8,-5)*{}="3";
    (8,-10);"1" **\dir{-}?(.6)*\dir{>};
    "2";"2'" **\dir{-};
    "1";"2" **\crv{(8,12) & (0,12)} ;
    "2'";"3" **\crv{(0,-12) & (-8,-12)};
    "3"; (-8,10) **\dir{-}?(1)*\dir{>};
    (15,9)*{\l-2};
    (-12,-9)*{\l};
    (-10,8)*{\scs };
    (10,-8)*{\scs };
    \endxy
    \qquad
    \xy 0;/r.17pc/:
    (-8,5)*{}="1";
    (0,5)*{}="2";
    (0,-5)*{}="2'";
    (8,-5)*{}="3";
    (-8,-10);"1" **\dir{-}?(0)*\dir{<};
    "2";"2'" **\dir{-}?(.5)*\dir{<};
    "1";"2" **\crv{(-8,12) & (0,12)} ;
    "2'";"3" **\crv{(0,-12) & (8,-12)};
    "3"; (8,10) **\dir{-}?(.5)*\dir{<};
    (15,-9)*{ \l+2};
    (-12,9)*{\l};
    (10,8)*{\scs };
    (-10,-8)*{\scs };
    \endxy
    \quad = \quad
      \xy 0;/r.17pc/:
 (0,10);(0,-10); **\dir{-} ?(0)*\dir{<}+(2.3,0)*{\scriptstyle{}}
 ?(.1)*\dir{ }+(2,0)*{\scs };
 (-6,5)*{ \l};
 (8,5)*{ \l+2};
 (-10,0)*{};(10,0)*{};(-2,-8)*{\scs };
 \endxy
\end{equation}

\item {\bf (Bubble relations)}

\noindent Dotted bubbles of negative degree are zero, so that
for all $m \geq 0$ one has
\begin{equation} \label{eq_positivity_bubbles}
\xy
 (-12,0)*{\lbbub{m}{}};
 (-8,6)*{\l};
 \endxy
  = 0
 \qquad  \text{if $m< \l-1$,} \qquad
 \xy
 (-12,0)*{\rbbub{m}{}};
 (-8,6)*{\l};
 \endxy = 0\quad
  \text{if $m< -\l -1$.}
\end{equation}
Dotted bubbles of degree $0$ are equal to the identity $2$-morphism:
\begin{equation} \label{eq:degzero-bub}
\xy
 (0,0)*{\lbbub{\l-1}};
  (4,6)*{\l};
 \endxy
  =  \Id_{\onen} \quad \text{for $\l \geq 1$,}
  \qquad \quad
  \xy
 (0,0)*{\lbbub{-\l-1}};
  (4,6)*{\l};
 \endxy  =  \Id_{\onen} \quad \text{for $\l \leq -1$.}
 \end{equation}
We use the following notation for the
 dotted  bubbles:
\begin{equation*}
 \xy
(0,0)*{\lbub{m}};
 (5,4)*{\l};
 (6,0)*{};
\endxy := \xy
 (5,4)*{\l};
(0,0)*{\lbbub{ m+\l-1}};
\endxy,
\quad\quad
 \xy
(0,0)*{\rbub{m}};
 (6,4)*{\l};
 (6,0)*{};
\endxy := \xy
 (5,4)*{\l};
(0,0)*{\rbbub{m-\l-1}};
\endxy,
\end{equation*}
so that
$$\deg\left( \xy
(0,0)*{\lbub{m}};
 (5,4)*{\l};
 (6,0)*{};
\endxy\right)= \deg\left(\xy
(0,0)*{\rbub{m}};
 (5,4)*{\l};
 (6,0)*{};
\endxy\right)=2m.$$
The degree 2 bubbles are given a special notation as follows:
\begin{align} \label{eq:oddbub}
\xy
 (5,4)*{\l};
(0,0)*{\bigotimes};
\endxy
\;\; := \;\;
\left\{
  \begin{array}{ll}
     \xy
(0,0)*{\lbub{1}};
 (5,4)*{\l};
 (6,0)*{};
\endxy \;\; =\;\; \xy
 (5,4)*{\l};
(0,0)*{\lbbub{\l}};
\endxy, & \hbox{$\l \geq 0$,} \\
     \xy
(0,0)*{\rbub{1}};
 (5,4)*{\l};
 (6,0)*{};
 (5,-2)*{\scriptstyle 1};
\endxy  \;\; = \;\; \xy
 (5,4)*{\l};
(0,0)*{\rbbub{-\l}};
\endxy, & \hbox{$\l \leq 0$.}
  \end{array}
\right.
\end{align}
By the superinterchange law this bubble squares to zero
\begin{align} \label{eq:oddbub-square}
\left( \;\xy
 (5,4)*{\l};
(0,0)*{\bigotimes};
\endxy \; \right)^2 = 0
\end{align}

We call a clockwise (resp. counterclockwise) bubble  fake
if $m+n-1<0$ and (resp. if $m-n-1<0$). The fake bubbles are defined recursively by
the homogeneous terms of the equation
\begin{equation}\label{rel:odd_inf_grassm}
  \sum_{\overset{r,s \geq 0}{r+s=t}} \;
   \xy
(-3,5)*{\lbub{2r}};
(2,-5)*{\rbub{2s}};
 (5,4)*{\l};
 (6,0)*{};
\endxy \;\; = \;\; \delta_{t,0}.
\end{equation}
\begin{equation}\label{rel:decomp_of_bubble_2n+1_into_2n_and_1}
  \xy
(0,0)*{\lbub{2n+1}};
 (5,4)*{\l};
 (6,0)*{};
\endxy \;\; =\;\; \vcenter{\xy
 (5,4)*{\l};
(-3,5)*{\lbub{2n}};
(0,-2)*{\bigotimes};
\endxy},
\qquad \qquad
  \xy
(0,0)*{\rbub{2n+1}};
 (5,4)*{\l};
 (6,0)*{};
\endxy \;\; =\;\; \vcenter{\xy
 (5,4)*{\l};
(-3,5)*{\rbub{2n}};
(-5,-2)*{\bigotimes};
\endxy}
\end{equation}

\item {\bf (Centrality of odd bubbles) }
By the super interchange law it follows that the odd bubble squares to zero.  Further, we have
\begin{equation} \label{rel: centrality of odd bubbles}
\xy 0;/r.17pc/:
 (0,8);(0,-8); **\dir{-} ?(1)*\dir{>}+(2.3,0)*{\scriptstyle{}}
 ?(.1)*\dir{ }+(2,0)*{\scs };
 (-5,-2)*{\txt\large{$\bigotimes$}};
 (-8,6)*{ \l};
 (-10,0)*{};(10,0)*{};(-2,-8)*{\scs };
 \endxy
\;\; = \;\;
\xy 0;/r.17pc/:
 (0,8);(0,-8); **\dir{-} ?(1)*\dir{>}+(2.3,0)*{\scriptstyle{}}
 ?(.1)*\dir{ }+(2,0)*{\scs };
 (5,-2)*{\txt\large{$\bigotimes$}};
 (-8,6)*{ \l};
 (-10,0)*{};(10,0)*{};(-2,-8)*{\scs };
 \endxy
 \qquad \qquad
\xy 0;/r.17pc/:
 (0,8);(0,-8); **\dir{-} ?(0)*\dir{<}+(2.3,0)*{\scriptstyle{}}
 ?(.1)*\dir{ }+(2,0)*{\scs };
 (-5,-2)*{\txt\large{$\bigotimes$}};
 (-8,6)*{ \l};
 (-10,0)*{};(10,0)*{};(-2,-8)*{\scs };
 \endxy
\;\; = \;\;
\xy 0;/r.17pc/:
 (0,8);(0,-8); **\dir{-} ?(0)*\dir{<}+(2.3,0)*{\scriptstyle{}}
 ?(.1)*\dir{ }+(2,0)*{\scs };
 (5,-2)*{\txt\large{$\bigotimes$}};
 (-8,6)*{ \l};
 (-10,0)*{};(10,0)*{};(-2,-8)*{\scs };
 \endxy
\end{equation}

\item \label{item_cycbiadjoint} {\bf (Cyclicity propeties)}

\begin{equation} \label{eq_cyclic_dot}
\xy 0;/r.17pc/:
 (0,10);(0,-10); **\dir{-} ?(0)*\dir{<}+(2.3,0)*{\scriptstyle{}};
 (0,0)*{\bullet};
 (-9,-3)*{\l};
 (-10,0)*{};(10,0)*{};
 \endxy
 \; := \;
    \xy 0;/r.17pc/:
    (8,5)*{}="1";
    (0,5)*{}="2";
    (0,-5)*{}="2'";
    (-8,-5)*{}="3";
    (8,-10);"1" **\dir{-}?(1)*\dir{<};
    "2";"2'" **\dir{-} ?(.5)*\dir{<};
    "1";"2" **\crv{(8,12) & (0,12)} ?(0)*\dir{<};
    "2'";"3" **\crv{(0,-12) & (-8,-12)};
    "3"; (-8,10) **\dir{-};
    (-12,-9)*{\l};
    (0,4)*{\txt\large{$\bullet$}};
    (-10,8)*{\scs };
    (10,-8)*{\scs };
    \endxy
     \quad = \quad 2 \;
           \xy 0;/r.17pc/:
 (0,10);(0,-10); **\dir{-} ?(.75)*\dir{<}+(2.3,0)*{\scriptstyle{}}
 ?(.1)*\dir{ }+(2,0)*{\scs };
 (-5,0)*{\txt\large{$\bigotimes$}};
 (-8,8)*{ \l};
 (-10,0)*{};(10,0)*{};(-2,-8)*{\scs };
 \endxy
 \;\; - \;\;
         \xy 0;/r.17pc/:
    (-8,5)*{}="1";
    (0,5)*{}="2";
    (0,-5)*{}="2'";
    (8,-5)*{}="3";
    (-8,-10);"1" **\dir{-};
    "2";"2'" **\dir{-} ?(.5)*\dir{<};
    "1";"2" **\crv{(-8,12) & (0,12)} ?(0)*\dir{<};
    "2'";"3" **\crv{(0,-12) & (8,-12)}?(1)*\dir{<};
    "3"; (8,10) **\dir{-};
    (-12,9)*{\l};
    (0,4)*{\txt\large{$\bullet$}};
    (10,8)*{\scs };
    (-10,-8)*{\scs };
    \endxy
\end{equation}

The cyclic relations for crossings are given by
\begin{equation} \label{eq_cyclic}
   \xy 0;/r.17pc/:
  (0,0)*{\xybox{
    (-4,4)*{};(4,-4)*{} **\crv{(-4,1) & (4,-1)}?(1)*\dir{>} ;
    (4,4)*{};(-4,-4)*{} **\crv{(4,1) & (-4,-1)}?(1)*\dir{>};
     (9,1)*{\scs  \lambda};
     (-10,0)*{};(10,0)*{};
     }};
  \endxy \quad := \quad
  \xy 0;/r.17pc/:
  (0,0)*{\xybox{
    (4,-4)*{};(-4,4)*{} **\crv{(4,-1) & (-4,1)}?(1)*\dir{>};
    (-4,-4)*{};(4,4)*{} **\crv{(-4,-1) & (4,1)};
     (-4,4)*{};(18,4)*{} **\crv{(-4,16) & (18,16)} ?(1)*\dir{>};
     (4,-4)*{};(-18,-4)*{} **\crv{(4,-16) & (-18,-16)} ?(1)*\dir{<}?(0)*\dir{<};
     (-18,-4);(-18,12) **\dir{-};(-12,-4);(-12,12) **\dir{-};
     (18,4);(18,-12) **\dir{-};(12,4);(12,-12) **\dir{-};
     (22,1)*{ \lambda};
     (-10,0)*{};(10,0)*{};
     (-4,-4)*{};(-12,-4)*{} **\crv{(-4,-10) & (-12,-10)}?(1)*\dir{<}?(0)*\dir{<};
      (4,4)*{};(12,4)*{} **\crv{(4,10) & (12,10)}?(1)*\dir{>}?(0)*\dir{>};
     }};
  \endxy
\quad =  \quad
\xy 0;/r.17pc/:
  (0,0)*{\xybox{
    (-4,-4)*{};(4,4)*{} **\crv{(-4,-1) & (4,1)}?(1)*\dir{>};
    (4,-4)*{};(-4,4)*{} **\crv{(4,-1) & (-4,1)};
     (4,4)*{};(-18,4)*{} **\crv{(4,16) & (-18,16)} ?(1)*\dir{>};
     (-4,-4)*{};(18,-4)*{} **\crv{(-4,-16) & (18,-16)} ?(1)*\dir{<}?(0)*\dir{<};
     (18,-4);(18,12) **\dir{-};(12,-4);(12,12) **\dir{-};
     (-18,4);(-18,-12) **\dir{-};(-12,4);(-12,-12) **\dir{-};
     (22,1)*{ \lambda};
     (-10,0)*{};(10,0)*{};
      (4,-4)*{};(12,-4)*{} **\crv{(4,-10) & (12,-10)}?(1)*\dir{<}?(0)*\dir{<};
      (-4,4)*{};(-12,4)*{} **\crv{(-4,10) & (-12,10)}?(1)*\dir{>}?(0)*\dir{>};
     }};
  \endxy
\end{equation}

Sideways crossings satisfy the following identities:
\begin{align} \label{eq_crossl-gen-cyc}
  \xy 0;/r.18pc/:
  (0,0)*{\xybox{
    (-4,-4)*{};(4,4)*{} **\crv{(-4,-1) & (4,1)}?(1)*\dir{>} ;
    (4,-4)*{};(-4,4)*{} **\crv{(4,-1) & (-4,1)}?(0)*\dir{<};
     (9,2)*{ \lambda};
     (-12,0)*{};(12,0)*{};
     }};
  \endxy
\quad & := \quad
 \xy 0;/r.17pc/:
  (0,0)*{\xybox{
    (4,-4)*{};(-4,4)*{} **\crv{(4,-1) & (-4,1)}?(1)*\dir{>};
    (-4,-4)*{};(4,4)*{} **\crv{(-4,-1) & (4,1)};
     (-4,4);(-4,12) **\dir{-};
     (-12,-4);(-12,12) **\dir{-};
     (4,-4);(4,-12) **\dir{-};(12,4);(12,-12) **\dir{-};
     (16,1)*{\lambda};
     (-10,0)*{};(10,0)*{};
     (-4,-4)*{};(-12,-4)*{} **\crv{(-4,-10) & (-12,-10)}?(1)*\dir{<}?(0)*\dir{<};
      (4,4)*{};(12,4)*{} **\crv{(4,10) & (12,10)}?(1)*\dir{>}?(0)*\dir{>};%
     }};
  \endxy
  \quad = \quad
 \xy 0;/r.17pc/:
  (0,0)*{\xybox{
    (-4,-4)*{};(4,4)*{} **\crv{(-4,-1) & (4,1)}?(1)*\dir{<};
    (4,-4)*{};(-4,4)*{} **\crv{(4,-1) & (-4,1)};
     (4,4);(4,12) **\dir{-};
     (12,-4);(12,12) **\dir{-};
     (-4,-4);(-4,-12) **\dir{-};(-12,4);(-12,-12) **\dir{-};
     (16,1)*{\lambda};
     (10,0)*{};(-10,0)*{};
     (4,-4)*{};(12,-4)*{} **\crv{(4,-10) & (12,-10)}?(1)*\dir{>}?(0)*\dir{>};
      (-4,4)*{};(-12,4)*{} **\crv{(-4,10) & (-12,10)}?(1)*\dir{<}?(0)*\dir{<};
     }};
  \endxy
\\
\label{eq_crossr-gen-cyc}
  \xy 0;/r.18pc/:
  (0,0)*{\xybox{
    (-4,-4)*{};(4,4)*{} **\crv{(-4,-1) & (4,1)}?(0)*\dir{<} ;
    (4,-4)*{};(-4,4)*{} **\crv{(4,-1) & (-4,1)}?(1)*\dir{>};%
     (9,2)*{ \lambda};
     (-12,0)*{};(12,0)*{};
     }};
  \endxy
\quad &:= (-1)^{\l+1}\quad
 \xy 0;/r.17pc/:
  (0,0)*{\xybox{
    (-4,-4)*{};(4,4)*{} **\crv{(-4,-1) & (4,1)}?(1)*\dir{>};
    (4,-4)*{};(-4,4)*{} **\crv{(4,-1) & (-4,1)};
     (4,4);(4,12) **\dir{-};
     (12,-4);(12,12) **\dir{-};
     (-4,-4);(-4,-12) **\dir{-};(-12,4);(-12,-12) **\dir{-};
     (16,-6)*{\lambda};
     (10,0)*{};(-10,0)*{};
     (4,-4)*{};(12,-4)*{} **\crv{(4,-10) & (12,-10)}?(1)*\dir{<}?(0)*\dir{<};
      (-4,4)*{};(-12,4)*{} **\crv{(-4,10) & (-12,10)}?(1)*\dir{>}?(0)*\dir{>};
     }};
  \endxy
  \quad = -\quad
  \xy 0;/r.17pc/:
  (0,0)*{\xybox{
    (4,-4)*{};(-4,4)*{} **\crv{(4,-1) & (-4,1)}?(1)*\dir{<};
    (-4,-4)*{};(4,4)*{} **\crv{(-4,-1) & (4,1)};
     (-4,4);(-4,12) **\dir{-};
     (-12,-4);(-12,12) **\dir{-};
     (4,-4);(4,-12) **\dir{-};(12,4);(12,-12) **\dir{-};
     (16,6)*{\lambda};
     (-10,0)*{};(10,0)*{};
     (-4,-4)*{};(-12,-4)*{} **\crv{(-4,-10) & (-12,-10)}?(1)*\dir{>}?(0)*\dir{>};
      (4,4)*{};(12,4)*{} **\crv{(4,10) & (12,10)}?(1)*\dir{<}?(0)*\dir{<};
     }};
  \endxy
\end{align}

\item {\bf (Odd sl(2) relations)}
\begin{equation}
\begin{gathered}
\label{rei2}
\xy
{\ar (0,-8)*{}; (0,8)*{}};
(-3,0)*{};(3,0)*{};
\endxy\xy
{\ar (0,8)*{}; (0,-8)*{}};
(4,4)*{\l};
(-3,0)*{};(6,0)*{};
\endxy
\;\; =\;\;
 -\text{$\xy
(-3,-8)*{};(-3,8)*{} **\crv{(-3,-4) & (3,-4) & (3,4) & (-3,4)}?(1)*\dir{>};
(3,-8)*{};(3,8)*{} **\crv{(3,-4) & (-3,-4) & (-3,4) & (3,4)}?(0)*\dir{<};
(6,4)*{\l};
(-7,0)*{};(8,0)*{};
\endxy$} + \sum_{\overset{f_1+f_2+f_3}{=\l-1}}
(-1)^{f_2}\;
\text{$\xy
 (3,9)*{};(-3,9)*{} **\crv{(3,4) & (-3,4)} ?(1)*\dir{>} ?(.2)*{\bullet};
(5,6)*{\scriptstyle f_1};
 (0,0)*{\smccbub{}};
 (2.5,0)*{ \bullet};
(7,0)*{\scriptstyle \ast +f_2};
 (-3,-9)*{};(3,-9)*{} **\crv{(-3,-4) & (3,-4)} ?(1)*\dir{>}?(.8)*{\bullet};
(5,-6)*{\scriptstyle f_3};
(-6,4)*{\l};
 (-8,0)*{};(8,0)*{};
\endxy$}\, ,
\\
\xy
{\ar (0,8)*{}; (0,-8)*{}};
(-3,0)*{};(3,0)*{};
\endxy\xy
{\ar (0,-8)*{}; (0,8)*{}};
(4,4)*{\l};
(-3,0)*{};(6,0)*{};
\endxy \;\;=\;\;
 -\text{$\xy
(-3,-8)*{};(-3,8)*{} **\crv{(-3,-4) & (3,-4) & (3,4) & (-3,4)}?(0)*\dir{<};
(3,-8)*{};(3,8)*{} **\crv{(3,-4) & (-3,-4) & (-3,4) & (3,4)}?(1)*\dir{>};
(6,4)*{\l};
(-7,0)*{};(8,0)*{};
\endxy$} + \sum_{\overset{f_1+f_2+f_3}{=-\l-1}}
(-1)^{f_2}\;
\text{$\xy
 (3,9)*{};(-3,9)*{} **\crv{(3,4) & (-3,4)} ?(0)*\dir{<} ?(.2)*{\bullet};
(5,6)*{\scriptstyle f_1};
 (0,0)*{\smcbub{}};
 (-2.5,0)*{ \bullet};
(-6,-3)*{\scriptstyle \ast +f_2};
 (-3,-9)*{};(3,-9)*{} **\crv{(-3,-4) & (3,-4)} ?(0)*\dir{<}?(.8)*{\bullet};
(5,-6)*{\scriptstyle f_3};
(-6,4)*{\l};
 (-8,0)*{};(8,0)*{};
\endxy$}\, .
\end{gathered}
\end{equation}
\end{enumerate}
\end{definition}

\begin{remark} \label{rem:even-odd}
There are no 1-morphisms that change the weight $\l$ by an odd number.  This implies that the 2-category splits
\begin{equation}
\oUcat \cong \oUcat^{\text{even}} \oplus \oUcat^{\text{odd}}
\end{equation}
where $\oUcat^{\text{even}}$ only has even weights and $\oUcat^{\text{odd}}$ only has odd weights.
\end{remark}

Let $\oUqp$ denote the $(Q,\Pi)$-envelope of the 2-category $\oUcat$ and $\oUbar$ the underlying 2-category of the $(Q,\Pi)$ envelope as defined in Definition~\ref{def:Uund}.

 \subsection{Additional properties of $\oUcat$}

For later convenience we record several relations that follows from those in the previous section, see \cite{BE2} for more details.  Let $\lfloor n \rfloor$ denote the greatest integer less than $n$.
\begin{enumerate}

\item {\bf (Dot Slide Relations)}
\begin{equation} \label{rel:rightward_dot_slide}
\vcenter{\xy
  (-4,0)*{}; (4,0)*{} **\crv{(-4,-8) & (4,-8)};
 ?(1)*\dir{>}  ?(.2)*\dir{}+(0,-.1)*{\bullet}+(-2,-2)*{\scs n};
   (-1,-9)*{\scs \l};
\endxy}
\;\;  = \;\;
(-1)^{\lfloor \frac{n}{2}\rfloor}\;\;
\vcenter{\xy
  (-4,0)*{}; (4,0)*{} **\crv{(-4,-8) & (4,-8)};
 ?(1)*\dir{>}  ?(.85)*\dir{}+(0,-.1)*{\bullet}+(2,-2)*{\scs n};
   (-1,-9)*{\scs \l};
\endxy}
\qquad \qquad
\vcenter{
\xy
  (-4,0)*{}; (4,0)*{} **\crv{(-4,8) & (4,8)};
 ?(1)*\dir{>}  ?(.2)*\dir{}+(0,-.1)*{\bullet}+(-2,2)*{\scs n};
   (-1,9)*{\scs \l};
\endxy}
\;\;  = \;\;
(-1)^{\lfloor \frac{n}{2}\rfloor}\;\;
\vcenter{
\xy
  (-4,0)*{}; (4,0)*{} **\crv{(-4,8) & (4,8)};
 ?(1)*\dir{>}  ?(.85)*\dir{}+(0,-.1)*{\bullet}+(2,2)*{\scs n};
   (-1,9)*{\scs \l};
\endxy}
\end{equation}

\begin{align} \label{rel:leftward_dot_slide}
(-1)^{\lfloor \frac{n}{2} \rfloor}
\vcenter{\xy
  (-4,0)*{}; (4,0)*{} **\crv{(-4,-8) & (4,-8)};
 ?(0)*\dir{<}  ?(.2)*\dir{}+(0,-.1)*{\bullet}+(-2,-2)*{\scs n};
   (-1,-9)*{\scs \l};
\endxy}
\;\;  = \;\;
\begin{cases}
\quad \vcenter{\xy
  (-4,0)*{}; (4,0)*{} **\crv{(-4,-8) & (4,-8)};
 ?(0)*\dir{<}  ?(.2)*\dir{}+(6.6,-.1)*{\bullet}+(2,-1.5)*{\scs n};
   (-1,-9)*{\scs \l};
\endxy} \mbox{\hspace{1cm} if $n$ is even}\\ \\
(-1)^{\l} \;\;
\vcenter{\xy
  (-4,0)*{}; (4,0)*{} **\crv{(-4,-8) & (4,-8)};
 ?(0)*\dir{<}  ?(.2)*\dir{}+(6.6,-.1)*{\bullet}+(2,-1.5)*{\scs n};
   (-1,-9)*{\scs \l};
\endxy}
\;\; + \;\;
2 \;\;
\vcenter{\xy
  (-4,0)*{}; (4,0)*{} **\crv{(-4,-8) & (4,-8)};
 ?(0)*\dir{<}  ?(.2)*\dir{}+(6.6,-.1)*{\bullet}+(2,-1.5)*{\scs n-1};
  (0,-10)*{\bigotimes};
   (3,-9)*{\scs \l};
\endxy} \mbox{\hspace{1cm} if $n$ is odd}
\end{cases}
\end{align}
\begin{align}
(-1)^{\lfloor \frac{n}{2} \rfloor} \;
\vcenter{\xy
  (-4,0)*{}; (4,0)*{} **\crv{(-4,8) & (4,8)};
 ?(0)*\dir{<}  ?(.2)*\dir{}+(6.5,0)*{\bullet}+(2.2,-2)*{\scs n};
   (1,9)*{\scs \l};
\endxy}
\;\;  = \;\;
\begin{cases}
\quad \vcenter{\xy
  (-4,0)*{}; (4,0)*{} **\crv{(-4,8) & (4,8)};
 ?(0)*\dir{<}  ?(.2)*\dir{}+(0,0)*{\bullet}+(2,0)*{\scs n};
   (1,9)*{\scs \l};
\endxy} \mbox{\hspace{1cm} if $n$ is even}\\ \\
(-1)^{\l} \;\;
\vcenter{\xy
  (-4,0)*{}; (4,0)*{} **\crv{(-4,8) & (4,8)};
 ?(0)*\dir{<}  ?(.2)*\dir{}+(0,0)*{\bullet}+(2,0)*{\scs n};
   (1,9)*{\scs \l};
\endxy}
\;\; + \;\;
2 \quad
\vcenter{\xy
  (-4,0)*{}; (4,0)*{} **\crv{(-4,8) & (4,8)};
 ?(0)*\dir{<}  ?(.2)*\dir{}+(0,0)*{\bullet}+(-4,0)*{\scs n-1};
 (0,9)*{\bigotimes};
  (3,9)*{\scs \l};
\endxy} \mbox{\hspace{1cm} if $n$ is odd}
\end{cases}
\end{align}

\item {\bf (Bubble Slide Relations)}\label{rel:buble_slide}
\begin{align} \label{eq:ccbubslide}
\xy 0;/r.17pc/:
 (-2,8);(-2,-8); **\dir{-} ?(1)*\dir{>}+(2.3,0)*{\scriptstyle{}}
 ?(.1)*\dir{ }+(2,0)*{\scs };
 (7,0)*{\rbub{n}};
 (8,6)*{\scs \l};
 (-10,0)*{};(10,0)*{};(-2,-8)*{\scs };
 \endxy
 \;\; &= \;\;
 \sum_{r \geq 0} \;\; (2r + 1) \;\;
\xy 0;/r.17pc/:
 (0,8);(0,-8); **\dir{-} ?(1)*\dir{>}+(2.3,0)*{\scriptstyle{}}
 ?(.1)*\dir{ }+(2,0)*{\scs };
 (0,0)*{\txt\large{$\bullet$}};
 (4,0)*{\scs 2r};
 (-9,0)*{\rbub{n-2r}};
 (8,6)*{\scs \l};
 (-10,0)*{};(10,0)*{};(-2,-8)*{\scs };
 \endxy
\qquad \quad
\xy 0;/r.17pc/:
 (3,8);(3,-8); **\dir{-} ?(1)*\dir{>}+(2.3,0)*{\scriptstyle{}}
 ?(.1)*\dir{ }+(2,0)*{\scs };
 (-7,0)*{\lbub{n}};
 (8,6)*{\scs \l};
 (-10,0)*{};(10,0)*{};(-2,-8)*{\scs };
 \endxy
 \;\; &= \;\;
 \sum_{r \geq 0} \;\; (2r + 1) \;\;
\xy 0;/r.17pc/:
 (-2,8);(-2,-8); **\dir{-} ?(1)*\dir{>}+(2.3,0)*{\scriptstyle{}}
 ?(.1)*\dir{ }+(2,0)*{\scs };
 (-2,0)*{\txt\large{$\bullet$}};
 (2,0)*{\scs 2r};
 (9,0)*{\lbub{n-2r}};
 (8,6)*{\scs \l};
 (-10,0)*{};(10,0)*{};(-2,-8)*{\scs };
 \endxy
\end{align}

\item {\bf{(Pitchfork Relations)}}
\begin{align}
\vcenter{
\xy
  (-6,0)*{}; (6,0)*{} **\crv{(-6,6) & (6,8)} ?(1)*\dir{>};
  (0,0)*{}; (-6,8)*{} **\crv{(0,5) & (-6,2)}  ?(1)*\dir{>};
  (-1,9)*{\scs\l};
\endxy}
\;\; = \;\;
\vcenter{\xy
  (-6,0)*{}; (6,0)*{} **\crv{(-6,8) & (6,6)} ?(1)*\dir{>};
  (0,0)*{}; (6,8)*{} **\crv{(0,5) & (6,2)}  ?(1)*\dir{>};
  (-1,9)*{\scs\l};
\endxy}
\quad \quad \quad \quad
\vcenter{
\xy
  (-6,0)*{}; (6,0)*{} **\crv{(-6,6) & (6,8)} ?(1)*\dir{>};
  (0,0)*{}; (-6,8)*{} **\crv{(0,5) & (-6,2)}  ?(0)*\dir{<};
  (-1,9)*{\scs\l};
\endxy}
\;\; = \;\;
\vcenter{\xy
  (-6,0)*{}; (6,0)*{} **\crv{(-6,8) & (6,6)} ?(1)*\dir{>};
  (0,0)*{}; (6,8)*{} **\crv{(0,5) & (6,2)}  ?(0)*\dir{<};
  (-1,9)*{\scs\l};
\endxy}
\end{align}
\begin{align}
\vcenter{
\xy
  (-6,5)*{}; (6,5)*{} **\crv{(-6,-1) & (6,-3)} ?(1)*\dir{>};
  (-6,-3)*{}; (0,5)*{} **\crv{(-6,2) & (0,-1)}  ?(1)*\dir{>};
  (-1,9)*{\scs\l};
\endxy}
\;\; = \;\;
\vcenter{\xy
  (-6,5)*{}; (6,5)*{} **\crv{(-6,-1) & (6,-3)} ?(1)*\dir{>};
  (6,-3)*{}; (0,5)*{} **\crv{(6,2) & (0,-1)}  ?(1)*\dir{>};
  (-1,9)*{\scs\l};
\endxy}
\quad \quad \quad \quad
\vcenter{
\xy
  (-6,5)*{}; (6,5)*{} **\crv{(-6,-1) & (6,-3)} ?(1)*\dir{>};
  (6,-3)*{}; (0,5)*{} **\crv{(6,2) & (0,-1)}  ?(0)*\dir{<};
  (-1,9)*{\scs\l};
\endxy}
\;\; = \;\;
\vcenter{\xy
  (-6,5)*{}; (6,5)*{} **\crv{(-6,-1) & (6,-3)} ?(1)*\dir{>};
  (-6,-3)*{}; (0,5)*{} **\crv{(-6,2) & (0,-1)}  ?(0)*\dir{<};
  (-1,9)*{\scs\l};
\endxy}
\end{align}
\begin{align}
\vcenter{\xy
  (-6,0)*{}; (6,0)*{} **\crv{(-6,6) & (6,8)} ?(0)*\dir{<};
  (0,0)*{}; (-6,8)*{} **\crv{(0,5) & (-6,2)}  ?(1)*\dir{>};
  (-1,9)*{\scs\l};
\endxy}
\;\; = \;\;
\vcenter{
\xy
  (-6,0)*{}; (6,0)*{} **\crv{(-6,8) & (6,6)} ?(0)*\dir{<};
  (0,0)*{}; (6,8)*{} **\crv{(0,5) & (6,2)}  ?(1)*\dir{>};
  (-1,9)*{\scs\l};
\endxy}
\quad \quad \quad \quad
\vcenter{
\xy
  (-6,5)*{}; (6,5)*{} **\crv{(-6,-1) & (6,-3)} ?(0)*\dir{<};
  (-6,-3)*{}; (0,5)*{} **\crv{(-6,2) & (0,-1)}  ?(1)*\dir{>};
  (-1,9)*{\scs\l};
\endxy}
\;\; = \;\;
\vcenter{\xy
  (-6,5)*{}; (6,5)*{} **\crv{(-6,-1) & (6,-3)} ?(0)*\dir{<};
  (6,-3)*{}; (0,5)*{} **\crv{(6,2) & (0,-1)}  ?(1)*\dir{>};
  (-1,9)*{\scs\l};
\endxy}
\end{align}
\begin{align}
\vcenter{
\xy
  (-6,5)*{}; (6,5)*{} **\crv{(-6,-1) & (6,-3)} ?(0)*\dir{<};
  (6,-3)*{}; (0,5)*{} **\crv{(6,2) & (0,-1)}  ?(0)*\dir{<};
  (-1,9)*{\scs\l};
\endxy}
\;\; = -\;\;
\vcenter{\xy
  (-6,5)*{}; (6,5)*{} **\crv{(-6,-1) & (6,-3)} ?(0)*\dir{<};
  (-6,-3)*{}; (0,5)*{} **\crv{(-6,2) & (0,-1)}  ?(0)*\dir{<};
  (-1,9)*{\scs\l};
\endxy}
\quad \quad \quad \quad
\vcenter{
\xy
  (-6,0)*{}; (6,0)*{} **\crv{(-6,6) & (6,8)} ?(0)*\dir{<};
  (0,0)*{}; (-6,8)*{} **\crv{(0,5) & (-6,2)}  ?(0)*\dir{<};
  (-1,9)*{\scs\l};
\endxy}
\;\; = -\;\;
\vcenter{\xy
  (-6,0)*{}; (6,0)*{} **\crv{(-6,8) & (6,6)} ?(0)*\dir{<};
  (0,0)*{}; (6,8)*{} **\crv{(0,5) & (6,2)}  ?(0)*\dir{<};
  (-1,9)*{\scs\l};
\endxy}
\end{align}

\item {\bf{(Curl Relations)}}
For all $n\geq0$ we have,
\begin{equation}\label{rel:Curls}
\xy
(0,-8)*{};(0,8)*{} **\crv{(0,4) & (6,4) & (6,-4) & (0,-4)}?(1)*\dir{>};
(2,-2.5)*{\bullet};
(3,-0.5)*{\scs n};
(6,6)*{\scs \l};
(-3,0)*{};(8,0)*{};
\endxy = - \sum_{r = 0} ^{n - \l} (-1)^{(r+1)}\;\;\xy
 (0,-8)*{};(0,8)*{} **\dir{-} ?(1)*\dir{>} ?(.8)*{\bullet};
(-2,5)*{\scriptstyle r};
 (8,0)*{\lbub{n-r-\l}};
(6,6)*{\scs \l};
 (-3,0)*{};(15,0)*{};
\endxy  \quad\quad\xy
(0,-8)*{};(0,8)*{} **\crv{(0,4) & (-6,4) & (-6,-4) & (0,-4)}?(1)*\dir{>};
(-2,-2.5)*{\bullet};
(-3,-0.5)*{\scs n};
(4,-6)*{\scs \l};
(3,0)*{};(-8,0)*{};
\endxy = \sum_{r=0}^{n +\l +2}(-1)^{\l r}\xy
 (0,-8)*{};(0,8)*{} **\dir{-} ?(1)*\dir{>} ?(.8)*{\bullet};
(2,5)*{\scriptstyle r};
 (-8,0)*{\rbub{\l +n -r}};
(6,-6)*{\scs \l};
 (3,0)*{};(-13,0)*{};
\endxy
\end{equation}
Note that the exact form of the dotted curl relation depends on the placement of the dots inside the curl. See for example, \cite[(5.18) -- (5.21)]{BE2}.  Using the adjunctions the relations
\begin{align} \label{eq:downcurl1}
\vcenter{\xy 
    (-3,4)*{};(3,-4)*{} **\crv{(-3,1) & (3,-1)}?(0)*\dir{<};
    (3,4)*{};(-3,-4)*{} **\crv{(3,1) & (-3,-1)};
    (-3,-4)*{};(3,-4)*{} **\crv{(-3.2,-8) & (3.2,-8)}?(.2)*\dir{};
    (5,0)*{\scs  \l};
    \endxy}
\;\; &= \;\;
\sum_{r =0}^{\l}\;\; (-1)^{(\l+r+1)}\;\;
 \vcenter{\xy 
    (-3,4)*{};(3,4)*{} **\crv{(-3.2,-2) & (3.2,-2)}?(0)*\dir{<} ?(.2)*\dir{}+(0,0)*{\bullet};
    (-4,0)*{\scs r};
    (5,0)*{\scs  \l};
    (2,-6)*{\rbub{(\l-r)}};
    (-8,0)*{};(8,0)*{};
  \endxy }
\\ \nn \\ \label{eq:downcurl2}
\vcenter{\xy 
    (-3,4)*{};(3,-4)*{} **\crv{(-3,1) & (3,-1)};
    (3,4)*{};(-3,-4)*{} **\crv{(3,1) & (-3,-1)}?(1)*\dir{>};
    (-3,4)*{};(3,4)*{} **\crv{(-3.2,8) & (3.2,8)}?(.2)*\dir{};
    (5,0)*{\scs  \l};
    \endxy}
\;\; &= \;\;
\sum_{r = 0}^{-\l}\;\; (-1)^{(\l+r)}\;\;
\vcenter{\xy 
    (-3,-4)*{};(3,-4)*{} **\crv{(-3.2,2) & (3.2,2)}?(0)*\dir{<} ?(.2)*\dir{}+(0,0)*{\bullet};
    (-7,-1)*{\scs (\l-r)};
    (5,0)*{\scs  \l};
    (0,7)*{\lbub{r}};
    (-8,0)*{};(8,0)*{};
  \endxy }
\end{align}
follow.
\end{enumerate}

\subsection{The nondegeneracy conjecture} \label{sec:nondegen}

A spanning set for the space $\Hom_{\oUcat}(x,y)$ between arbitrary 1-morphisms $x,y$ was defined in
\cite[Section 3.4]{Lau-odd}
and simplified in \cite[Section 8]{BE2}.  In both instances it was conjectured that this spanning set is a basis.  For our classification of differentials we need bases for certain hom spaces that are a subset of the full nondegeneracy conjecture.

\medskip
\noindent {\bf Weak nondegeneracy conjecture} The following Hom spaces are spanned over $\Bbbk$ by the elements predicted by the non-degeneracy conjecture:
\begin{equation}
\begin{split}
  \Hom_{\oUcat}^2(\1_{\l},\Pi 1_{\l})
    &=
\left\langle \;
\xy
 (5,4)*{\l};
(0,0)*{\bigotimes};
\endxy
\; \right\rangle \\
 \Hom_{\oUcat}^4(E\1_{\l},E\Pi 1_{\l})
&=
\left\langle \;
\xy 0;/r.17pc/:
 (0,6);(0,-6); **\dir{-} ?(1)*\dir{>}+(2.3,0)*{\scriptstyle{}}
 ?(.1)*\dir{ }+(2,0)*{\scs };
 (0,0)*{\txt\large{$\bullet$}};
 (8,3)*{\scs \l};
 (3,1)*{\scs 2};
 (-6,0)*{};(10,0)*{};(-2,-8)*{\scs };
 \endxy
\; , \;\;
 \xy 0;/r.17pc/:
 (0,6);(0,-6); **\dir{-} ?(1)*\dir{>}+(2.3,0)*{\scriptstyle{}}
 ?(.1)*\dir{ }+(2,0)*{\scs };
 (0,1)*{\txt\large{$\bullet$}};
 (8,3)*{\scs \l};
 (-8,0)*{};(10,0)*{};(-2,-8)*{\scs };
 (5,-4)*{\bigotimes};
 \endxy
\;, \;\;
 \xy 0;/r.17pc/:
 (0,6);(0,-6); **\dir{-} ?(1)*\dir{>}+(2.3,0)*{\scriptstyle{}}
 ?(.1)*\dir{ }+(2,0)*{\scs };
 (8,5)*{\scs \l};
 (-6,0)*{};(10,0)*{};
 (8,-2)*{\lbub{2}};
 \endxy
\; \right\rangle \\
\Hom_{\oUcat}^2(EE\1_{\l},EE\1_{\l})
&=
\left\langle \;
   \xy 0;/r.18pc/:
  (3,4);(3,-4) **\dir{-}?(0)*\dir{<}+(2.3,0)*{};
  (-3,4);(-3,-4) **\dir{-}?(0)*\dir{<}+(2.3,0)*{}?(.5)*{\bullet};
  (6,2)*{\scs\l};
 \endxy
\; , \quad
\xy 0;/r.18pc/:
  (3,4);(3,-4) **\dir{-}?(0)*\dir{<}+(2.3,0)*{}?(.5)*{\bullet};
  (-3,4);(-3,-4) **\dir{-}?(0)*\dir{<}+(2.3,0)*{};
  (6,2)*{\scs \l};
 \endxy
\;, \quad
  \xy 0;/r.18pc/:
  (3,4);(3,-5) **\dir{-}?(0)*\dir{<}+(2.3,0)*{};
  (-3,4);(-3,-5) **\dir{-}?(0)*\dir{<}+(2.3,0)*{};
  (6,-2)*{\bigotimes};
  (6,2)*{\scs \l};
 \endxy
\;, \quad
  \xy 0;/r.18pc/:
  (0,0)*{\xybox{
    (-4,-4)*{};(4,4)*{} **\crv{(-4,-1) & (4,1)}?(1)*\dir{>};
    (4,-4)*{};(-4,4)*{} **\crv{(4,-1) & (-4,1)}?(1)*\dir{>}?(.75)*{\bullet};
    (-.9,2.3)*{\scs 2};
    (7,1)*{\scs \l};
     }};
  \endxy \;,
\xy 0;/r.18pc/:
  (0,0)*{\xybox{
    (4,4);(4,8) **\dir{-}?(1)*\dir{>} ;
    (-4,4);(-4,8) **\dir{-}?(1)*\dir{>}?(0.25)*{\bullet};
    (-4,-4)*{};(4,4)*{} **\crv{(-4,-1) & (4,1)}?(0.75)*{\bullet};
    (4,-4)*{};(-4,4)*{} **\crv{(4,-1) & (-4,1)};
     (8,1)*{\scs \l};
     }};
  \endxy
\right.
\\
& \quad
\left.
 \quad
\;, \quad
  \xy 0;/r.18pc/:
  (0,0)*{\xybox{
    (-4,-4)*{};(4,4)*{} **\crv{(-4,-1) & (4,1)}?(1)*\dir{>}?(.75)*{\bullet};
    (4,-4)*{};(-4,4)*{} **\crv{(4,-1) & (-4,1)}?(1)*\dir{>};
    (.9,2.3)*{\scs 2};
    (7,1)*{\scs \l};
     }};
  \endxy
\;, \quad
    \xy 0;/r.18pc/:
  (0,0)*{\xybox{
    (4,-4);(4,-6) **\dir{-}; (-4,-4);(-4,-6) **\dir{-};
    (-4,-4)*{};(4,4)*{} **\crv{(-4,-1) & (4,-1)}?(1)*\dir{>};
    (4,-4)*{};(-4,4)*{} **\crv{(4,-1) & (-4,-1)}?(1)*\dir{>}?(.75)*{\bullet};
    (8,1)*{\scs \l};
    (7,-4)*{\bigotimes};
    }};
  \endxy
\;, \quad
    \xy 0;/r.18pc/:
  (0,0)*{\xybox{
    (4,-4);(4,-6) **\dir{-}; (-4,-4);(-4,-6) **\dir{-};
    (-4,-4)*{};(4,4)*{} **\crv{(-4,-1) & (4,-1)}?(1)*\dir{>}?(.75)*{\bullet};
    (4,-4)*{};(-4,4)*{} **\crv{(4,-1) & (-4,-1)}?(1)*\dir{>};
    (8,1)*{\scs \l};
    (7,-4)*{\bigotimes};
    }};
  \endxy
\;, \quad
  \xy 0;/r.17pc/:
  (0,0)*{\xybox{    0;/r.18pc/:
    (4,-4);(4,-8) **\dir{-} ;
    (-4,-4);(-4,-8) **\dir{-} ;
    (12,-4)*{ \lbub{\scs 2}};
    (-4,-4)*{};(4,4)*{} **\crv{(-4,1) & (4,1)} ?(1)*\dir{>};
    (4,-4)*{};(-4,4)*{} **\crv{(4,1) & (-4,1)} ?(1)*\dir{>};
     (8,2)*{\scs \l};
     }};
  \endxy
\right\rangle
\end{split}
\end{equation}
\medskip

The results of \cite[Theorem 7.1]{Lau-odd} and \cite{BE2} coupled with the results from \cite{KKO,KKO2} imply that the 2-category $\oUcat$ admits a 2-representation on categories of modules over cyclotomic odd nilHecke algebras.  It should be possible to show the spanning sets above are a basis using this action.     However, it is difficult to extract formulas for the bubbles under this 2-representation so the weak form of the nondegeneracy conjecture remains open.  Note that from these assumptions and the adjunction axioms it is possible to deduce bases for hom spaces involving caps and cups in the corresponding degree.

%
\section{Derivations on the odd 2-category}
%
In this section we give a classification of derivations  on the odd 2-category $\oUcat$ assuming the weak nondegeneracy conjecture from Section~\ref{sec:nondegen}.  Assuming these spanning sets form a basis we are able to reduce degrees of freedom by comparing coefficients of basis elements.  We note that even without the weak nondegeneracy conjecture, we arrive at well defined derivations that suit our purposes for our main categorification result.

Here we look for derivations that are compatible with a natural dg-structure on odd (skew) polynomials which was shown by Ellis and Qi to extend to the odd nilHecke algebra.  To that end, we restrict our attention to differentials of bidgree $(2, \bar{1})$.  Recall that a derivation on a 2-category is just a derivation on the space of 2-morphisms which satisfies the Leibniz rule for both  horizontal and vertical composition of 2-morphisms.

\subsection{General form of derivations}

The most general form of a bidgree $(2,\bar{1})$ differential on the generating 2-morphisms of $\oUcat$ is given by
\begin{align} \label{eq:der1}
\partial \left (
 \xy 0;/r.17pc/:
 (0,6);(0,-6); **\dir{-} ?(1)*\dir{>}+(2.3,0)*{\scriptstyle{}}
 ?(.1)*\dir{ }+(2,0)*{\scs };
 (0,0)*{\txt\large{$\bullet$}};
 (8,3)*{\scs \l};
 (-7,0)*{};(10,0)*{};(-2,-8)*{\scs };
 \endxy \right )
 \;\; &:= \;\;
 \alpha_{1,\l}
\xy 0;/r.17pc/:
 (0,6);(0,-6); **\dir{-} ?(1)*\dir{>}+(2.3,0)*{\scriptstyle{}}
 ?(.1)*\dir{ }+(2,0)*{\scs };
 (0,0)*{\txt\large{$\bullet$}};
 (8,4)*{\scs \l};
 (3,1)*{\scs 2};
 (-7,0)*{};(10,0)*{};(-2,-8)*{\scs };
 \endxy
 \;\; + \;\;
 \alpha_{2,\l}
 \xy 0;/r.17pc/:
 (0,6);(0,-6); **\dir{-} ?(1)*\dir{>}+(2.3,0)*{\scriptstyle{}}
 ?(.1)*\dir{ }+(2,0)*{\scs };
 (0,1)*{\txt\large{$\bullet$}};
 (8,3)*{\scs \l};
 (-7,0)*{};(10,0)*{};(-2,-8)*{\scs };
 (5,-4)*{\bigotimes};
 \endxy
 \;\; + \;\;
 \alpha_{3,\l}
 \xy 0;/r.17pc/:
 (0,6);(0,-6); **\dir{-} ?(1)*\dir{>}+(2.3,0)*{\scriptstyle{}}
 ?(.1)*\dir{ }+(2,0)*{\scs };
 (8,5)*{\scs \l};
 (-6,0)*{};(10,0)*{};
 (8,-2)*{\lbub{2}};
 \endxy
 \\ & \nn \\
   \partial \left(\;
 \xy 0;/r.18pc/:
  (0,0)*{\xybox{
    (-4,-4)*{};(4,4)*{} **\crv{(-4,-1) & (4,1)}?(1)*\dir{>}?(.25)*{};
    (4,-4)*{};(-4,4)*{} **\crv{(4,-1) & (-4,1)}?(1)*\dir{>};
     (8,1)*{\scs \l};
     (-6,0)*{};(10,0)*{};
     }};
  \endxy
	\right) \;\;  &:= \;
    \beta_{1, {\l}}\;\;
 \xy  0;/r.18pc/:
  (3,4);(3,-4) **\dir{-}?(0)*\dir{<}+(2.3,0)*{};
  (-3,4);(-3,-4) **\dir{-}?(0)*\dir{<}+(2.3,0)*{};
  (8,2)*{\scs \l};
 \endxy
 \; + \;
 \beta_{2, {\l}}
 \xy  0;/r.18pc/:
  (0,0)*{\xybox{
    (-4,-4)*{};(4,4)*{} **\crv{(-4,-1) & (4,1)}?(1)*\dir{>};
    (4,-4)*{};(-4,4)*{} **\crv{(4,-1) & (-4,1)}?(1)*\dir{>}?(.7)*{\bullet};
     (8,1)*{\scs \l};
     (-5,0)*{};(10,0)*{};
     }};
  \endxy
  \; + \;
  \beta_{3, {\l}}
  \xy  0;/r.18pc/:
  (0,0)*{\xybox{
    (-4,-4)*{};(4,4)*{} **\crv{(-4,-1) & (4,1)}?(1)*\dir{>}?(.7)*{\bullet};
    (4,-4)*{};(-4,4)*{} **\crv{(4,-1) & (-4,1)}?(1)*\dir{>};
     (8,1)*{\scs \l};
     (-5,0)*{};(10,0)*{};
     }};
  \endxy
  \; + \;
  \beta_{4,{\l}}\;
  \xy  0;/r.18pc/:
  (0,0)*{\xybox{
    (-4,-4)*{};(4,4)*{} **\crv{(-4,-1) & (4,1)}?(1)*\dir{>}?(.25)*{};
    (4,-4)*{};(-4,4)*{} **\crv{(4,-1) & (-4,1)}?(1)*\dir{>};
     (8,3)*{ \scs \l};
     (8,-2)*{\scs \bigotimes};
     (-5,-6)*{};(10,0)*{};
     }};
  \endxy \qquad \qquad \quad \;
\end{align}
\begin{alignat}{2}
  \partial \left(\;
  \vcenter{\xy 
    (-3,-4)*{};(3,-4)*{} **\crv{(-3.2,2) & (3.2,2)}?(1)*\dir{>};
    (4,2)*{\scs  \l};
  \endxy }
    \right)\;  &:= \;
a_{\lambda-2} \;
\vcenter{\xy 
    (-3,-4)*{};(3,-4)*{} **\crv{(-3.2,2) & (3.2,2)}?(1)*\dir{>} ?(.2)*\dir{}+(0,0)*{\bullet};
    (4,2)*{\scs  \l};
  \endxy }
\;  + \;
b_{\lambda-2}
\vcenter{\xy 
    (-3,-4)*{};(3,-4)*{} **\crv{(-3.2,2) & (3.2,2)}?(1)*\dir{>} ;
    (6,2)*{\scs  \l};
    (0,5)*{\bigotimes};
  \endxy }
 \qquad \quad&&  \partial \left(\;
  \vcenter{\xy 
    (-3,4)*{};(3,4)*{} **\crv{(-3.2,-2) & (3.2,-2)}?(1)*\dir{>};
    (4,-1)*{\scs  \l};
  \endxy }\;
    \right)  \; :=\;
\bar{a}_{\lambda}
\vcenter{\xy 
    (-3,4)*{};(3,4)*{} **\crv{(-3.2,-2) & (3.2,-2)}?(1)*\dir{>} ?(.2)*\dir{}+(5,0)*{\bullet};
    (4,-2)*{\scs  \l};
  \endxy }
  + \;
\bar{b}_{\lambda}
\vcenter{\xy 
    (-6,4)*{};(6,4)*{} **\crv{(-5.5,-6) & (5.5,-6)}?(1)*\dir{>} ;
    (6,-2)*{\scs  \l};
    (0,2)*{\bigotimes};
  \endxy } \quad
   \\ & \nn\\
  \partial \left(\;\;
  \vcenter{\xy 
    (-3,-4)*{};(3,-4)*{} **\crv{(-3.2,2) & (3.2,2)}?(0)*\dir{<};
    (4,2)*{\scs  \l};
  \endxy }
    \right) \;  &:=\;
c_{\lambda}
\vcenter{\xy 
    (3,-4)*{};(-3,-4)*{} **\crv{(3.2,2) & (-3.2,2)}?(1)*\dir{>} ?(.2)*\dir{}+(0,0)*{\bullet};
    (4,2)*{\scs  \l};
  \endxy }
   + \;
d_{\lambda}
\vcenter{\xy 
    (6,-4)*{};(-6,-4)*{} **\crv{(5.5,6) & (-5.5,6)}?(1)*\dir{>} ;
    (6,2)*{\scs  \l};
    (0,-2)*{\bigotimes};
  \endxy }
\qquad \quad
  &&\partial \left( \;\;
  \vcenter{\xy 
    (-3,4)*{};(3,4)*{} **\crv{(-3.2,-2) & (3.2,-2)}?(0)*\dir{<};
    (4,-2)*{\scs  \l};
  \endxy }\;
    \right)
\; := \;
\bar{c}_{\lambda-2}\;
\vcenter{\xy 
    (-3,4)*{};(3,4)*{} **\crv{(-3.2,-2) & (3.2,-2)}?(0)*\dir{<} ?(.2)*\dir{}+(0,0)*{\bullet};
    (4,-2)*{\scs  \l};
  \endxy }
  + \;
\bar{d}_{\lambda-2}
\vcenter{\xy 
    (-3,4)*{};(3,4)*{} **\crv{(-3.2,-2) & (3.2,-2)}?(0)*\dir{<} ;
    (6,-2)*{\scs  \l};
    (0,-5)*{\bigotimes};
  \endxy }   \label{eq:derend}
\end{alignat}
for some coefficients in $\Bbbk$.
The image of all identity 2-morphisms are zero.  This definition is extended to arbitrary composites using the Leibniz rule.
By Remark~\ref{rem:even-odd} the derivations can be defined independently on $\oUcat^{\text{even}}$ and on $\oUcat^{\text{odd}}$.

In order for this assignment to define a derivation on $\oUcat$ it must respect the defining relations of the 2-category $\oUcat$.
For example, let us consider the right adjunction axiom \eqref{eq_biadj1}.   The left-hand-side is vertical composite of two 2-morphsism, call them $x$ and $y$.
\begin{equation}
  \xy 0;/r.17pc/:
    (-8,5)*{}="1";
    (0,5)*{}="2";
    (0,-5)*{}="2'";
    (8,-5)*{}="3";
    (-8,-10);"1" **\dir{-}?(.5)*\dir{>};
    (-12,0);(10,0); **[red]\dir{.};
    (-15,4)*{\scs x};
    (-15,-4)*{\scs y};
    "2";"2'" **\dir{-};
    "1";"2" **\crv{(-8,12) & (0,12)} ;
    "2'";"3" **\crv{(0,-12) & (8,-12)};
    "3"; (8,10) **\dir{-}?(1)*\dir{>};
    (15,-9)*{\scs \l};
    (10,8)*{\scs };
    (-10,-8)*{\scs };
    \endxy
    \quad = \quad
  \xy 0;/r.17pc/:
    (0,10);(0,-10); **\dir{-} ?(1)*\dir{>}+(2.3,0)*{\scriptstyle{}}
    ?(.1)*\dir{ }+(2,0)*{\scs };
    (-6,5)*{\scs \l+2};
    (8,5)*{\scs \l};
    (-10,0)*{};(10,0)*{};(-2,-8)*{\scs };
 \endxy
 \end{equation}
Using the Leibniz rule for this vertical composition $x \circ y$ of $x$ and $y$ gives that $\partial (x\circ y) = \partial(x)y + (-1)^{|x|}\partial(y)$, and the parity of $x$ is even, $|x| = 0$.
Hence,
\begin{equation} \label{d of right adjunction}
\partial \left (
 \xy 0;/r.17pc/:
    (-8,5)*{}="1";
    (0,5)*{}="2";
    (0,-5)*{}="2'";
    (8,-5)*{}="3";
    (-8,-10);"1" **\dir{-}?(.5)*\dir{>};
    "2";"2'" **\dir{-};
    "1";"2" **\crv{(-8,12) & (0,12)} ;
    "2'";"3" **\crv{(0,-12) & (8,-12)};
    "3"; (8,10) **\dir{-}?(1)*\dir{>};
    (15,-9)*{\scs \l};
    (-12,9)*{\scs \l+2};
    (10,8)*{\scs };
    (-10,-8)*{\scs };
    \endxy \right)
    \quad = \quad
    (a_{\l} + \bar{a}_{\l})
  \xy 0;/r.17pc/:
    (0,10);(0,-10); **\dir{-} ?(1)*\dir{>}+(2.3,0)*{\scriptstyle{}}
    ?(.1)*\dir{ }+(2,0)*{\scs };
    (0,0)*{\txt\large{$\bullet$}};
    (-6,5)*{\scs \l+2};
    (8,5)*{\scs \l};
    (-10,0)*{};(10,0)*{};(-2,-8)*{\scs };
 \endxy
 \;\; +\;\;
 (b_{\l} +\bar{b}_{\l})
 \xy 0;/r.17pc/:
    (0,10);(0,-10); **\dir{-} ?(1)*\dir{>}+(2.3,0)*{\scriptstyle{}}
    ?(.1)*\dir{ }+(2,0)*{\scs };
    (5,-2)*{\bigotimes};
    (-6,5)*{\scs \l+2};
    (8,5)*{\scs \l};
    (-10,0)*{};(10,0)*{};(-2,-8)*{\scs };
 \endxy
\end{equation}
The image of the right hand side of \eqref{eq_biadj1} under $\partial$ is zero, hence, using the linear independence of the 2-morphisms in \eqref{d of right adjunction} we obtain a relationship between the coefficients
\[
(a_{\l} + \bar{a}_{\l}) = 0, \qquad  (b_{\l} +\bar{b}_{\l})=0.
\]

\begin{lemma} \label{lem:derivation-start}
For the map $\partial \maps \oUcat \to \oUcat$ defined by \eqref{eq:der1}--\eqref{eq:derend} to preserve the odd nilHecke relations, the right adjunction axioms, and the parity left adjoint relations, the coefficients must take the form
\begin{align} \label{lemeq:der1}
  \partial \left ( \;\;
 \xy 0;/r.17pc/:
 (0,6);(0,-6); **\dir{-} ?(1)*\dir{>}+(2.3,0)*{\scriptstyle{}}
 ?(.1)*\dir{ }+(2,0)*{\scs };
 (0,0)*{\txt\large{$\bullet$}};
 (6,3)*{\scs \l};
 \endxy \;\;\right)
 \;\; &:= \;\;
 \alpha_{1,\l}
\xy 0;/r.17pc/:
 (0,6);(0,-6); **\dir{-} ?(1)*\dir{>}+(2.3,0)*{\scriptstyle{}}
 ?(.1)*\dir{ }+(2,0)*{\scs };
 (0,0)*{\txt\large{$\bullet$}};
 (8,3)*{\scs \l};
 (3,1)*{\scs 2};
 (-8,0)*{};(10,0)*{};(-2,-8)*{\scs };
 \endxy
 \;\; + \;\;
 \alpha_{2}
 \xy 0;/r.17pc/:
 (0,6);(0,-6); **\dir{-} ?(1)*\dir{>}+(2.3,0)*{\scriptstyle{}}
 ?(.1)*\dir{ }+(2,0)*{\scs };
 (0,1)*{\txt\large{$\bullet$}};
 (8,5)*{\scs \l};
 (-8,0)*{};(10,0)*{};(-2,-8)*{\scs };
 (5,-4)*{\bigotimes};
 \endxy
 \\ & \nn \\
   \partial \left(
 \xy 0;/r.18pc/:
  (0,0)*{\xybox{
    (-4,-4)*{};(4,4)*{} **\crv{(-4,-1) & (4,1)}?(1)*\dir{>}?(.25)*{};
    (4,-4)*{};(-4,4)*{} **\crv{(4,-1) & (-4,1)}?(1)*\dir{>};
     (8,1)*{\scs \l};
     (-6,0)*{};(10,0)*{};
     }};
  \endxy
	\right)\;\; & := \;
    \beta_1,_{\l}
 \xy  0;/r.18pc/:
  (3,4);(3,-4) **\dir{-}?(0)*\dir{<}+(2.3,0)*{};
  (-3,4);(-3,-4) **\dir{-}?(0)*\dir{<}+(2.3,0)*{};
  (8,2)*{\scs \l};
 \endxy
 \; + \;
 (\beta_{1,\l}-\alpha_{1,\l})
 \xy  0;/r.18pc/:
  (0,0)*{\xybox{
    (-4,-4)*{};(4,4)*{} **\crv{(-4,-1) & (4,1)}?(1)*\dir{>};
    (4,-4)*{};(-4,4)*{} **\crv{(4,-1) & (-4,1)}?(1)*\dir{>}?(.7)*{\bullet};
     (8,1)*{\scs \l};
     (-5,0)*{};(10,0)*{};
     }};
  \endxy
  \; + \;
  (\alpha_{1,\l} - \beta_{1,\l})
  \xy  0;/r.18pc/:
  (0,0)*{\xybox{
    (-4,-4)*{};(4,4)*{} **\crv{(-4,-1) & (4,1)}?(1)*\dir{>}?(.7)*{\bullet};
    (4,-4)*{};(-4,4)*{} **\crv{(4,-1) & (-4,1)}?(1)*\dir{>};
     (8,1)*{\scs \l};
     (-5,0)*{};(10,0)*{};
     }};
  \endxy
  \; + \;
  \alpha_2
  \xy  0;/r.18pc/:
  (0,0)*{\xybox{
    (-4,-4)*{};(4,4)*{} **\crv{(-4,-1) & (4,1)}?(1)*\dir{>}?(.25)*{};
    (4,-4)*{};(-4,4)*{} **\crv{(4,-1) & (-4,1)}?(1)*\dir{>};
     (8,3)*{ \scs \l};
     (8,-2)*{\scs \bigotimes};
     (-5,-6)*{};(10,0)*{};
     }};
  \endxy \quad
\end{align}
\begin{alignat}{2}
  \partial \left(\;
  \vcenter{\xy 
    (-3,-4)*{};(3,-4)*{} **\crv{(-3.2,2) & (3.2,2)}?(1)*\dir{>};
    (4,2)*{\scs  \l};
  \endxy } \;
    \right) \;  &:= \;
a_{\lambda-2} \;
\vcenter{\xy 
    (-3,-4)*{};(3,-4)*{} **\crv{(-3.2,2) & (3.2,2)}?(1)*\dir{>} ?(.2)*\dir{}+(0,0)*{\bullet};
    (4,2)*{\scs  \l};
  \endxy }
  + \;
b_{\lambda-2}\;
\vcenter{\xy 
    (-3,-4)*{};(3,-4)*{} **\crv{(-3.2,2) & (3.2,2)}?(1)*\dir{>} ;
    (6,2)*{\scs  \l};
    (0,5)*{\bigotimes};
  \endxy }
 \qquad && \partial \left( \;
  \vcenter{\xy 
    (-3,4)*{};(3,4)*{} **\crv{(-3.2,-2) & (3.2,-2)}?(1)*\dir{>};
    (4,-2)*{\scs  \l};
  \endxy } \;
    \right) \; :=\;
-a_\l\;
\vcenter{\xy 
    (-3,4)*{};(3,4)*{} **\crv{(-3.2,-2) & (3.2,-2)}?(1)*\dir{>} ?(.2)*\dir{}+(5,0)*{\bullet};
    (4,-2)*{\scs  \l};
  \endxy }
  - \;
b_{\l}
\vcenter{\xy 
    (-6,4)*{};(6,4)*{} **\crv{(-5.5,-6) & (5.5,-6)}?(1)*\dir{>} ;
    (6,-2)*{\scs  \l};
    (0,2)*{\bigotimes};
  \endxy }
   \\ & \nn\\
  \partial \left( \;\;
  \vcenter{\xy 
    (-3,-4)*{};(3,-4)*{} **\crv{(-3.2,2) & (3.2,2)}?(0)*\dir{<};
    (4,2)*{\scs  \l};
  \endxy } \;
    \right) \;& :=\;
c_{\lambda} \;
\vcenter{\xy 
    (3,-4)*{};(-3,-4)*{} **\crv{(3.2,2) & (-3.2,2)}?(1)*\dir{>} ?(.2)*\dir{}+(0,0)*{\bullet};
    (4,2)*{\scs  \l};
  \endxy }
  + \;
d_{\lambda}\;
\vcenter{\xy 
    (6,-4)*{};(-6,-4)*{} **\crv{(5.5,6) & (-5.5,6)}?(1)*\dir{>} ;
    (6,2)*{\scs  \l};
    (0,-2)*{\bigotimes};
  \endxy }
 \qquad
&&
\partial \left(\;\;
  \vcenter{\xy 
    (-3,4)*{};(3,4)*{} **\crv{(-3.2,-2) & (3.2,-2)}?(0)*\dir{<};
    (4,-2)*{\scs  \l};
  \endxy } \;
    \right)   \; := \;
(-1)^{\l}c_{\l-2}\;
\vcenter{\xy 
    (-3,4)*{};(3,4)*{} **\crv{(-3.2,-2) & (3.2,-2)}?(0)*\dir{<} ?(.2)*\dir{}+(0,0)*{\bullet};
    (4,-2)*{\scs  \l};
  \endxy }
  - \;
d_{\l-2}
\vcenter{\xy 
    (-3,4)*{};(3,4)*{} **\crv{(-3.2,-2) & (3.2,-2)}?(0)*\dir{<} ;
    (4,-2)*{\scs  \l};
    (0,-5)*{\bigotimes};
  \endxy }   \label{lemeq:derend}
\end{alignat}
where
\[
2\beta_{1,\l} = \alpha_{1,\l+2} + \alpha_{1,\l}.
\]
\end{lemma}

\begin{proof}
This is a direct computation using the Leibniz rule.  The right adjunction axiom implies $\bar{a}_{\l} = a_{\l}$ and  $\bar{b}_{\l}=-b_{\l}$.  Similarly, the parity left adjoint equation implies $\bar{c}_{\l} = (-1)^{\l}c_{\l}$ and $\bar{d}_{\l} = -d_{\l}$.  The first nilHecke relation in \eqref{eq_nilHecke1} implies
\begin{align*}
& 0 = \partial \left(\;\;
	\vcenter{\xy 0;/r.18pc/:
    (-4,-4)*{};(4,4)*{} **\crv{(-4,-1) & (4,1)};
    (4,-4)*{};(-4,4)*{} **\crv{(4,-1) & (-4,1)};
    (-4,4)*{};(4,12)*{} **\crv{(-4,7) & (4,9)}?(1)*\dir{>};
    (4,4)*{};(-4,12)*{} **\crv{(4,7) & (-4,9)}?(1)*\dir{>};
    (7,8)*{ \l};
 \endxy}
 \;\;\right)
 \;\; = \;\;
 \beta_{1,\l}
 \xy 0;/r.18pc/:
  (0,0)*{\xybox{
    (-4,-4)*{};(4,4)*{} **\crv{(-4,-1) & (4,1)}?(1)*\dir{>}?(.25)*{};
    (4,-4)*{};(-4,4)*{} **\crv{(4,-1) & (-4,1)}?(1)*\dir{>};
     (8,1)*{ \l};
     (-10,0)*{};(10,0)*{};
     }};
  \endxy
  \;\;+\;\;
  \beta_{2,\l} \;\;
  \vcenter{\xy 0;/r.18pc/:
    (-4,-4)*{};(4,4)*{} **\crv{(-4,-1) & (4,1)};
    (4,-4)*{};(-4,4)*{} **\crv{(4,-1) & (-4,1)};
    (-4,4)*{};(4,12)*{} **\crv{(-4,7) & (4,9)}?(1)*\dir{>};
    (4,4)*{};(-4,12)*{} **\crv{(4,7) & (-4,9)}?(1)*\dir{>}?(.7)*{\bullet};
    (7,8)*{ \l};
 \endxy}
 \;\;+\;\;
  \beta_{3,\l}\;\;
  \vcenter{\xy 0;/r.18pc/:
    (-4,-4)*{};(4,4)*{} **\crv{(-4,-1) & (4,1)};
    (4,-4)*{};(-4,4)*{} **\crv{(4,-1) & (-4,1)};
    (-4,4)*{};(4,12)*{} **\crv{(-4,7) & (4,9)}?(1)*\dir{>}?(.7)*{\bullet};
    (4,4)*{};(-4,12)*{} **\crv{(4,7) & (-4,9)}?(1)*\dir{>};
    (7,8)*{ \l};
 \endxy}
 \;\; +\;\;
 \beta_{4,\l} \;\;
 \vcenter{\xy 0;/r.18pc/:
    (-4,-4)*{};(4,4)*{} **\crv{(-4,-1) & (4,1)};
    (4,-4)*{};(-4,4)*{} **\crv{(4,-1) & (-4,1)};
    (-4,4)*{};(4,12)*{} **\crv{(-4,7) & (4,9)}?(1)*\dir{>};
    (4,4)*{};(-4,12)*{} **\crv{(4,7) & (-4,9)}?(1)*\dir{>};
    (7,8)*{ \l};
    (7,2)*{\bigotimes};
 \endxy}
 \\ \nn \\
 & \hspace{3cm} - \;\;
 \beta_1,\l
 \xy 0;/r.18pc/:
  (0,0)*{\xybox{
    (-4,-4)*{};(4,4)*{} **\crv{(-4,-1) & (4,1)}?(1)*\dir{>}?(.25)*{};
    (4,-4)*{};(-4,4)*{} **\crv{(4,-1) & (-4,1)}?(1)*\dir{>};
     (8,1)*{ \l};
     (-10,0)*{};(10,0)*{};
     }};
  \endxy
   -\;\;
  \beta_{2,\l}\;\;
  \vcenter{\xy 0;/r.18pc/:
    (-4,-4)*{};(4,4)*{} **\crv{(-4,-1) & (4,1)};
    (4,-4)*{};(-4,4)*{} **\crv{(4,-1) & (-4,1)}?(.75)*{\bullet};
    (-4,4)*{};(4,12)*{} **\crv{(-4,7) & (4,9)}?(1)*\dir{>};
    (4,4)*{};(-4,12)*{} **\crv{(4,7) & (-4,9)}?(1)*\dir{>};
    (7,8)*{ \l};
 \endxy}
 \;\; -\;\;
  \beta_{3,\l}\;\;
  \vcenter{\xy 0;/r.18pc/:
    (-4,-4)*{};(4,4)*{} **\crv{(-4,-1) & (4,1)}?(.75)*{\bullet};
    (4,-4)*{};(-4,4)*{} **\crv{(4,-1) & (-4,1)};
    (-4,4)*{};(4,12)*{} **\crv{(-4,7) & (4,9)}?(1)*\dir{>};
    (4,4)*{};(-4,12)*{} **\crv{(4,7) & (-4,9)}?(1)*\dir{>};
    (7,8)*{ \l};
 \endxy}
 \;\; - \;\;
\beta_{4,\l}  \;\;
 \vcenter{\xy 0;/r.18pc/:
    (-4,-4)*{};(4,4)*{} **\crv{(-4,-1) & (4,1)};
    (4,-4)*{};(-4,4)*{} **\crv{(4,-1) & (-4,1)};
    (-4,4)*{};(4,12)*{} **\crv{(-4,7) & (4,9)}?(1)*\dir{>};
    (4,4)*{};(-4,12)*{} **\crv{(4,7) & (-4,9)}?(1)*\dir{>};
    (7,8)*{\scs \l};
    (7,-2)*{\bigotimes};
 \endxy}
 \\ \nn \\
 &\hspace{3cm} = \;\;
 (-\beta_{2,\l} - \beta_{3,\l})
 \xy  0;/r.18pc/:
  (0,0)*{\xybox{
    (-4,-4)*{};(4,4)*{} **\crv{(-4,-1) & (4,1)}?(1)*\dir{>}?(.25)*{};
    (4,-4)*{};(-4,4)*{} **\crv{(4,-1) & (-4,1)}?(1)*\dir{>};
     (8,1)*{ \l};
     (-10,0)*{};(10,0)*{};
     }};
  \endxy
  \\ & \nn
\end{align*}
which implies $\beta_{3,\l} = -\beta_{2,\l}$.  Making these substitutions the odd nilHecke relation  \eqref{eq_nilHecke2} involves the terms
\begin{align*}
 \partial \left(
\xy
  (0,0)*{\xybox{
    (-4,-4)*{};(4,4)*{} **\crv{(-4,-1) & (4,-1)}?(1)*\dir{>};
    (4,-4)*{};(-4,4)*{} **\crv{(4,-1) & (-4,-1)}?(1)*\dir{>}?(.75)*{\bullet};
     (8,1)*{\scs \l};
     (-7,0)*{};(10,0)*{};
     }};
  \endxy
  \right)
 & \;\;  = \;\;
  \alpha_{1,\l +2}
  \xy
  (0,0)*{\xybox{
    (-4,-4)*{};(4,4)*{} **\crv{(-4,-1) & (4,-1)}?(1)*\dir{>};
    (4,-4)*{};(-4,4)*{} **\crv{(4,-1) & (-4,-1)}?(1)*\dir{>}?(.75)*{\bullet};
    (-1,2.3)*{\scs 2};
    (7,1)*{\scs \l};
    (-10,0)*{};(10,0)*{};
     }};
  \endxy
  \; + \;\;
  \alpha_{2,\l+2}
  \xy
  (0,0)*{\xybox{
    (4,4);(4,8) **\dir{-}?(1)*\dir{>} ;
    (-4,4);(-4,8) **\dir{-}?(1)*\dir{>}?(0.45)*{\bullet};
    (0,3)*{ \bigotimes};
    (-4,-4)*{};(4,4)*{} **\crv{(-4,-1) & (4,1)};
    (4,-4)*{};(-4,4)*{} **\crv{(4,-1) & (-4,1)};
     (7,1)*{\scs  \l};
     (-10,0)*{};(10,0)*{};
     }};
  \endxy
  \;\; + \;\;
  \alpha_{3,\l+2}
  \xy
  (0,0)*{\xybox{ 0;/r.18pc/:
    (8,4);(8,8) **\dir{-}?(1)*\dir{>} ;
    (-8,4);(-8,8) **\dir{-}?(1)*\dir{>};
     (0,5)*{\lbub{2}};
    (-6,-6)*{};(8,4)*{} **\crv{(-6,-3) & (8,-3)};
    (6,-6)*{};(-8,4)*{} **\crv{(6,-3) & (-8,-3)};
     (8,-3)*{\scs \l};
     (-10,0)*{};(10,0)*{};
     }};
  \endxy
  \\ \nn \\
  & - \;\;
  \beta_{1,\l}
  \xy
  (3,4);(3,-4) **\dir{-}?(0)*\dir{<}+(2.3,0)*{};
  (-3,4);(-3,-4) **\dir{-}?(0)*\dir{<}+(2.3,0)*{}?(.5)*{\bullet};
  (7,2)*{\scs\l};
 \endxy
 \nn
 \;\; - \;\;
 \beta_{2,\l}
  \xy
  (0,0)*{\xybox{
    (-4,-4)*{};(4,4)*{} **\crv{(-4,-1) & (4,-1)}?(1)*\dir{>};
    (4,-4)*{};(-4,4)*{} **\crv{(4,-1) & (-4,-1)}?(1)*\dir{>}?(.75)*{\bullet};
    (-1,2.3)*{\scs 2};
    (7,1)*{\scs \l};
    (-10,0)*{};(10,0)*{};
     }};
  \endxy
  \;\; + \;\;
  \beta_{2,\l}
  \xy
  (0,0)*{\xybox{
    (4,4);(4,8) **\dir{-}?(1)*\dir{>} ;
    (-4,4);(-4,8) **\dir{-}?(1)*\dir{>}?(0.25)*{\bullet};
    (-4,-4)*{};(4,4)*{} **\crv{(-4,-1) & (4,1)}?(0.75)*{\bullet};
    (4,-4)*{};(-4,4)*{} **\crv{(4,-1) & (-4,1)};
     (8,1)*{\scs \l};
     (-10,0)*{};(10,0)*{};
     }};
  \endxy
  \;\; -  \;\;
  \beta_{3,\l}
  \xy
  (0,0)*{\xybox{
    (4,-4);(4,-6) **\dir{-}; (-4,-4);(-4,-6) **\dir{-};
    (-4,-4)*{};(4,4)*{} **\crv{(-4,-1) & (4,-1)}?(1)*\dir{>};
    (4,-4)*{};(-4,4)*{} **\crv{(4,-1) & (-4,-1)}?(1)*\dir{>}?(.75)*{\bullet};
    (8,1)*{\scs \l};
    (7,-4)*{\bigotimes};
    (-10,0)*{};(10,0)*{};
    }};
  \endxy
  \\ \nn \\
 & = \;
   (\alpha_{1,\l +2} -   \beta_{2,\l})
  \xy
  (0,0)*{\xybox{
    (-4,-4)*{};(4,4)*{} **\crv{(-4,-1) & (4,1)}?(1)*\dir{>};
    (4,-4)*{};(-4,4)*{} **\crv{(4,-1) & (-4,1)}?(1)*\dir{>}?(.75)*{\bullet};
    (-1,2.3)*{\scs 2};
    (7,1)*{\scs \l};
     }};
  \endxy
  \;  -  \;
  (\alpha_{2,\l+2} +   \beta_{3,\l})
    \xy
  (0,0)*{\xybox{
    (4,-4);(4,-6) **\dir{-}; (-4,-4);(-4,-6) **\dir{-};
    (-4,-4)*{};(4,4)*{} **\crv{(-4,-1) & (4,-1)}?(1)*\dir{>};
    (4,-4)*{};(-4,4)*{} **\crv{(4,-1) & (-4,-1)}?(1)*\dir{>}?(.75)*{\bullet};
    (8,1)*{\scs \l};
    (7,-4)*{\bigotimes};
    }};
  \endxy
  \;  +  \;
  \alpha_{3,\l+2}
  \xy
  (0,0)*{\xybox{    0;/r.18pc/:
    (4,-4);(4,-8) **\dir{-} ;
    (-4,-4);(-4,-8) **\dir{-} ;
    (12,-4)*{ \lbub{\scs 2}};
    (-4,-4)*{};(4,4)*{} **\crv{(-4,1) & (4,1)} ?(1)*\dir{>};
    (4,-4)*{};(-4,4)*{} **\crv{(4,1) & (-4,1)} ?(1)*\dir{>};
     (8,4)*{\scs \l};
     }};
  \endxy
  \\ \nn \\
\;\; & \qquad
  \;  +  \;
  3\alpha_{3,\l+2}
  \xy
  (0,0)*{\xybox{
    (4,4);(4,4) **\dir{-}?(1)*\dir{>};
    (-4,4);(-4,4) **\dir{-}?(1)*\dir{>};
    (-4,-4)*{};(4,4)*{} **\crv{(-4,-1) & (4,-1)}?(0.75)*{\bullet};
    (4,-4)*{};(-4,4)*{} **\crv{(4,-1) & (-4,-1)};
    (1,2)*{\scs 2};
    (8,2)*{\scs \l};
    }};
  \endxy
- \;\;
  \beta_{1,\l}\;
  \xy
  (3,4);(3,-4) **\dir{-}?(0)*\dir{<}+(2.3,0)*{};
  (-3,4);(-3,-4) **\dir{-}?(0)*\dir{<}+(2.3,0)*{}?(.5)*{\bullet};
  (8,2)*{\scs\l};
 \endxy
   \; + \;
  \beta_{2,\l}
  \xy
  (0,0)*{\xybox{
    (4,4);(4,8) **\dir{-}?(1)*\dir{>} ;
    (-4,4);(-4,8) **\dir{-}?(1)*\dir{>}?(0.25)*{\bullet};
    (-4,-4)*{};(4,4)*{} **\crv{(-4,-1) & (4,1)}?(0.75)*{\bullet};
    (4,-4)*{};(-4,4)*{} **\crv{(4,-1) & (-4,1)};
     (8,1)*{\scs \l};
     (-7,0)*{};(10,0)*{};
     }};
  \endxy
\end{align*}
where the last equality follows from bubble slide relation \eqref{rel:buble_slide}.  Similarly,
\begin{align*}
&\partial \left(
\xy
  (0,0)*{\xybox{
    (-4,-4)*{};(4,4)*{} **\crv{(-4,-1) & (4,1)}?(1)*\dir{>} ;
    (4,-4)*{};(-4,4)*{} **\crv{(4,-1) & (-4,1)}?(1)*\dir{>}?(.25)*{\bullet};
     (8,1)*{\scs \l};
     (-8,0)*{};(10,0)*{};
     }};
  \endxy
  \right)
  \;  =
 \; -
( \beta_{2,\l} + \alpha_{1,\l} )\;\;
 \xy
  (0,0)*{\xybox{
    (-4,-4)*{};(4,4)*{} **\crv{(-4,-1) & (4,-1)}?(1)*\dir{>} ;
    (4,-4)*{};(-4,4)*{} **\crv{(4,-1) & (-4,-1)}?(1)*\dir{>}?(.75)*{\bullet};
    (-5,2)*{\scs 2};
    (7,1)*{\scs \l};
     }};
  \endxy
 \;\;+\;\;
 (\beta_{3,\l} + \alpha_{2,\l})\;\;
 \xy
  (0,0)*{\xybox{
    (4,-4);(4,-6) **\dir{-}; (-4,-4);(-4,-6) **\dir{-};
    (-4,-4)*{};(4,4)*{} **\crv{(-4,-1) & (4,-1)}?(1)*\dir{>};
    (4,-4)*{};(-4,4)*{} **\crv{(4,-1) & (-4,-1)}?(1)*\dir{>}?(.75)*{\bullet};
    (8,1)*{\scs \l};
    (7,-4)*{\bigotimes};
    }};
  \endxy
\;\;-\;\;
  \alpha_{3,\l}
  \xy
  (0,0)*{\xybox{    0;/r.18pc/:
    (4,-4);(4,-8) **\dir{-} ;
    (-4,-4);(-4,-8) **\dir{-} ;
    (12,-4)*{ \lbub{\scs 2}};
    (-4,-4)*{};(4,4)*{} **\crv{(-4,1) & (4,1)} ?(1)*\dir{>};
    (4,-4)*{};(-4,4)*{} **\crv{(4,1) & (-4,1)} ?(1)*\dir{>};
     (8,4)*{\scs \l};
     (-7,0)*{};(10,0)*{};
     }};
  \endxy
  \\ \nn \\
  & \quad
  +
  (\beta_{1,\l}  -\beta_{2,\l} -  \alpha_{1,\l})\;\;
  \xy
  (3,4);(3,-4) **\dir{-}?(0)*\dir{<}+(2.3,0)*{}?(.5)*{\bullet};
  (-3,4);(-3,-4) **\dir{-}?(0)*\dir{<}+(2.3,0)*{};
  (7,2)*{\scs \l};
 \endxy
    +
 (\beta_{2,\l} +  \alpha_{1,\l}) \;\;
  \xy
  (3,4);(3,-4) **\dir{-}?(0)*\dir{<}+(2.3,0)*{};
  (-3,4);(-3,-4) **\dir{-}?(0)*\dir{<}+(2.3,0)*{}?(.5)*{\bullet};
  (7,2)*{\scs \l};
 \endxy
\; -
  \beta_{2,\l}
  \xy
  (0,0)*{\xybox{
    (4,4);(4,8) **\dir{-}?(1)*\dir{>} ;
    (-4,4);(-4,8) **\dir{-}?(1)*\dir{>}?(0.25)*{\bullet};
    (-4,-4)*{};(4,4)*{} **\crv{(-4,-1) & (4,1)}?(0.75)*{\bullet};
    (4,-4)*{};(-4,4)*{} **\crv{(4,-1) & (-4,1)};
     (8,1)*{\scs \l};
     }};
  \endxy
  -\;
 ( \beta_{3,\l} + \alpha_{2,\l}) \;\;
  \xy
  (3,4);(3,-5) **\dir{-}?(0)*\dir{<}+(2.3,0)*{};
  (-3,4);(-3,-5) **\dir{-}?(0)*\dir{<}+(2.3,0)*{};
  (6,-2)*{\bigotimes};
  (8,2)*{\scs \l};
 \endxy
\end{align*}
Therefore, compatibility with relation \eqref{eq_nilHecke2} requires
\begin{align*}
\partial \left(
\xy 0;/r.19pc/:
  (0,0)*{\xybox{
    (-4,-4)*{};(4,4)*{} **\crv{(-4,-1) & (4,1)}?(1)*\dir{>};
    (4,-4)*{};(-4,4)*{} **\crv{(4,-1) & (-4,1)}?(1)*\dir{>}?(.73)*{\bullet};
     (8,1)*{\scs \l};
     (-10,0)*{};(10,0)*{};
     }};
  \endxy
  \right)
  \;\; + \;\;
\partial \left(
\xy 0;/r.19pc/:
  (0,0)*{\xybox{
    (-4,-4)*{};(4,4)*{} **\crv{(-4,-1) & (4,1)}?(1)*\dir{>} ;
    (4,-4)*{};(-4,4)*{} **\crv{(4,-1) & (-4,1)}?(1)*\dir{>}?(.25)*{\bullet};
     (8,1)*{\scs \l};
     (-10,0)*{};(10,0)*{};
     }};
  \endxy
  \right)
  \;\; = 0
\end{align*} so   assuming the weak non-degeneracy conjecture
we get the following set of equations:
\begin{equation}
\begin{split}
  & \alpha_{1,\l+2} -\alpha_{1,\l} -2\beta_{2,\l} =0 \\
& \alpha_{2,\l+2} -\alpha_{2,\l}=0
\end{split}
\qquad
\begin{split}
& \alpha_{3,\l+2} - \alpha_{3,\l}=0 \\
& \alpha_{3,\l+2}=0
\end{split}
\qquad
\begin{split}
  & \beta_{3,\l}+\alpha_{2,\l}=0\\
& \beta_{1,\l} - \beta_{2,\l} - \alpha_{1,\l}=0
\end{split}
\end{equation}
From which we can deduce that  $\alpha_{2,\l}$ does not depend on the weight $\l$ in $\oUcat^{\text{even}}$ or $\l$ in $\oUcat^{\text{even}}$, so we set
$
\alpha_2 := \alpha_{2,\l} = \alpha_{2,\l+2}$,
and $\alpha_{3,\l}=0$ for all $\l$.
If we combine the first and the last equations we get
\begin{equation} \label{eqn:d of odd nilhecke2}
2\beta_{1,\l} = \alpha_{1,\l+2} +\alpha_{1,\l}.
\end{equation}
Equation \eqref{eqn:d of odd nilhecke2}  is redundantly implied by preserving the second nilhecke relation of \eqref{eq_nilHecke1}.
\end{proof}

\begin{lemma}
For $n\geq 0$, the map $\partial$ in Lemma~\ref{lem:derivation-start} satisfies
\begin{align}
\partial \left( \;
    \xy 0;/r.17pc/:
 (0,10);(0,-10); **\dir{-} ?(1)*\dir{>}+(2.3,0)*{\scriptstyle{}}
 ?(.1)*\dir{ }+(2,0)*{\scs };
 (0,0)*{\txt\large{$\bullet$}};
 (-6,5)*{ \scs \l+2};
 (8,5)*{ \scs \l};
 (4,0)*{\scs n};
 (-10,0)*{};(10,0)*{};(-2,-8)*{\scs };
 \endxy
  \; \right)
  \;\; &= \;\;
  \alpha_{1,\l} \delta_{n,\text{odd}}\;\;
   \xy 0;/r.17pc/:
 (0,10);(0,-10); **\dir{-} ?(1)*\dir{>}+(2.3,0)*{\scriptstyle{}}
 ?(.1)*\dir{ }+(2,0)*{\scs };
 (0,1)*{\txt\large{$\bullet$}};
 (7,1)*{\scs n+1};
 (-6,6)*{ \scs \l+2};
 (8,6)*{ \scs \l};
 (-10,0)*{};(10,0)*{};(-2,-8)*{\scs };
 \endxy
 \;\; + \;\;
  (-1)^{n+1}n \alpha_2
 \xy 0;/r.17pc/:
 (0,10);(0,-10); **\dir{-} ?(1)*\dir{>}+(2.3,0)*{\scriptstyle{}}
 ?(.1)*\dir{ }+(2,0)*{\scs };
 (0,1)*{\txt\large{$\bullet$}};
 (3,1)*{\scs n};
 (-6,5)*{\scs \l+2};
 (8,5)*{\scs \l};
 (4,-5.5)*{\bigotimes};
 \endxy
\\
\partial \left(\;
 \xy 0;/r.17pc/:
 (0,10);(0,-10); **\dir{-} ?(0)*\dir{<}+(2.3,0)*{\scriptstyle{}}
 ?(.1)*\dir{ }+(2,0)*{\scs };
 (0,0)*{\txt\large{$\bullet$}};
 (4,0)*{\scs n};
 (-6,5)*{ \scs \l-2};
 (8,5)*{ \scs \l};
 (-10,0)*{};(10,0)*{};(-2,-8)*{\scs };
 \endxy \;
\right)
\;\; &= \;\;
(-2 a_{\l} - \alpha_{1,\l})\delta_{n,\text{odd}}
\xy 0;/r.17pc/:
 (0,10);(0,-10); **\dir{-} ?(0)*\dir{<}+(2.3,0)*{\scriptstyle{}}
 ?(.1)*\dir{ }+(2,0)*{\scs };
 (0,0)*{\txt\large{$\bullet$}};
 (-6,5)*{\scs \l-2};
 (8,5)*{\scs \l};
 (5,0)*{\scs n+1};
 (-10,0)*{};(10,0)*{};(-2,-8)*{\scs };
 \endxy
 \;\; + \;\;
 (-1)^{n+1}n\alpha_2
 \xy 0;/r.17pc/:
 (0,10);(0,-10); **\dir{-} ?(0)*\dir{<}+(2.3,0)*{\scriptstyle{}}
 ?(.1)*\dir{ }+(2,0)*{\scs };
 (0,1)*{\txt\large{$\bullet$}};
 (4,2)*{\scs n};
 (-6,5)*{\scs \l-2};
 (8,5)*{\scs \l};
 (-10,0)*{};(10,0)*{};(-2,-8)*{\scs };
 (4,-5.5)*{\bigotimes};
 \endxy
\end{align}
\end{lemma}

\begin{proof}
The claim follows by induction on the number of dots using the Leibniz rule.
\end{proof}

\subsection{Derivations and bubble relations}

The remaining relations in $\oUcat$ involve dotted bubbles.  We first compute the image of the map defined in Lemma~\ref{lem:derivation-start} on the odd bubble defined in \eqref{d of odd bubble}.    By a direct computation we have
\begin{align} \label{d of odd bubble}
\partial \left(
 \xy 0;/r.17pc/:
 (6,5)*{\scs \l};
 (-10,0)*{};(10,0)*{};(-2,-8)*{\scs };
 (0,0)*{\bigotimes};
 \endxy
 \right)
 \;\; = \;\;
 \begin{cases}
 (a_{\l -2} + c_{\l -2} + \alpha_{1,\l-2}\delta_{\l,\text{odd}})
 \xy 0;/r.17pc/:
 (6,5)*{\scs \l};
 (-6,0)*{};(10,0)*{};
 (0,0)*{\lbub{\scs 2}};
 \endxy
 \mbox{ if } \l\geq 0 \\ \\
 (a_{\l} + c_{\l } + \alpha_{1,\l }\delta_{\l,\text{odd}})
 \xy 0;/r.17pc/:
 (6,5)*{\scs \l};
 (-6,0)*{};(10,0)*{};
 (0,0)*{\rbub{\scs 2}};
 \endxy
 \mbox{ if } \l\leq 0
 \end{cases}
\end{align}

\begin{lemma}
 For the map $\partial$ defined in Lemma~\ref{lem:derivation-start} to preserve the odd cyclicity relation
 \eqref{eq_cyclic_dot}
\begin{equation*} 
\partial \left (
    \xy 0;/r.17pc/:
    (8,5)*{}="1";
    (0,5)*{}="2";
    (0,-5)*{}="2'";
    (-8,-5)*{}="3";
    (8,-10);"1" **\dir{-}?(1)*\dir{<};
    "2";"2'" **\dir{-} ?(.5)*\dir{<};
    "1";"2" **\crv{(8,12) & (0,12)} ?(0)*\dir{<};
    "2'";"3" **\crv{(0,-12) & (-8,-12)};
    "3"; (-8,10) **\dir{-};
    (-12,-9)*{\l};
    (0,4)*{\txt\large{$\bullet$}};
    (-10,8)*{\scs };
    (10,-8)*{\scs };
    \endxy \right )
     \quad = \quad 2 \;
     \partial \left (
           \xy 0;/r.17pc/:
 (4,10);(4,-10); **\dir{-} ?(0)*\dir{<}+(2.3,0)*{\scriptstyle{}}
 ?(.1)*\dir{ }+(2,0)*{\scs };
 (-5,0)*{\txt\large{$\bigotimes$}};
 (-8,8)*{ \l};
 (-10,0)*{};(10,0)*{};(-2,-8)*{\scs };
 \endxy \right )
 \;\; - \;\;
 \partial \left (
         \xy 0;/r.17pc/:
    (-8,5)*{}="1";
    (0,5)*{}="2";
    (0,-5)*{}="2'";
    (8,-5)*{}="3";
    (-8,-10);"1" **\dir{-};
    "2";"2'" **\dir{-} ?(.5)*\dir{<};
    "1";"2" **\crv{(-8,12) & (0,12)} ?(0)*\dir{<};
    "2'";"3" **\crv{(0,-12) & (8,-12)}?(1)*\dir{<};
    "3"; (8,10) **\dir{-};
    (-12,9)*{\l};
    (0,4)*{\txt\large{$\bullet$}};
    (10,8)*{\scs };
    (-10,-8)*{\scs };
    \endxy \right )
\end{equation*}
we must have
\begin{equation} \label{eqn:c_in_terms_of_a_alpha}
c_{\l} = -a_{\l} -\delta_{\l,\text{odd}}\alpha_{1,\l}.
\end{equation}
\end{lemma}

\begin{proof}
Applying $\partial$ to \eqref{eq_cyclic_dot} implies
\begin{equation*}
(-2a_{\l} - \alpha_{1,\l})
 \xy 0;/r.17pc/:
 (4,10);(4,-10); **\dir{-} ?(0)*\dir{<}+(2.3,0)*{\scriptstyle{}}
 ?(.1)*\dir{ }+(2,0)*{\scs };
 (4,0)*{\txt\large{$\bullet$}};
 (1,0)*{\scs 2};
 (-8,8)*{\scs \l};
 (-10,0)*{};(10,0)*{};(-2,-8)*{\scs };
 \endxy
  +\;\;
 \alpha_2
 \xy 0;/r.17pc/:
 (0,10);(0,-10); **\dir{-} ?(0)*\dir{<}+(2.3,0)*{\scriptstyle{}}
 ?(.1)*\dir{ }+(2,0)*{\scs };
 (0,0)*{\txt\large{$\bullet$}};
 (5,-4)*{\txt\large{$\bigotimes$}};
 (-8,8)*{\scs \l};
 (-10,0)*{};(10,0)*{};(-2,-8)*{\scs };
 \endxy
 \;\; = \;\;
(2c_{\l} + (-1)^{\l+1}\alpha_{1,\l})
 \xy 0;/r.17pc/:
 (4,10);(4,-10); **\dir{-} ?(0)*\dir{<}+(2.3,0)*{\scriptstyle{}}
 ?(.1)*\dir{ }+(2,0)*{\scs };
 (4,0)*{\txt\large{$\bullet$}};
 (1,0)*{\scs 2};
 (-8,8)*{\scs \l};
 (-10,0)*{};(10,0)*{};(-2,-8)*{\scs };
 \endxy
 +\;\;
 \alpha_2
 \xy 0;/r.17pc/:
 (0,10);(0,-10); **\dir{-} ?(0)*\dir{<}+(2.3,0)*{\scriptstyle{}}
 ?(.1)*\dir{ }+(2,0)*{\scs };
 (0,0)*{\txt\large{$\bullet$}};
 (5,-4)*{\txt\large{$\bigotimes$}};
 (-8,8)*{\scs \l};
 (-10,0)*{};(10,0)*{};(-2,-8)*{\scs };
 \endxy
\end{equation*}
so comparing coefficients of the basis elements in the weak nondegeneracy conjecture implies
\[
2c_{\l} +(-1)^{\l+1}\alpha_{1,\l} = -2a_{\l} - \alpha_{1,\l}
\]
and the result follows.
\end{proof}

The lemma implies that any derivation $\partial$ must kill the odd bubble
\begin{align}
\partial \left(
 \xy 0;/r.17pc/:
 (6,5)*{\scs \l};
 (-10,0)*{};(10,0)*{};(-2,-8)*{\scs };
 (0,0)*{\bigotimes};
 \endxy
 \right)  \;\; = \;\; 0 ,
\end{align}
so that the centrality of the odd bubble relation \eqref{rel: centrality of odd bubbles} holds trivially.
Note that the real odd bubble is equal to the fake odd bubble using the relations of odd 2-category $\oUcat$
\begin{align*}
\xy 0;/r.17pc/:
 (8,5)*{\scs \l};
 (-10,0)*{};(10,0)*{};(-2,-8)*{\scs };
 (4,0)*{\lbub{\scs 1}};
 \endxy
 \;\; = \;\;
 \xy 0;/r.17pc/:
 (8,5)*{\scs \l};
 (-10,0)*{};(10,0)*{};(-2,-8)*{\scs };
 (4,0)*{\rbub{\scs 1}};
 \endxy
\end{align*}
for all $\l \in \Z$.  This is an immediate consequence of \cite[equation (5.8)]{BE2}.

\begin{lemma} \label{lem:odd-bub}
The derivation of an odd labeled (real) bubble is zero. That is, for $n\geq 0$,
\begin{align}
& \partial \left(
 \xy 0;/r.17pc/:
 (6,5)*{\scs \l};
 (-10,0)*{};(10,0)*{};(-2,-8)*{\scs };
 (0,0)*{\lbub{\scs 2n+1}};
 \endxy
 \right)
 \;\; = 0, \quad \text{for $\l \geq 0$}
\qquad \quad \partial \left(
 \xy 0;/r.17pc/:
 (6,5)*{\scs \l};
 (-10,0)*{};(10,0)*{};(-2,-8)*{\scs };
 (0,0)*{\rbub{\scs 2n+1}};
 \endxy
 \right)
 \;\; = 0 \quad \text{for $\l \leq 0$.}
\end{align}
\end{lemma}
\begin{proof}
The proof of the statement follows easily using the relation \eqref{rel:decomp_of_bubble_2n+1_into_2n_and_1}, the previous Lemma, and the Leibniz rule.
\end{proof}

\begin{lemma} \label{lem:even-bub}
 For the map $\partial$ defined in Lemma~\ref{lem:derivation-start} to preserve the degree zero bubble relation \eqref{eq:degzero-bub} we have
 \begin{equation} \label{eqn:d(deg-0 bubble)=0}
- a_{\l} - b_{\l} + c_{\l}-(-1)^{\l}d_{\l}-\alpha_{1,\l}\delta_{\l,\text{even}} + (\l+1)\alpha_2=0
\end{equation}
for all $\l \in \Z$,
so that any derivation of a dotted bubble must be given by
\begin{align}
\partial \left(
 \xy 0;/r.17pc/:
 (6,5)*{\scs \l};
 (-10,0)*{};(10,0)*{};(-2,-8)*{\scs };
 (0,0)*{\lbub{\scs n}};
 \endxy
 \right)
 \;\; = \;\;
 -\delta_{n,\text{even}}n\alpha_2
 \xy 0;/r.17pc/:
 (6,5)*{\scs \l};
 (-6,0)*{};(10,0)*{};
 (0,0)*{\lbub{\scs {n+1}}};
 \endxy \quad
 \mbox{ for $\l \geq 0$,}
\end{align}
\begin{align}
\partial \left(
 \xy 0;/r.17pc/:
 (6,5)*{\scs \l};
 (-10,0)*{};(10,0)*{};(-2,-8)*{\scs };
 (0,0)*{\rbub{\scs n}};
 \endxy
 \right)
 \;\; = \;\;
 -\delta_{n,\text{even}}n\alpha_2
 \xy 0;/r.17pc/:
 (6,5)*{\scs \l};
 (-6,0)*{};(10,0)*{};
 (0,0)*{\rbub{\scs {n+1}}};
 \endxy \quad
 \mbox{ for $\l \leq 0$.}
\end{align}
\end{lemma}

\begin{proof}
For $n \geq 0$ the image under $\partial$ of the $n$-labelled dotted  bubble  is given by
\begin{align}
\partial \left(
 \xy 0;/r.17pc/:
 (6,5)*{\scs \l};
 (-10,0)*{};(10,0)*{};(-2,-8)*{\scs };
 (0,0)*{\lbub{\scs n}};
 \endxy
 \right)
 \;\; = \;\;
 \delta_{n,\text{even}}(a_{\l-2} +b_{\l -2} - c_{\l-2} + (-1)^{\l}d_{\l-2} + \alpha_{1,\l-2}\delta_{\l,\text{even}} -(n+\l -1)\alpha_2 )
 \xy 0;/r.17pc/:
 (6,5)*{\scs \l};
 (-6,0)*{};(10,0)*{};
 (0,0)*{\lbub{\scs {n+1}}};
 \endxy
\end{align}
for $\l \geq 0$, and
\begin{align}
\partial \left(
 \xy 0;/r.17pc/:
 (6,5)*{\scs \l};
 (-10,0)*{};(10,0)*{};(-2,-8)*{\scs };
 (0,0)*{\rbub{\scs n}};
 \endxy
 \right)
 \;\; = \;\;
 \delta_{n,\text{even}}(- a_{\l} -b_{\l}+ c_{\l}+(-1)^{\l+1}d_{\l}  -\alpha_{1,\l}\delta_{\l,\text{even}} -(n-\l-1)\alpha_2 )
 \xy 0;/r.17pc/:
 (6,5)*{\scs \l};
 (-6,0)*{};(10,0)*{};
 (0,0)*{\rbub{\scs {n+1}}};
 \endxy
\end{align}
for $\l \leq 0$.  The identity by \eqref{eq:degzero-bub} then implies that the degree zero bubble vanishes in the image of $\partial$
\begin{align*} 
&0 = \partial \left(
\xy 0;/r.17pc/:
 (6,5)*{\scs \l};
 (-6,0)*{};(10,0)*{};
 (0,0)*{\lbub{\scs 0}};
 \endxy
 \right )
  =
 \big(a_{\l-2}+b_{\l-2}- c_{\l-2} +(-1)^{\l}d_{\l-2}+\alpha_{1,\l-2}\delta_{\l,\text{even}} - (\l-1)\alpha_2\big)
 \xy 0;/r.17pc/:
 (6,5)*{\scs \l};
 (-6,0)*{};(10,0)*{};
 (0,0)*{\bigotimes};
 \endxy
\quad \text{for $\l \geq 1$,}
\\
&0 = \partial \left(
\xy 0;/r.17pc/:
 (6,5)*{\scs \l};
 (-6,0)*{};(10,0)*{};
 (0,0)*{\rbub{\scs 0}};
 \endxy
 \right )
=
 \big( - a_{\l} - b_{\l} +c_{\l}-(-1)^{\l}d_{\l}-\alpha_{1,\l}\delta_{\l,\text{even}} + (\l+1)\alpha_2 \big)
 \xy 0;/r.17pc/:
 (6,5)*{\scs \l};
 (-6,0)*{};(10,0)*{};
 (0,0)*{\bigotimes};
 \endxy
\quad
\text{for $\l \leq 1$,}
\end{align*}
so the result follows.
\end{proof}

\begin{remark}
The computations above are technically for real bubbles -- those with a positive number of dots.  However,  using odd infinite Grassmannian relation \eqref{rel:odd_inf_grassm} and \eqref{rel:decomp_of_bubble_2n+1_into_2n_and_1} to express fake bubbles in terms of the real bubbles,   the same formulas given in Lemma~\ref{lem:odd-bub} and \ref{lem:even-bub} will apply to fake bubbles as well.
\end{remark}

If we combine \eqref{eqn:c_in_terms_of_a_alpha} with the equation \eqref{eqn:d(deg-0 bubble)=0}  obtained from $\partial$ of degree-0 bubble is zero, we can express $d_{\l}$ as
\begin{equation} \label{eq:dlambda0}
d_{\l} = (-1)^{\l+1}(2a_{\l} +\alpha_{1,\l} +b_{\l} - (\l+1)\alpha_2) \qquad \quad \text{for all $\l \in \Z$.}
\end{equation}

\subsection{Derivations and curl relations}
Before proving the odd $\mf{sl}(2)$-relations it is convenient to study the image of some of the curl relations under the map $\partial$. We continue using the definition Lemma~\ref{lem:derivation-start} imposing the additional constraints from  \eqref{eqn:c_in_terms_of_a_alpha} and \eqref{eq:dlambda0}.

\begin{lemma}
Fix either   $\oUcat^{\text{even}}$ or   $\oUcat^{\text{odd}}$.  For the map $\partial$ defined in Lemma~\ref{lem:derivation-start} to preserve the curl relations
\begin{align} \label{eq:dot-curl1}
\vcenter{\xy 
    (-3,4)*{};(3,-4)*{} **\crv{(-3,1) & (3,-1)}?(1)*\dir{>};
    (3,4)*{};(-3,-4)*{} **\crv{(3,1) & (-3,-1)};
    (-3,4)*{};(3,4)*{} **\crv{(-3.2,8) & (3.2,8)}?(.2)*\dir{}+(5.5,-1.5)*{\bullet};
    (6,4)*{\scs  -\l};
    (-5,0)*{\scs  \l};
    \endxy}
    \;\; = \;\;
\vcenter{\xy 
    (-3,-4)*{};(3,-4)*{} **\crv{(-3.2,2) & (3.2,2)}?(1)*\dir{>} ?(.2)*\dir{};
    (-5,0)*{\scs  \l};
  \endxy }
  \qquad \text{for $\l \leq 0$,}
\qquad \qquad
  \vcenter{\xy 
    (-3,4)*{};(3,-4)*{} **\crv{(-3,1) & (3,-1)};
    (3,4)*{};(-3,-4)*{} **\crv{(3,1) & (-3,-1)}?(0)*\dir{<};
    (-3,-4)*{};(3,-4)*{} **\crv{(-3.2,-8) & (3.2,-8)}?(.2)*\dir{}+(0,0)*{\bullet};
    (-5,-5)*{\scs  \l};
    (5,0)*{\scs  \l};
    \endxy}
    \;\;=\;\;
\vcenter{\xy 
    (-3,4)*{};(3,4)*{} **\crv{(-3.2,-2) & (3.2,-2)}?(1)*\dir{>};
    (4,-2)*{\scs  \l};
    \endxy} \qquad \text{for $\l \geq 0$,}
\end{align}
we must have
\begin{equation} \label{eq:alpha1}
\alpha_{1}:= \alpha_{1,\l} = \alpha_{1,\l+2} = \beta_{1,\l}=\beta_{1,\l+2}, \qquad \quad \text{for all $\l \in \Z$.}
\end{equation}

\end{lemma}

\begin{proof}
This is a straightforward computation  after deriving the formulas for sideways crossings.  For the $\l \geq 0$ case we have
\begin{align*}
\partial \left( \;\;
\vcenter{\xy 
    (-3,4)*{};(3,-4)*{} **\crv{(-3,1) & (3,-1)};
    (3,4)*{};(-3,-4)*{} **\crv{(3,1) & (-3,-1)}?(0)*\dir{<};
    (-3,-4)*{};(3,-4)*{} **\crv{(-3.2,-8) & (3.2,-8)}?(.2)*\dir{}+(0,0)*{\bullet};
    (-5,-5)*{\scs  \l};
    (5,0)*{\scs  \l};
    \endxy} \;\; \right)
    \;=\;
-a_{\l}
\vcenter{\xy 
    (-3,4)*{};(3,4)*{} **\crv{(-3.2,-2) & (3.2,-2)}?(1)*\dir{>}+(0,-2)*{\bullet};
    (4,-2)*{\scs  \l};
    (-6,0)*{};(8,0)*{};
    \endxy}
      + \;\;
(2a_{\l-2} + \alpha_{1,\l} -b_{\l} + b_{\l-2} +(1-\l)\alpha_2 + (-1)^{\l}d_{\l-2})
\vcenter{\xy 
    (-3,4)*{};(3,4)*{} **\crv{(-3.2,-2) & (3.2,-2)}?(1)*\dir{>};
    (4,-2)*{\scs  \l};
    (0,-5)*{\bigotimes};
    (-8,0)*{};(8,0)*{};
    \endxy}
\end{align*}
whereas
\begin{align*}
\partial \left(
\vcenter{\xy 
    (-3,4)*{};(3,4)*{} **\crv{(-3.2,-2) & (3.2,-2)}?(1)*\dir{>};
    (4,-2)*{\scs  \l};
    (-6,0)*{};(6,0)*{};
    \endxy} \;\; \right)
\;\; = \;\;
-a_{\l}
\vcenter{\xy 
    (-3,4)*{};(3,4)*{} **\crv{(-3.2,-2) & (3.2,-2)}?(1)*\dir{>}+(0,-2)*{\bullet};
    (4,-2)*{\scs  \l};
    (-6,0)*{};(6,0)*{};
    \endxy}
\;\; -b_{\l}
\vcenter{\xy 
    (-3,4)*{};(3,4)*{} **\crv{(-3.2,-2) & (3.2,-2)}?(1)*\dir{>};
    (4,-2)*{\scs  \l};
    (0,-5)*{\bigotimes};
    (-6,0)*{};(6,0)*{};
    \endxy}
\end{align*}
equating coefficients of the corresponding terms implies
\[
-b_{\l} = 2a_{\l-2} + \alpha_{1,\l} -b_{\l} + b_{\l-2} +(1-\l)\alpha_2 + (-1)^{\l}d_{\l-2}
\]
 or
\[
(-1)^{\l+1}d_{\l-2}  = 2a_{\l-2} + \alpha_{1,\l}  + b_{\l-2} +(1-\l)\alpha_2 \qquad \text{for all $\l \geq 0$.}
\]
Likewise, the $\l\leq 0$ case implies
\[
(-1)^{\l}d_{\l} = (\l+1)\alpha_2 -b_{\l}   -2a_{\l} - \alpha_{1,\l+2}  \qquad \text{for all $\l\leq 0$. }
\]
Hence, \eqref{eq:alpha1} must hold for all values of $\l$.
Then combining \eqref{eq:alpha1} with \eqref{eq:dlambda0} implies
\[
\alpha_1:= \alpha_{1,\l-2} = \alpha_{1,\l}
\]
which together with \eqref{eqn:d of odd nilhecke2} implies
\[
\beta_{1,\l} = \beta_{1,\l-2} =  \alpha_{1}.
\]
\end{proof}

\subsection{Derivations and odd $sl(2)$ relations}

\begin{lemma} \label{lem:EFr2}
The map $\partial$ defined in Lemma~\ref{lem:derivation-start} with the constraints from \eqref{eq:alpha1} satisfies the following identities:
\begin{align*}
\partial \left( \;\; \vcenter{\xy 0;/r.18pc/:
     (-4,-4)*{};(4,4)*{} **\crv{(-4,-1) & (4,1)}?(1)*\dir{>};
     (4,-4)*{};(-4,4)*{} **\crv{(4,-1) & (-4,1)}?(0)*\dir{<};
     (-4,4)*{};(4,12)*{} **\crv{(-4,7) & (4,9)}?(0)*\dir{<};
     (4,4)*{};(-4,12)*{} **\crv{(4,7) & (-4,9)}?(1)*\dir{>};
     (7,8)*{\scs \l};
  \endxy} \;\; \right)
  \;\; & = \;\;
   (-a_{\l-2} - \alpha_1) \;\;
\vcenter{\xy 
    (-3,6)*{};(3,-2)*{} **\crv{(-3,3) & (3,1)}?(0)*\dir{<};
    (3,6)*{};(-3,-2)*{} **\crv{(3,3) & (-3,1)};
    (-3,-2)*{};(3,-2)*{} **\crv{(-3.2,-6) & (3.2,-6)}?(.2)*\dir{};
    (-3,-10)*{};(3,-10)*{} **\crv{(-3.2,-4) & (3.2,-4)}?(1)*\dir{>} ?(.2)*\dir{};
    (5,0)*{\scs  \l};
    \endxy}
\;\; - \;\;
(a_{\l-2} +\alpha_1\delta_{\l,\text{even}})\;\;
\vcenter{\xy 
    (-3,8)*{};(3,8)*{} **\crv{(-3.2,2) & (3.2,2)}?(0)*\dir{<};
    (-3,-1)*{};(3,-9)*{} **\crv{(-3,-4) & (3,-6)}?(1)*\dir{>};
    (3,-1)*{};(-3,-9)*{} **\crv{(3,-4) & (-3,-6)};
    (-3,-1)*{};(3,-1)*{} **\crv{(-3.2,4) & (3.2,4)}?(.2)*\dir{};
    (5,0)*{\scs  \l};
    \endxy}
    \\
\partial \left( \;\; \vcenter{\xy 0;/r.18pc/:
     (-4,-4)*{};(4,4)*{} **\crv{(-4,-1) & (4,1)}?(0)*\dir{<};
     (4,-4)*{};(-4,4)*{} **\crv{(4,-1) & (-4,1)}?(1)*\dir{>};
     (-4,4)*{};(4,12)*{} **\crv{(-4,7) & (4,9)}?(1)*\dir{>};
     (4,4)*{};(-4,12)*{} **\crv{(4,7) & (-4,9)}?(0)*\dir{<};
     (7,8)*{\scs \l};
  \endxy} \;\; \right)
  \;\; & = \;\;
  (a_{\l} + \alpha_1\delta_{\l,\text{even}})\;\;
\vcenter{\xy 
    (-3,6)*{};(3,-2)*{} **\crv{(-3,3) & (3,1)};
    (3,6)*{};(-3,-2)*{} **\crv{(3,3) & (-3,1)}?(0)*\dir{<};
    (-3,-2)*{};(3,-2)*{} **\crv{(-3.2,-6) & (3.2,-6)}?(.2)*\dir{};
    (-3,-10)*{};(3,-10)*{} **\crv{(-3.2,-4) & (3.2,-4)}?(0)*\dir{<} ?(.2)*\dir{};
    (5,0)*{\scs  \l};
    \endxy}
\;\; + \;\;
(a_{\l} +\alpha_1)\;\;
\vcenter{\xy 
    (-3,8)*{};(3,8)*{} **\crv{(-3.2,2) & (3.2,2)}?(1)*\dir{>};
    (-3,-1)*{};(3,-9)*{} **\crv{(-3,-4) & (3,-6)};
    (3,-1)*{};(-3,-9)*{} **\crv{(3,-4) & (-3,-6)}?(1)*\dir{>};
    (-3,-1)*{};(3,-1)*{} **\crv{(-3.2,4) & (3.2,4)}?(.2)*\dir{};
    (5,0)*{\scs  \l};
    \endxy}
\end{align*}
\end{lemma}

\begin{proof}
The sideways crossings take the form
\begin{equation} \label{eq:partial-leftcrossing}
\begin{split}
\partial \left(\;\;
\xy 0;/r.18pc/:
  (0,0)*{\xybox{
    (-4,-4)*{};(4,4)*{} **\crv{(-4,-1) & (4,1)}?(1)*\dir{>} ;
    (4,-4)*{};(-4,4)*{} **\crv{(4,-1) & (-4,1)}?(0)*\dir{<};
    (9,2)*{\scs \lambda};
     }};
  \endxy \right )
&=
(a_{\l}-a_{\l-2}) \;\;
\xy 0;/r.18pc/:
  (0,0)*{\xybox{
    (-4,-4)*{};(4,4)*{} **\crv{(-4,-1) & (4,1)}?(1)*\dir{>} ;
    (4,-4)*{};(-4,4)*{} **\crv{(4,-1) & (-4,1)}?(0)*\dir{<}?(0.25)*{\bullet};
     (9,2)*{\scs \lambda};
     }};
  \endxy
  \; + \;
(b_{\l}-b_{\l-2} -\alpha_2) \;\;
 \xy 0;/r.18pc/:
  (0,0)*{\xybox{
    (-4,-4)*{};(4,4)*{} **\crv{(-4,-1) & (4,1)}?(1)*\dir{>};
    (4,-4)*{};(-4,4)*{} **\crv{(4,-1) & (-4,1)}?(0)*\dir{<};
     (9,2)*{\scs \l};
     (8,-4)*{\bigotimes};
     }};
  \endxy
  \; +\;
  (\alpha_1 + a_{\l})\; \;
  \vcenter{\xy 
    (-3,6)*{};(3,6)*{} **\crv{(-3.2,0) & (3.2,0)}?(1)*\dir{>};
    (-3,-4)*{};(3,-4)*{} **\crv{(-3.2,2) & (3.2,2)}?(1)*\dir{>};
    (6,2)*{\scs  \l};
  \endxy }
\\
\partial \left (\;\;
\xy 0;/r.18pc/:
  (0,0)*{\xybox{
    (-4,-4)*{};(4,4)*{} **\crv{(-4,-1) & (4,1)}?(0)*\dir{<} ;
    (4,-4)*{};(-4,4)*{} **\crv{(4,-1) & (-4,1)}?(1)*\dir{>};
     (8,2)*{\scs \lambda};
     }};
  \endxy \right )
 &=
  (-a_{\l} +a_{\l-2}) \;\;
 \xy 0;/r.18pc/:
  (0,0)*{\xybox{
    (-4,-4)*{};(4,4)*{} **\crv{(-4,-1) & (4,1)}?(0)*\dir{<}?(0.25)*{\bullet} ;
    (4,-4)*{};(-4,4)*{} **\crv{(4,-1) & (-4,1)}?(1)*\dir{>};
     (7,2)*{\scs \lambda};
     }};
  \endxy
  -
(b_{\l} - b_{\l-2} - \alpha_2) \;\;
 \xy 0;/r.18pc/:
  (0,0)*{\xybox{
    (-4,-4)*{};(4,4)*{} **\crv{(-4,-1) & (4,1)}?(0)*\dir{<};
    (4,-4)*{};(-4,4)*{} **\crv{(4,-1) & (-4,1)}?(1)*\dir{>};
     (9,2)*{\scs \l};
     (8,-4)*{\bigotimes};
     }};
  \endxy
 \; +
  (-a_{\l-2} -\alpha_1\delta_{\l,\text{even}}) \;\;
  \vcenter{\xy 
    (-3,6)*{};(3,6)*{} **\crv{(-3.2,0) & (3.2,0)}?(0)*\dir{<};
    (-3,-4)*{};(3,-4)*{} **\crv{(-3.2,2) & (3.2,2)}?(0)*\dir{<};
    (6,2)*{\scs  \l};
  \endxy }
\end{split}
\end{equation}
and the result follows by direct computation.
\end{proof}

\begin{lemma}
The map $\partial$ defined in Lemma~\ref{lem:derivation-start} with the constraints from \eqref{eq:alpha1} preserves the odd $\mf{sl}(2)$ relations \eqref{rei2} without any additional constraints.
\end{lemma}

\begin{proof}
We prove the first relation in \eqref{rei2}.  The second can be proven similarly.  First we compute
\begin{align*}
&\partial \left(
\sum_{\overset{r+n+k}{=\l-1}}
(-1)^{k}\;
\vcenter{\xy
 (3,9)*{};(-3,9)*{} **\crv{(3,4) & (-3,4)} ?(1)*\dir{>}?(.8)*{\bullet};
 (-4,6)*{\scriptstyle r};
 (0,0)*{\smccbub{}};
 (2.4,0)*{\bullet};
 (6,0)*{\scriptstyle *+k};
 (-3,-9)*{};(3,-9)*{} **\crv{(-3,-4) & (3,-4)} ?(1)*\dir{>}?(.2)*{\bullet};
 (-4,-6)*{\scriptstyle n};
 (6,4)*{\scs \l};
 (-8,0)*{};(8,0)*{};
\endxy} \right)
= \;\;
\sum_{\overset{r'+n+k}{\overset{=\l}{r'\geq 1}}}
(-1)^{n}\;
(a_{\l-2}+  \alpha_{1}\delta_{\l+r,\text{even}} )
\vcenter{\xy
 (3,9)*{};(-3,9)*{} **\crv{(3,4) & (-3,4)} ?(1)*\dir{>}?(.8)*{\bullet};
 (-6,6)*{\scriptstyle r'};
 (0,0)*{\smccbub{}};
 (2.4,0)*{\bullet};
 (6,0)*{\scriptstyle *+k};
 (-3,-9)*{};(3,-9)*{} **\crv{(-3,-4) & (3,-4)} ?(1)*\dir{>}?(.2)*{\bullet};
 (-4,-6)*{\scriptstyle n};
 (6,4)*{\scs \l};
 (-8,0)*{};(8,0)*{};
\endxy} \\ &
\quad
\;\; + \;\;
\sum_{\overset{r+n'+k}{\overset{=\l}{\scs n' \geq 1} }}
(-1)^{n'+1+k}\;
(a_{\l-2} + \alpha_1\delta_{n',\text{even}})
\vcenter{\xy
 (3,9)*{};(-3,9)*{} **\crv{(3,4) & (-3,4)} ?(1)*\dir{>}?(.8)*{\bullet};
 (-5,6)*{\scriptstyle r};
 (0,0)*{\smccbub{}};
 (2.4,0)*{\bullet};
 (6,0)*{\scriptstyle *+k};
 (-3,-9)*{};(3,-9)*{} **\crv{(-3,-4) & (3,-4)} ?(1)*\dir{>}?(.2)*{\bullet};
 (-6,-6)*{\scriptstyle n'};
 (6,4)*{\scs \l};
 (-8,0)*{};(8,0)*{};
\endxy}
\;\; + \;\;\sum_{\overset{r+n+k'}{=\l}}
(-1)^{n} \delta_{k',\text{odd}}\;
(   -2a_{\l-2}  -\alpha_{1} )
\vcenter{\xy
 (3,9)*{};(-3,9)*{} **\crv{(3,4) & (-3,4)} ?(1)*\dir{>}?(.8)*{\bullet};
 (-6,6)*{\scriptstyle r };
 (0,0)*{\smccbub{}};
 (2.4,0)*{\bullet};
 (7,0)*{\scriptstyle *+k'};
 (-3,-9)*{};(3,-9)*{} **\crv{(-3,-4) & (3,-4)} ?(1)*\dir{>}?(.2)*{\bullet};
 (-4,-6)*{\scriptstyle n};
 (6,4)*{\scs \l};
 (-8,0)*{};(8,0)*{};
\endxy}
\end{align*}
After simplifying this reduces to

\begin{align*}
& \qquad = - (a_{\l-2}+  \alpha_{1}\delta_{\l ,\text{even}} )\sum_{\overset{n+k}{=\l}}
 (-1)^{n}\;
\vcenter{\xy
 (3,9)*{};(-3,9)*{} **\crv{(3,4) & (-3,4)} ?(1)*\dir{>} ;
 (-6,6)*{\scriptstyle  };
 (0,0)*{\smccbub{}};
 (2.4,0)*{\bullet};
 (6,0)*{\scriptstyle *+k};
 (-3,-9)*{};(3,-9)*{} **\crv{(-3,-4) & (3,-4)} ?(1)*\dir{>}?(.2)*{\bullet};
 (-4,-6)*{\scriptstyle n};
 (6,4)*{\scs \l};
 (-8,0)*{};(8,0)*{};
\endxy}
\;\; - \;\;(a_{\l-2} + \alpha_1 )
\sum_{\overset{r +k}{=\l}}
(-1)^{ 1+k}\;
\vcenter{\xy
 (3,9)*{};(-3,9)*{} **\crv{(3,4) & (-3,4)} ?(1)*\dir{>}?(.8)*{\bullet};
 (-5,6)*{\scriptstyle r};
 (0,0)*{\smccbub{}};
 (2.4,0)*{\bullet};
 (6,0)*{\scriptstyle *+k};
 (-3,-9)*{};(3,-9)*{} **\crv{(-3,-4) & (3,-4)} ?(1)*\dir{>} ;
 (-6,-6)*{\scriptstyle  };
 (6,4)*{\scs \l};
 (-8,0)*{};(8,0)*{};
\endxy}
\end{align*}
The claim follow using Lemma~\ref{lem:EFr2} and the curl relations \eqref{eq:downcurl1} and \eqref{eq:downcurl2}.
\end{proof}

\subsection{Classification of derivations}
We summarize our results up to this point in the following:

\begin{proposition}\label{prop:derivation_on_oUcat}
There is a bidegree $(2,\bar{1})$ derivation $\partial$   of the odd 2-category $\oUcat^{\text{even}}$ or  $\oUcat^{\text{odd}}$ defined on generating 2-morphisms:
\begin{alignat}{2}
%
&\partial \left( \;\;
 \xy 0;/r.17pc/:
 (0,6);(0,-6); **\dir{-} ?(1)*\dir{>}+(2.3,0)*{\scriptstyle{}}
 ?(.1)*\dir{ }+(2,0)*{\scs };
 (0,0)*{\txt\large{$\bullet$}};
 (6,3)*{ \scs \l};
 \endxy \;\;
\right)
\; = \;
\alpha_{1 } \;\;
\xy 0;/r.17pc/:
 (0,6);(0,-6); **\dir{-} ?(1)*\dir{>}+(2.3,0)*{\scriptstyle{}}
 ?(.1)*\dir{ }+(2,0)*{\scs };
 (0,0)*{\txt\large{$\bullet$}};
 (8,5)*{\scs \l};
 (4,0)*{2};
 \endxy
 \; + \;
 \alpha_2 \;\;
 \xy 0;/r.17pc/:
 (0,6);(0,-6); **\dir{-} ?(1)*\dir{>}+(2.3,0)*{\scriptstyle{}}
 ?(.1)*\dir{ }+(2,0)*{\scs };
 (0,1)*{\txt\large{$\bullet$}};
 (8,5)*{\scs \l};
 (4,-5.5)*{\bigotimes};
 \endxy
\qquad && \partial \left(
 \xy 0;/r.19pc/:
  (0,0)*{\xybox{
    (-4,-4)*{};(4,4)*{} **\crv{(-4,-1) & (4,1)}?(1)*\dir{>}?(.25)*{};
    (4,-4)*{};(-4,4)*{} **\crv{(4,-1) & (-4,1)}?(1)*\dir{>};
     (7,1)*{\scs \l};
     (-8,0)*{};(9,0)*{};
     }};
  \endxy
	\right) = \;\;
 \alpha_1  \;\;
 \xy 0;/r.19pc/:
  (3,4);(3,-4) **\dir{-}?(0)*\dir{<}+(2.3,0)*{};
  (-3,4);(-3,-4) **\dir{-}?(0)*\dir{<}+(2.3,0)*{};
  (8,2)*{\scs \l};
 \endxy
  \;\; - \;\;
  \alpha_2 \;
  \xy 0;/r.19pc/:
  (0,0)*{\xybox{
    (-4,-4)*{};(4,4)*{} **\crv{(-4,-1) & (4,1)}?(1)*\dir{>}?(.25)*{};
    (4,-4)*{};(-4,4)*{} **\crv{(4,-1) & (-4,1)}?(1)*\dir{>};
     (8,2)*{ \scs \l};
     (8,-3)*{\bigotimes};
     }};
  \endxy
\\
& \partial \left(\;
  \vcenter{\xy 
    (-3,-4)*{};(3,-4)*{} **\crv{(-3.2,2) & (3.2,2)}?(1)*\dir{>};
    (4,2)*{\scs  \l};
  \endxy } \;
    \right)  \; = \;
a_{\l -2}
\vcenter{\xy 
    (-3,-4)*{};(3,-4)*{} **\crv{(-3.2,2) & (3.2,2)}?(1)*\dir{>} ?(.2)*\dir{}+(0,0)*{\bullet};
    (4,2)*{\scs  \l};
  \endxy }
\;  + \;
b_{\l -2} \;
\vcenter{\xy 
    (-3,-4)*{};(3,-4)*{} **\crv{(-3.2,2) & (3.2,2)}?(1)*\dir{>} ;
    (6,2)*{\scs  \l};
    (0,4)*{\bigotimes};
    (-0,10)*{};
  \endxy }
\qquad
&& \partial
\left(\; \label{def:d-right cup}
  \vcenter{\xy 
    (-3,4)*{};(3,4)*{} **\crv{(-3.2,-2) & (3.2,-2)}?(1)*\dir{>};
    (4,-2)*{\scs  \l};
  \endxy } \;
    \right)  \; = \;
-a_{\l}\;
\vcenter{\xy 
    (-3,4)*{};(3,4)*{} **\crv{(-3.2,-2) & (3.2,-2)}?(1)*\dir{>} ?(.2)*\dir{}+(5,0)*{\bullet};
    (4,-2)*{\scs  \l};
  \endxy }
\;  - \;
b_{\l}\;
\vcenter{\xy 
    (-6,4)*{};(6,4)*{} **\crv{(-5.5,-6) & (5.5,-6)}?(1)*\dir{>} ;
    (6,-2)*{\scs  \l};
    (0,2)*{\bigotimes};
  \endxy }
 \\
& \partial \left(\;\;
  \vcenter{\; \xy 
    (-3,-4)*{};(3,-4)*{} **\crv{(-3.2,2) & (3.2,2)}?(0)*\dir{<};
    (4,2)*{\scs  \l};
  \endxy } \;
    \right) \; = \;
c_{\l} \;
\vcenter{\xy 
    (3,-4)*{};(-3,-4)*{} **\crv{(3.2,2) & (-3.2,2)}?(1)*\dir{>} ?(.2)*\dir{}+(0,0)*{\bullet};
    (4,2)*{\scs  \l};
  \endxy }
 \; + \;
d_{\l} \;
\vcenter{\xy 
    (6,-4)*{};(-6,-4)*{} **\crv{(5.5,6) & (-5.5,6)}?(1)*\dir{>} ;
    (6,2)*{\scs  \l};
    (0,-2)*{\bigotimes};
  \endxy }
\qquad
&& \partial \left(\;\;  \label{def:d-left cup}
  \vcenter{\xy 
    (-3,4)*{};(3,4)*{} **\crv{(-3.2,-2) & (3.2,-2)}?(0)*\dir{<};
    (4,-2)*{\scs  \l};
  \endxy }\;
    \right)  \; = \;
(-1)^{\l}c_{\l-2} \;
\vcenter{\xy 
    (-3,4)*{};(3,4)*{} **\crv{(-3.2,-2) & (3.2,-2)}?(0)*\dir{<} ?(.2)*\dir{}+(0,0)*{\bullet};
    (4,-2)*{\scs  \l};
  \endxy }
\; - \;
d_{\l-2}\;
\vcenter{\xy 
    (-3,4)*{};(3,4)*{} **\crv{(-3.2,-2) & (3.2,-2)}?(0)*\dir{<} ;
    (4,-2)*{\scs  \l};
    (0,-5)*{\bigotimes};
  \endxy }
\end{alignat}
with relations
\begin{align}
& c_{\l} = -a_{\l} -\delta_{\l,\text{odd}}\alpha_{1 } \\
& d_{\l} = (-1)^{\l+1}(2a_{\l} +\alpha_{1} +b_{\l} - (\l+1)\alpha_2).
\end{align}
Furthermore, assuming the weak nondegeneracy conjecture from section~\ref{sec:nondegen} this is the most general bidegree $(2,\bar{1})$ derivation on $\oUcat$.
\end{proposition}

%
\section{Differentials and fantastic filtrations }
%

\subsection{Classification of differentials}
 \begin{proposition}\label{prop:differential_on_oUcat}
Given parameters $\alpha_1,a_{\l},b_{\l},c_{\l}$ and $d_{\l}$ with $\l\in \Z$, there is a bidegree $(2,\bar{1})$ differential $\partial$  (i.e. $\partial^2=0$) on  the odd 2-category $\oUcat^{\text{even}}$ or $\l$ in $\oUcat^{\text{odd}}$ defined on generating 2-morphisms:
\begin{alignat}{2}
%
&\partial \left( \;\;
 \xy 0;/r.17pc/:
 (0,6);(0,-6); **\dir{-} ?(1)*\dir{>}+(2.3,0)*{\scriptstyle{}}
 ?(.1)*\dir{ }+(2,0)*{\scs };
 (0,0)*{\txt\large{$\bullet$}};
 (8,3)*{ \scs \l};
 \endxy \;
\right)
\;  = \;
\alpha_{1 } \;\;
\xy 0;/r.17pc/:
 (0,6);(0,-6); **\dir{-} ?(1)*\dir{>}+(2.3,0)*{\scriptstyle{}}
 ?(.1)*\dir{ }+(2,0)*{\scs };
 (0,0)*{ \bullet };
 (8,3)*{\scs \l};
 (4,0)*{\scs 2};
 \endxy
 \qquad \quad
&& \partial \left( \label{def:d_crossing}
 \xy 0;/r.19pc/:
  (0,0)*{\xybox{
    (-4,-4)*{};(4,4)*{} **\crv{(-4,-1) & (4,1)}?(1)*\dir{>}?(.25)*{};
    (4,-4)*{};(-4,4)*{} **\crv{(4,-1) & (-4,1)}?(1)*\dir{>};
     (7,1)*{\scs \l};
     (-8,0)*{};(9,0)*{};
     }};
  \endxy
	\right) = \;\;
 \alpha_1  \;\;
 \xy 0;/r.19pc/:
  (3,4);(3,-4) **\dir{-}?(0)*\dir{<}+(2.3,0)*{};
  (-3,4);(-3,-4) **\dir{-}?(0)*\dir{<}+(2.3,0)*{};
  (8,2)*{\scs \l};
 \endxy
 \\ \smallskip
& \partial \left(  \;
  \vcenter{\xy 
    (-3,-4)*{};(3,-4)*{} **\crv{(-3.2,2) & (3.2,2)}?(1)*\dir{>};
    (4,2)*{\scs  \l};
  \endxy }
   \; \right) \; = \;
a_{\l -2} \;
\vcenter{\xy 
    (-3,-4)*{};(3,-4)*{} **\crv{(-3.2,2) & (3.2,2)}?(1)*\dir{>} ?(.2)*\dir{}+(0,0)*{\bullet};
    (4,2)*{\scs  \l};
  \endxy }
\;  + \;
b_{\l -2}\;
\vcenter{\xy 
    (-3,-4)*{};(3,-4)*{} **\crv{(-3.2,2) & (3.2,2)}?(1)*\dir{>} ;
    (6,2)*{\scs  \l};
    (0,5)*{\bigotimes};
  \endxy }
\qquad
&& \partial\left(\;\; 
  \vcenter{\xy 
    (-3,4)*{};(3,4)*{} **\crv{(-3.2,-2) & (3.2,-2)}?(1)*\dir{>};
    (4,-2)*{\scs  \l};
  \endxy } \;
    \right)  \; =\;
-a_{\l} \;
\vcenter{\xy 
    (-3,4)*{};(3,4)*{} **\crv{(-3.2,-2) & (3.2,-2)}?(1)*\dir{>} ?(.2)*\dir{}+(5,0)*{\bullet};
    (4,-2)*{\scs  \l};
  \endxy }
\;  - \;
b_{\l} \;
\vcenter{\xy 
    (-6,4)*{};(6,4)*{} **\crv{(-5.5,-6) & (5.5,-6)}?(1)*\dir{>} ;
    (6,-2)*{\scs  \l};
    (0,2)*{\bigotimes};
  \endxy }
 \\  \bigskip
& \partial \left( \; \label{def:d-left cap}
  \vcenter{\xy 
    (-3,-4)*{};(3,-4)*{} **\crv{(-3.2,2) & (3.2,2)}?(0)*\dir{<};
    (4,2)*{\scs  \l};
  \endxy } \;
    \right) \; = \;
c_{\l}\;
\vcenter{\xy 
    (3,-4)*{};(-3,-4)*{} **\crv{(3.2,2) & (-3.2,2)}?(1)*\dir{>} ?(.2)*\dir{}+(0,0)*{\bullet};
    (4,2)*{\scs  \l};
  \endxy }
\;  + \;
d_{\l} \;
\vcenter{\xy 
    (6,-4)*{};(-6,-4)*{} **\crv{(5.5,6) & (-5.5,6)}?(1)*\dir{>} ;
    (6,2)*{\scs  \l};
    (0,-2)*{\bigotimes};
  \endxy }
\qquad &&
\partial \left(\; \; 
  \vcenter{\xy 
    (-3,4)*{};(3,4)*{} **\crv{(-3.2,-2) & (3.2,-2)}?(0)*\dir{<};
    (4,-2)*{\scs  \l};
  \endxy } \;
    \right) \;=\;
(-1)^{\l}c_{\l-2}\;
\vcenter{\xy 
    (-3,4)*{};(3,4)*{} **\crv{(-3.2,-2) & (3.2,-2)}?(0)*\dir{<} ?(.2)*\dir{}+(0,0)*{\bullet};
    (4,-2)*{\scs  \l};
  \endxy }
\; - \;
d_{\l-2}\;
\vcenter{\xy 
    (-3,4)*{};(3,4)*{} **\crv{(-3.2,-2) & (3.2,-2)}?(0)*\dir{<} ;
    (6,-2)*{\scs  \l};
    (0,-5)*{\bigotimes};
  \endxy }
\end{alignat}
with relations
\begin{align} \label{eq:clambda}
& c_{\l} = -a_{\l} -\delta_{\l,\text{odd}} \alpha_{1 }\\
& d_{\l} = (-1)^{\l+1}(2a_{\l} +\alpha_{1} +b_{\l}) \label{eq:dlambda} \\
& a_{\l}(a_{\l} + \alpha_1) = 0.
\end{align}
Furthermore, assuming the weak nondegeneracy conjecture from section~\ref{sec:nondegen}, this is the most general bidegree $(2,\bar{1})$ differential on $\oUcat$.
\end{proposition}

\begin{proof}
We compute $\partial^2$ of each generating 2-morphism from the general derivation in Proposition~\ref{prop:derivation_on_oUcat} and set the resulting equation equal to zero.  This produces the equations
\begin{equation}
\begin{split}
  &\alpha_2(2+\alpha_1) = 0 \\
  &\alpha_1\alpha_2 = 0
\end{split}\qquad
\begin{split}
  & a_{\l}(a_{\l} + \alpha_1) = 0 \\
  & a_{\l}\alpha_2 = 0
\end{split}
\qquad
\begin{split}
& (a_{\l} + \alpha_1\delta_{\l,\text{odd}})(a_{\l} + \alpha_1\delta_{\l,\text{even}}) = 0 \\
& \alpha_2(a_{\l} + \alpha_1\delta_{\l,\text{odd}}) = 0.
\end{split}
\end{equation}
Hence, for $\partial^2=0$ we must have $\alpha_2=0$ and $a_{\l}(a_{\l} + \alpha_1)=0$.
\end{proof}

Note that Lemmas \ref{lem:odd-bub} and \ref{lem:even-bub} imply that the differential kills all dotted bubbles:
\begin{align*}
& \partial \left(
 \xy 0;/r.17pc/:
 (6,5)*{\scs \l};
 (-10,0)*{};(10,0)*{};(-2,-8)*{\scs };
 (0,0)*{\lbub{\scs n}};
 \endxy
 \right)
 \;\; = \;\;
  \partial \left(
 \xy 0;/r.17pc/:
 (6,5)*{\scs \l};
 (-10,0)*{};(10,0)*{};(-2,-8)*{\scs };
 (0,0)*{\rbub{\scs n}};
 \endxy
 \right)
 \;\; = 0
\end{align*}
for all $\l\in \Z$ and $n\geq 0$.


\subsection{Fantastic filtrations on $\cal{E}\cal{F}$ and $\cal{F}\cal{E}$} \label{sec:fantastic-EF-FE}

In this section we show that the odd $sl(2)$-isomorphisms \eqref{rei2} give rise to differentials on $\oUcat$ providing fantastic filtrations for $\cal{E}\cal{F}\onel$ and $\cal{F}\cal{E}\onel$. We refer the reader to Section~\ref{sec:fantastic} for the preliminaries on the Fantastic filtration.

For each $\l \in \Z$ define $I=\{0, 1, \dots, |\lambda|\}$.  We define data $\{u_i, v_i\}_{i \in I}$ giving rise to an idempotent factorization determined by the odd $sl(2)$-relation.
We begin with case $\l\geq 0$ corresponding to the first relation in \eqref{rei2}.  Recall the family of 2-categorical differentials defined in Proposition~\ref{prop:differential_on_oUcat}.

Consider the set of 1-morphisms
\[
 \mathbb{X}_{\l} := \{
 \cal{E}\cal{F}\onel, \cal{F}\cal{E}\onel, \onel 
\},
\]
and its endomorphism dg-superalgebra $R = \End_{\oUcat}(\mathbb{X}_{\l})$.
We provide a fantastic filtration of the representable dg-supermodule associated to $\cal{E}\cal{F}\onel$.
Here our investigation departs from \cite{EQ2} in that the most natural filtration
\begin{alignat}{2} \label{eq:EF-fant}
&u_n := \;\;
\sum_{r\geq 0} (-1)^{(\l +n+r+1)}
   \vcenter{\xy 
    (-3,4)*{};(3,4)*{} **\crv{(-3.2,-2) & (3.2,-2)}?(0)*\dir{<} ?(.2)*\dir{}+(0,0)*{\bullet};
    (-4,0)*{\scs r};
    (5,0)*{\scs  \l};
    (0,-5)*{\smccbub{}}+(2.5,0)*{\bullet};
    (9,-5)*{\scs -n-r-2};
    (-8,0)*{};(8,0)*{};
  \endxy }
  \quad   (0\leq n \leq \l-1),
\qquad \quad &&u_{\l} := \;\;
\xy 0;/r.17pc/:
  (0,0)*{\xybox{
    (-4,4)*{};(4,-4)*{} **\crv{(-4,1) & (4,-1)}?(0)*\dir{<} ;
    (4,4)*{};(-4,-4)*{} **\crv{(4,1) & (-4,-1)}?(1)*\dir{>};
     (8,1)*{\scs \l};
     }};
  \endxy
\\ & && \nn\\ \nn
 &v_n :=
\vcenter{\xy 
    (-3,-4)*{};(3,-4)*{} **\crv{(-3.2,2) & (3.2,2)}?(1)*\dir{>} ?(.2)*\dir{}+(0,0)*{\bullet};
    (4,2)*{\scs  \l};
    (-5,-1)*{\scs  n};
    (-8,0)*{};(8,0)*{};
  \endxy }
  \quad   (0\leq n \leq \l-1),
\qquad \quad  &&v_{\l} := \;\;-
\xy 0;/r.17pc/:
  (0,0)*{\xybox{
    (-4,4)*{};(4,-4)*{} **\crv{(-4,1) & (4,-1)}?(1)*\dir{>} ;
    (4,4)*{};(-4,-4)*{} **\crv{(4,1) & (-4,-1)}?(0)*\dir{<};
     (8,1)*{\scs \l};
     }};
  \endxy
\end{alignat}\\
on the morphism $\cal{E}\cal{F}\onel$ leads to a trivial differential when we impose the fantastic filtration condition
\begin{equation}\label{eqn:fc.filtration_requirement}
v_i \partial(u_j) =0 \text{,  for $i \leq j$.}
\end{equation}
In Definition~\ref{def:modorder} we define an order $\prec$ on $I$ for which the maps in \eqref{eq:EF-fant} give rise to fantastic filtrations.

\medskip

The conditions on $\{u_i,v_i\}$ in \eqref{eq:uv-equations} follow immediately from the axioms of $\oUcat$  using \eqref{rei2} , \eqref{rel:Curls} and \eqref{rel:odd_inf_grassm}, see for example \cite[Equations (5.13) and (5.14)]{BE2}.
We check $v_{i}\partial(u_{j}) = 0$ for $0 \leq i \leq j \leq \l-1$.
\begin{align*}
&0 = v_{i}\partial(u_j) = \\
&
 \sum_{r \geq 0} (-1)^{\l + j + r + 1} \left(
(\alpha_1\delta_{r,\text{odd}} + (-1)^{r+\l}c_{\l-2})\hspace{-0.2in}
\vcenter{\xy 
    (5,3)*{\smcbub{}}+(-2.5,0)*{\bullet};
    (-8,3)*{\scs (r+i+2-\l)+*};
    (10,0)*{\scs  \l};
    (1,-5)*{\smccbub{}}+(2.5,0)*{\bullet};
    (14,-5)*{\scs *+(\l-j-r-1) };
    (-8,0)*{};(8,0)*{};
  \endxy }
  +
(-1)^{r+1}d_{\l-2} \hspace{-0.2in}
\vcenter{\xy 
    (5,3)*{\smcbub{}}+(-2.5,0)*{\bullet};
    (-8,3)*{\scs (r+i+1-\l)+*};
    (10,0)*{\scs  \l};
    (5,-2.5)*{\bigotimes};
    (1,-9)*{\smccbub{}}+(2.5,0)*{\bullet};
    (14,-8)*{\scs (\l-j-r-1)+*};
  \endxy } \right )
\\
& =
  \sum_{r' = \max(0, i-\l+2)}^{i-j+1} (-1)^{i + j + r' - 1}
(\alpha_1\delta_{r'-\l+i,\text{odd}} + (-1)^{r'+i}c_{\l-2})
\vcenter{\xy 
    (5,3)*{\smcbub{}}+(-2.5,0)*{\bullet};
    (-2,3)*{\scs r'+*};
    (10,0)*{\scs  \l};
    (5,-5)*{\smccbub{}}+(2.5,0)*{\bullet};
    (18,-5)*{\scs *+(i-j+1-r') };
    (-8,0)*{};(8,0)*{};
  \endxy }
\\ \nn \\
& \qquad \; +  \;
  \sum_{r' = 0}^{i-j} (-1)^{\l+j}
 d_{\l-2}
\vcenter{\xy 
    (5,3)*{\smcbub{}}+(-2.5,0)*{\bullet};
    (-2,3)*{\scs r'+*};
    (10,0)*{\scs  \l};
    (5,-2.5)*{\bigotimes};
    (5,-9)*{\smccbub{}}+(2.5,0)*{\bullet};
    (18,-8)*{\scs i-j-r'+*};
    (-8,0)*{};(8,0)*{};
  \endxy }
\end{align*}
where we set $r'=r-\l+2+i$ in the first sum and $r'=r-\l+1+i$ in the second.  Note that only the even bubbles are nonzero in the second sum by \eqref{eq:oddbub-square}, so that by \eqref{rel:odd_inf_grassm} this term simplifies
\begin{align}
& =
 \sum_{r = \max(0, i-\l+2)}^{i-j+1} (-1)^{i + j + r - 1}
(\alpha_1\delta_{r-\l+i,\text{odd}} + (-1)^{r+i}c_{\l-2})\hspace{-0.2in}
\vcenter{\xy 
    (5,3)*{\smcbub{}}+(-2.5,0)*{\bullet};
    (-2,3)*{\scs r+*};
    (10,0)*{\scs  \l};
    (5,-5)*{\smccbub{}}+(2.5,0)*{\bullet};
    (18,-5)*{\scs *+(i-j+1-r) };
    (-8,0)*{};(8,0)*{};
  \endxy }
  +
  \delta_{i,j} (-1)^{\l+j}
 d_{\l-2}\;\;
\vcenter{\xy 
    (10,0)*{\scs  \l};
    (5,-2.5)*{\bigotimes};
  \endxy } \nn
  \\
& =
 \sum_{r = \max(0, i-\l+2)}^{i-j+1} \left( (-1)^{i +j + r+1}  \alpha_1\delta_{i+r,\text{odd}}
 +(-1)^{ j}a_{\l-2} \right)
\vcenter{\xy 
    (5,3)*{\smcbub{}}+(-2.5,0)*{\bullet};
    (-2,3)*{\scs r+*};
    (10,0)*{\scs  \l};
    (5,-5)*{\smccbub{}}+(2.5,0)*{\bullet};
    (18,-5)*{\scs *+(i-j+1-r) };
    (-8,0)*{};(8,0)*{};
  \endxy }
\label{eq:simple-fc} \\\nn
  & \qquad \quad
\; \hspace{1.3cm} + \;
  \delta_{i,j}(-1)^{j-1}
(2a_{\l-2} +\alpha_{1} +b_{\l-2})\;\;
\vcenter{\xy 
    (10,0)*{\scs  \l};
    (5,-2.5)*{\bigotimes};
  \endxy }
\end{align}
where we used \eqref{eq:clambda} and \eqref{eq:dlambda} to eliminate $c_{\l-2}$ and $d_{\l-2}$.

If we are interested in the case when $i \leq j$ then this equation only provides constraints when $i=j$ and when $j=i+1$.   At $i=j$ we get
\[
  \begin{array}{cl}
     - \alpha_1\delta_{i,\text{odd}}
 +(-1)^{ i}a_{\l-2}   +  \alpha_1\delta_{i,\text{even}}
 +(-1)^{ i}a_{\l-2}
 -
(-1)^{i}
(2a_{\l-2} +\alpha_{1} +b_{\l-2})   =0, \quad & \text{if $i \leq \l-2$} \\
     \left(   \alpha_1\delta_{i,\text{even}}
 +(-1)^{ i}a_{\l-2} \right)
 +
 (-1)^{i-1}
(2a_{\l-2} +\alpha_{1} +b_{\l-2})  =0 & \text{if $i = \l-1$ }
  \end{array}
\]
which imply
\begin{align}
  b_{\l-2} &= 0  \label{eq:blambda}\\
  a_{\l-2} &=-\alpha_1\delta_{\l,\text{even}}\label{eq:a-alpha}
\end{align}

At $j = i+1 \leq \l-1$ we must have $r=0$ in \eqref{eq:simple-fc} which requires
\begin{align}
      \alpha_1\delta_{i,\text{odd}}
 +(-1)^{ i+1}a_{\l-2}   = 0
\end{align}
or
\begin{align}
   \alpha_1\delta_{i,\text{odd}} = -
 (-1)^{ i+1}a_{\l-2} = (-1)^{i+1}   \alpha_1\delta_{\l,\text{even}}.
\end{align}
If $\l$ and $i$ are both even, or if they are both odd, this implies that $\alpha_1=0$  and the differential collapses.  Note that if $i$ is odd this reduces to \eqref{eq:a-alpha}.  To avoid the collapse of the differential we modify the total order on $I$.
\begin{definition} \label{def:modorder}
Define a total order $\prec$ on the set $I=I_{\l}=\{0, 1, \dots, |\lambda|\}$ by modifying the standard order $i < j$ by declaring that
\begin{align}
  i+1 \prec i \qquad \text{if $i,\l$ are both even, or both odd}.
\end{align}
\end{definition}

With the order $(I,\prec)$, the condition \eqref{eqn:fc.filtration_requirement} becomes
\begin{equation}\label{eqn:fc.filtration_requirement-mod}
v_i \partial(u_j) =0 \text{,  for $i \preceq j$.}
\end{equation}
With this modified order we still must verify that $v_{i+1}\partial(u_i)=0$ when $i$ and $\l$ have the same parity.
Expressed in our previous $i,j$ notation this condition says $v_{i}\partial(u_j)=0$ when $i=j+1\leq \l-1$ and $j,\l$ both even, or both odd.  From \eqref{eq:simple-fc} we see that this amounts to checking that
\begin{align}
 \sum_{r  = \max(0, j+1-\l+2)}^{2} \left( (-1)^{2+ r }  \alpha_1\delta_{j+1+r,\text{odd}}
 +(-1)^{ j}a_{\l-2} \right)
\vcenter{\xy 
    (5,3)*{\smcbub{}}+(-2.5,0)*{\bullet};
    (-2,3)*{\scs r+*};
    (10,0)*{\scs  \l};
    (5,-5)*{\smccbub{}}+(2.5,0)*{\bullet};
    (15,-5)*{\scs *+(2-r) };
    (-8,0)*{};(8,0)*{};
  \endxy }
\end{align}
which requires
\begin{align}
\left(    \alpha_1\delta_{j+1,\text{odd}}
 +(-1)^{ j}a_{\l-2} \right)&=0 
\end{align}
  since the odd bubble squares to zero.  Since we assume $j$ and $\l$ have the same parity this agrees with~\eqref{eq:a-alpha}.

Next we consider the case $i=j=\l$. Using the derivation of the sideways crossing from \eqref{eq:partial-leftcrossing} implies  
\begin{align*}
v_{\l}\partial(u_{\l}) =
(a_{\l} - a_{\l-2})\;\;
 \vcenter{\xy 0;/r.18pc/:
     (-4,-4)*{};(4,4)*{} **\crv{(-4,-1) & (4,1)}?(0)*\dir{<}?(.25)*{\bullet};
     (4,-4)*{};(-4,4)*{} **\crv{(4,-1) & (-4,1)}?(1)*\dir{>};
     (-4,4)*{};(4,12)*{} **\crv{(-4,7) & (4,9)}?(1)*\dir{>};
     (4,4)*{};(-4,12)*{} **\crv{(4,7) & (-4,9)}?(0)*\dir{<};
     (7,8)*{\scs \l};
  \endxy}
\;\; + \;\;
(b_{\l} - b_{\l-2}) \;\;
\vcenter{\xy 0;/r.18pc/:
     (-4,-4)*{};(4,4)*{} **\crv{(-4,-1) & (4,1)}?(0)*\dir{<};
     (4,-4)*{};(-4,4)*{} **\crv{(4,-1) & (-4,1)}?(1)*\dir{>};
     (-4,4)*{};(4,12)*{} **\crv{(-4,7) & (4,9)}?(1)*\dir{>};
     (4,4)*{};(-4,12)*{} **\crv{(4,7) & (-4,9)}?(0)*\dir{<};
     (8,-5)*{\bigotimes};
     (7,8)*{\scs \l};
  \endxy}
\;\; + \;\;
(a_{\l-2} + \alpha_1\delta_{\l,\text{even}})\;\;
\vcenter{\xy 
    (-3,6)*{};(3,-2)*{} **\crv{(-3,3) & (3,1)};
    (3,6)*{};(-3,-2)*{} **\crv{(3,3) & (-3,1)}?(0)*\dir{<};
    (-3,-2)*{};(3,-2)*{} **\crv{(-3.2,-6) & (3.2,-6)}?(.2)*\dir{};
    (-3,-10)*{};(3,-10)*{} **\crv{(-3.2,-4) & (3.2,-4)}?(0)*\dir{<} ?(.2)*\dir{};
    (5,0)*{\scs  \l};
    \endxy}
\end{align*}
Together with \eqref{eq:blambda} and \eqref{eq:a-alpha} the termwise vanishing of the coefficients above imply  that
\begin{align*}
&   a_{\l} = a_{\l-2} = -\alpha_1\delta_{\l,\text{even}}\\
&   b_{\l} = b_{\l-2} = 0
\end{align*}
Then we can further simplify the remaining coefficients from \eqref{eq:clambda} and \eqref{eq:dlambda} to
\begin{align} \label{eq:cd-alpha}
c_{\l} = (-1)^{\l}\alpha_1, \qquad  d_{\l} = \alpha_1
\end{align}
and all the coefficients have been reduced to a single parameter $\alpha_1$.

The only remaining cases are $v_i\partial(u_{\l})$ for $i < \l$.  With the constraints derived thus far it is not hard to show that $\partial(u_{\l})=0$, so that $v_i\partial(u_{\l}) =0$ is satisfied for all $i < \l$.

\begin{definition} \label{def:partial-fc}
Define a bidegree $(2,\bar{1})$ differential $\partial_{\alpha}$  on the space of 2-morphisms of the odd 2-category $\oUcat^{\text{even}}$ or $\l$ in $\oUcat^{\text{odd}}$ given on generating 2-morphisms:
\begin{alignat*}{2}
%
&\partial_{\alpha} \left( \;\;
 \xy 0;/r.17pc/:
 (0,6);(0,-6); **\dir{-} ?(1)*\dir{>}+(2.3,0)*{\scriptstyle{}}
 ?(.1)*\dir{ }+(2,0)*{\scs };
 (0,0)*{ \bullet };
 (8,5)*{ \scs \l};
 \endxy \;\;
\right)
\;  = \;
\alpha\;\;
\xy 0;/r.17pc/:
 (0,8);(0,-8); **\dir{-} ?(1)*\dir{>}+(2.3,0)*{\scriptstyle{}}
 ?(.1)*\dir{ }+(2,0)*{\scs };
 (0,0)*{ \bullet };
 (8,5)*{\scs \l};
 (4,0)*{\scs 2};
 \endxy
\qquad
&&  \partial_{\alpha} \left(
 \xy 0;/r.17pc/:
  (0,0)*{\xybox{
    (-4,-4)*{};(4,4)*{} **\crv{(-4,-1) & (4,1)}?(1)*\dir{>}?(.25)*{};
    (4,-4)*{};(-4,4)*{} **\crv{(4,-1) & (-4,1)}?(1)*\dir{>};
     (7,1)*{\scs \l};
     (-8,0)*{};(9,0)*{};
     }};
  \endxy
	\right) = \;\;
 \alpha  \;\;
 \xy 0;/r.17pc/:
  (3,4);(3,-4) **\dir{-}?(0)*\dir{<}+(2.3,0)*{};
  (-3,4);(-3,-4) **\dir{-}?(0)*\dir{<}+(2.3,0)*{};
  (8,2)*{\scs \l};
 \endxy
 \\
& \partial_{\alpha} \left(\;\;
   \vcenter{\xy 
    (-3,-4)*{};(3,-4)*{} **\crv{(-3.2,2) & (3.2,2)}?(1)*\dir{>};
    (4,2)*{\scs  \l};
  \endxy }
   \; \right)  \; = \;
   -\alpha \delta_{\l,\text{even}} \;
\vcenter{\xy 
    (-3,-4)*{};(3,-4)*{} **\crv{(-3.2,2) & (3.2,2)}?(1)*\dir{>} ?(.2)*\dir{}+(0,0)*{\bullet};
    (4,2)*{\scs  \l};
  \endxy }
\quad
&&
  \partial_{\alpha} \left( \;\;
  \vcenter{\xy 
    (-3,4)*{};(3,4)*{} **\crv{(-3.2,-2) & (3.2,-2)}?(1)*\dir{>};
    (4,-2)*{\scs  \l};
  \endxy }
    \;\right)  \; =\;
  \alpha \delta_{\l,\text{even}}
\vcenter{\xy 
    (-3,4)*{};(3,4)*{} **\crv{(-3.2,-2) & (3.2,-2)}?(1)*\dir{>} ?(.2)*\dir{}+(5,0)*{\bullet};
    (4,-2)*{\scs  \l};
  \endxy }
 \\
& \partial_{\alpha} \left(\;\;
  \vcenter{\xy 
    (-3,-4)*{};(3,-4)*{} **\crv{(-3.2,2) & (3.2,2)}?(0)*\dir{<};
    (4,2)*{\scs  \l};
  \endxy }
    \;\right) \; =\;
(-1)^{\l}\alpha \;\;
\vcenter{\xy 
    (3,-4)*{};(-3,-4)*{} **\crv{(3.2,2) & (-3.2,2)}?(1)*\dir{>} ?(.2)*\dir{}+(0,0)*{\bullet};
    (4,2)*{\scs  \l};
  \endxy }
\;  + \;
\alpha
\vcenter{\xy 
    (6,-4)*{};(-6,-4)*{} **\crv{(5.5,6) & (-5.5,6)}?(1)*\dir{>} ;
    (6,2)*{\scs  \l};
    (0,-2)*{\bigotimes};
    (-8,0)*{};(8,0)*{};
  \endxy }
 \qquad \;\; &&
  \partial_{\alpha} \left( \;\;
  \vcenter{\xy 
    (-3,4)*{};(3,4)*{} **\crv{(-3.2,-2) & (3.2,-2)}?(0)*\dir{<};
    (4,-2)*{\scs  \l};
  \endxy }
  \;  \right) \; =\;
 \alpha \;\;
\vcenter{\xy 
    (-3,4)*{};(3,4)*{} **\crv{(-3.2,-2) & (3.2,-2)}?(0)*\dir{<} ?(.2)*\dir{}+(0,0)*{\bullet};
    (4,-2)*{\scs  \l};
  \endxy }
\; - \;
\alpha \;
\vcenter{\xy 
    (-3,4)*{};(3,4)*{} **\crv{(-3.2,-2) & (3.2,-2)}?(0)*\dir{<} ;
    (6,-2)*{\scs  \l};
    (0,-5)*{\bigotimes};
  \endxy }
\end{alignat*}
\end{definition}

The computations above have established the following.

\begin{proposition} \label{prop:filtration}
Consider either $\oUcat^{\text{even}}$ or   $\oUcat^{\text{odd}}$ and suppose  that $\partial_{\alpha}$ is as in Definition~\ref{def:partial-fc}.  Then  the data $\{u_c, v_c\}_{c\in I}$, with the total order $(I,\prec)$ from Definition~\ref{def:modorder}, yield a fantastic filtration on $\cal{E}\cal{F}\onel$ when $\l \geq 0$ and on $\cal{F}\cal{E}\onel$ when $\l \leq 0$.

\end{proposition}

\section{Covering Kac-Moody algebras}\label{subsec-covering-sl2}
In this section we review the rank one covering Kac-Moody algebra from \cite{CHW}, see also \cite{Clark-IV}.
In subsection~\ref{sec:small} and \ref{sec:qless} we consider specializations at certain roots of unity.  For more on covering Kac-Moody algebras at a root of unity see~\cite{CSW}.

\subsection{Covering quantum group}

Set $\Q(q)^\pi=\Q(q)[\pi]/(\pi^2-1)$.
\begin{definition}
The covering quantum group ${\bf U}_{q,\pi}={\bf U}_{q,\pi}(\mf{sl}_2)$ associated to $\mf{sl}_2$ is the $\Q(q)^{\pi}$-algebra with
generators $E$, $F$, $K$, $K^{-1}$, $J$, and $J^{-1}$ and relations
\begin{enumerate}
  \item  $KK^{-1}=1=K^{-1}K$,
  \quad $JJ^{-1}=1=J^{-1}J$,
  \item  $KE = q^2EK, \qquad \qquad  \;\; KF=q^{-2}FK$,
  \item  $JE = \pi^2EK, \qquad \qquad  \;\;JF=\pi^{-2}FK$,
  \item  $EF-\pi FE=\frac{JK-K^{-1}}{\pi q-q^{-1}}$.
\end{enumerate}
\end{definition}


Define the $(q,\pi)$-analogues of integers, factorials, and binomial coefficients by
\[
[n]=\frac{(\pi q)^n-q^{-n}}{\pi q-q^{-1}},\qquad
[a]!= \prod_{i=1}^{a}[i], \qquad
\left[\!\!
 \begin{array}{c}
   n \\
   a
 \end{array}
 \!\!\right]
 =
 \frac{\prod_{i=1}^a[n+i-a]}{[a]!}.
\]
Note as in \cite{CHW} that
$\left[\!\!
 \begin{array}{c}
   n \\
   a
 \end{array}
 \!\!\right] = \frac{[n]!}{[a]![n-a]!}$ for $n \geq a \geq 0$ and $[-n]=-\pi^n[n]$.  Let $\cal{A}=\Z[q,q^{-1}]$, $\cal{A}_{\pi}=\Z[q,q^{-1},\pi]/(\pi^2-1)$, and $\Q(q)^\pi=\Q(q)[\pi]/(\pi^2-1)$.

The idempoteneted (or modified) form $\U_{q,\pi}$ of the covering algebra ${\bf U}_{q,\pi}$ is obtained by replacing the unit of ${\bf U}_{q,\pi}$ with a collection of orthogonal idempotents $\lbrace 1_\l :\lambda\in\Z\rbrace$ indexed by the weight lattice of ${\rm U}_{q,\pi}$.  In particular, there is no need for generators $K$ or $J$ since
\begin{equation}
 K^{\pm} 1_{\l} = q^{\pm \l} 1_{\l}, \qquad  J^{\pm} 1_{\l} = \pi^{\pm \l} 1_{\l},
\end{equation}
in $\U_{q,\pi}$, see for example \cite[Section 6.1]{ClarkWang} or \cite[Definition 3.1]{Clark-IV}.

\begin{definition}\label{definition-covering-sl2} The \emph{idempotented form} $\U_{q,\pi}$ of \emph{quantum covering $\mf{sl}_2$} is the (non-unital) $\Q(q)^\pi$-algebra generated by orthogonal idempotents $\lbrace1_{\l}:\lambda\in\Z\rbrace$ and elements
\begin{equation}
1_{\l+2} E 1_{\l} =E1_{\l}=1_{\l+2} E,\qquad1_{\l}F1_{\l+2}=F1_{\l+2}=1_{\l}F, \qquad\lambda\in\Z,
\end{equation}
subject to the \emph{covering $\mf{sl}_2$ relation},
\begin{equation}\label{eqn-covering-sl2-relation}
EF1_{\l}-\pi FE1_{\l}=[\l]1_{\l}.
\end{equation}
The \emph{integral idempotented form} is the $\cal{A}_\pi$-subalgebra $\AUdotpi\subset\Udotpi$ generated by the divided powers
\begin{equation}
E^{(a)}1_{\l}=\frac{E^a1_{\l}}{[a]!},\quad1_{\l}F^{(a)}=\frac{1_{\l}F^a}{[a]!}.
\end{equation}
\end{definition}

There are direct sum decompositions of algebras
\[
 \U_{q,\pi} = \bigoplus_{\l,\mu \in \Z}1_{\mu}\U_{q,\pi}1_{\l} \qquad \qquad \UA_{q,\pi} = \bigoplus_{\l,\mu \in
 \Z}1_{\mu}(\UA_{q,\pi})1_{\l}
\]
with $1_{\mu}(\UA_{q,\pi})1_{\l}$ the $\Z[q,q^{-1},\pi]$-subalgebra spanned by $1_{\mu}\E{a}\F{b}1_{\l}$
and $1_{\mu}\F{b}\E{a}1_{\l}$ for $a,b \in \Z_+$.

\subsection{Canonical basis} \label{sec:canonical}
Clark and Wang show in~\cite[Theorem 6.2]{ClarkWang} that the algebra $\U_{q,\pi}$ has a $\cal{A}_{\pi}$-canonical basis $\B_{q,\pi}$, extending Lusztig's basis~\cite[Proposition 25.3.2]{Lus4} for $\mf{sl}_2$,  given by 
\begin{enumerate}[(i)]
     \item $E^{(a)}F^{(b)}1_{\l} \quad \;\;\;\quad $for $a$,$b\in \Z_+$,
     $n\in\Z$, $\l\leq b-a$,
     \item $\pi^{ab} F^{(b)}E^{(a)}1_{\l} \quad$ for $a$,$b\in\Z_+$, $\l\in\Z$,
     $\l\geq
     b-a$,
\end{enumerate}
where $\E{a}\F{b}1_{b-a}=\pi^{ab}\F{b}\E{a}1_{b-a}$.  The importance of this
basis is that the structure constants are in $\N[q,q^{-1},\pi]/(\pi^2-1)$.  In particular, for
$x,y \in \B_{q,\pi}$
\[
 xy = \sum_{x \in \B_{q,\pi}}m_{x,y}^z z
\]
with $z\in \B_{q,\pi}$ and $m_{x,y}^z \in \N[q,q^{-1},\pi]/(\pi^2-1)$.
Let $_{\mu}(\B_{q,\pi})_{\l}$ denote the set of
elements in $\B_{q,\pi}$ belonging to $1_{\mu}(\U_{q,\pi})1_{\l}$.  Then the set $\B_{q,\pi}$ is a union
\[
 \B_{q,\pi} = \coprod_{\l,\mu\in \Z} {}_{\mu}(\B_{q,\pi})_{\l}.
\]

%
\subsection{Quotients of the covering algebra} \label{sec:cov-quotient}
%

The following can be found in \cite[Section 7.3]{ClarkWang}.  For our purposes we take this as the definition of the (super)algebras $\U(\mf{sl}_2)$ and $\U(\mf{osp}(1|2)$.

\begin{proposition} \label{prop:pi-special}
Specializing $\pi=1$, the quotient $\U_{q,\pi}/\la \pi -1\ra$ is isomorphic to the quantum group $\U(\mf{sl}_2)$.  Specializing $\pi=-1$, the quotient $\U_{q,\pi}/\la \pi +1\ra$ is isomorphic to $\U(\mf{osp}(1|2)$ -- the idempotent form of the quantum superalgebra for $\mf{osp}(1|2)$. The canonical basis of $\U_{q,\pi}$ specializes at $\pi=1$, respectively $\pi=-1$, to a canonical basis for $\U(\mf{sl}_2)$, resp. $\U(\mf{osp}(1|2)$\footnote{It is important to note that the positivity of the canonical basis for the superalgebra $\U(\mf{osp}(1|2)$ is quite unexpected and would not be possible without the parameter $\pi$.}.
\end{proposition}

We now describe various further specializations of the $q$ parameter.
Define a quotient of $\cal{A}_{\pi}$ given by
\[
\mathcal{R}=\Z[q,q^{-1},\pi]/(\pi^2-1, 1+q^2\pi).
\]
Here we have set $q^2 = - \pi$ with $\pi^2 =1$.  Hence, at
 $\pi=-1$ we have $q^2=1$ so that $\mathcal{R}=\Z$. At $\pi=1$, $q^2=-1$, so that $\mathcal{R}=\Z[\sqrt{-1}]$. In $\mathcal{R}$ we have $\pi q = -q^{-1}$ so that the $(q,\pi)$ quantum integers become
\begin{equation} \label{eq:Rint}
[n]_{\cal{R}} = \frac{\pi^n q^n-q^{-n} }{\pi q - q^{-1} } =  = \frac{(-1)^n q^{-n}-q^{-n} }{- 2q^{-1} }
=
q^{-n+1}\delta_{n,\text{odd}}.
\end{equation}

Since $\UA_{q,\pi}$ s has an $\cal{A}_{\pi}$-canonical basis (see \cite[Section 7.1]{ClarkWang}) we change base
\begin{equation}
\U_{\mathcal{R}} := \UA_{q,\pi} \otimes_{\cal{A}_{\pi}} \mathcal{R}.
\end{equation}
Equation~\eqref{eq:Rint} implies
\begin{equation}
 E^2 = [2] E^{(2)} = 0, \qquad F^2 = [2]F^{(2)} = 0
\end{equation}
 in $\U_{\mathcal{R}}$.  This implies $E^{a} = F^{a} =0$ in $\cal{R}$ for $a >1$.  Further, from the presentation of $\UA_{q,\pi}$ given in \cite[Proposition 6.1]{ClarkWang} we see that there are no other relations.  Hence, we have  the following.

\begin{proposition} \label{prop:piq-special}
The $\cal{R}$-algebra $\U_{\mathcal{R}}$ has a presentation given as the nonunital associative $\cal{R}$-algebra  given by generators
$
 \left\{
  E1_{\l}, \;\;F1_{\l}, \;\; 1_{\l}, \;\; \l \in \Z
 \right\}
 $
subject to the relations
\begin{enumerate}[(i)]
  \item $1_{\l} 1_{\mu} = \delta_{\l,\mu}$
  \item $E1_{\l} = 1_{\l+2} E, \qquad F1_{\l} = 1_{\l-2}F$,
  \item $EF1_{\l} - \pi FE1_{\l} = [\lambda]_{\cal{R}} 1_{\l}$, \label{eq:cov-quot-rel3}
  \item $E^{2} =0, \qquad F^2=0$. \label{eq:cov-quot-rel4}
\end{enumerate}
Further, $\U_{\mathcal{R}}$ has an $\cal{R}$-basis given by the elements\footnote{Our use of divided power notation is not needed in the case of the fourth root of unity.  We use this notation for ease in converting between the canonical basis at generic $q$. }
\begin{equation}
 \dot{\mathbb{B}}_{\cal{R}} := \left\{ E^{(a)}F^{(b)}1_{\l} \mid a,b \in \{0,1\}, \; \l \leq b-a\right\} \cup
  \left\{\pi^{ab}F^{(b)}E^{(a)}1_{\l} \mid a,b \in \{0,1\}, \; \l \geq b-a\right\},
\end{equation}
over all $\l\in \Z$ with it understood that $\E{a}\F{b}1_{b-a}=\pi^{ab}\F{b}\E{a}1_{b-a}$.
\end{proposition}

The algebra $\U_{\mathcal{R}} $ splits as a direct sum
\[
\U_{\mathcal{R}} = \U_{\mathcal{R}}^{\text{even}} \oplus \U_{\mathcal{R}}^{\text{odd}}
\]
where $\U_{\mathcal{R}}^{\text{even}}$ , respectively $\U_{\mathcal{R}}^{\text{odd}}$ corresponds to the subalgebra containing only even, respectively odd, weights $\l \in \Z$.

%
\subsection{Small quantum $\mf{sl}_2$} \label{sec:small}
%

In this section we connect the covering algebra at parameters $(q,\pi)=(\sqrt{-1},1)$ with the small quantum group.
The small quantum group introduced by Lusztig is a finite dimensional Hopf algebra over the field of cyclotomic integers~\cite{Lus90}. Here we consider the small quantum group at a fourth root of unity.


Let $\sqrt{-1}$ be a primitive fourth root of unity and consider the ring of cyclotomic integers
\begin{equation}
  \Z[\sqrt{-1}] = \Z[q,q^{-1}] /\Psi_4(q) = \Z[q,q^{-1}]/(1+q^2),
\end{equation}
where  $\Psi_n$ denote the $n$th cyclotomic polynomial.  Denote by $\U_{\Z[\sqrt{-1}]}$ the idempotented $\Z[\sqrt{-1}]$-algebra defined by change of basis
\[
 \U_{\Z[\sqrt{-1}]} = \UA \otimes_{\Z[q,q^{-1}]} \Z[\sqrt{-1}].
\]

Set $[k]_{\sqrt{-1}}$ to be the quantum integer $[k]$ evaluated at $\sqrt{-1}$.
 The divided power relation implies that in  $\U_{\Z[\sqrt{-1}]}$ the elements
\begin{equation}
 E^k1_{\l}  = [k]_{\sqrt{-1}}E^{(k)}1_{\l}, \qquad  F^k1_{\l} = [k]_{\sqrt{-1}}F^{(k)}1_{\l}
\end{equation}
are only nonzero when $0 \leq k \leq 2$.

%
%
%
%
%
%

The following Proposition follows immediately from Proposition~\ref{prop:pi-special} and \ref{prop:piq-special}.

\begin{proposition}
The specialization $\U_{\cal{R}}|_{\pi=1} = \U_{q,\pi}|_{\pi=1,q=\sqrt{-1}}$ is isomorphic to the small quantum group
$\dot{u}_{\sqrt{-1}}(\mf{sl}_2)$.
\end{proposition}

%
\subsection{$q$-less subalgebra} \label{sec:qless}
%

In this section we consider the specialization $(q,\pi) = (-1,-1)$, corresponding to setting the quantum parameter $q=-1$ in $\U(\mf{osp}(1|2))$.  We show this specialization has a connection with the superalgebra $\mf{gl}(1|1)$ via its $\mf{sl}(1|1)$ subalgebras.

The quantum group ${\bf U}_q(\mf{sl}(1|1)$ is the unital associative $\Q(q)$-algebra with generators $E$, $F$, $H$, $H^{-1}$ and relations
\begin{equation}
\begin{split}
 & HH^{-1} = H^{-1}H = 1, \\
 & E^2 = F^2 =0, \\
 & HE = EF, \quad HF = FH, \\
 & EF+FE = \frac{H-H^{-1}}{q-q^{-1}}
\end{split}
\end{equation}
This algebra also admits a modified form~\cite{Yin1} given below.
\begin{definition}
The modified form $\U(\mf{sl}(1|1))$ of quantum $\mf{sl}(1|1)$ the (non-unital) $\Q(q)$-algebra obtained from ${\bf U}_q(\mf{sl}(1|1)$ by replacing the unit by a collection of orthogonal idempotents $1_\l$ for $\l \in \Z$ such that
\[
1_{\l} 1_{\mu} = \delta_{\l,\mu}, \quad
H1_{\l} = 1_{\l}H = q^n 1_\l, \quad
1_{\l} E = E1_{\l},  \quad
1_{\l} F = F 1_{\l}
\]
so that
\[
EF 1_{\l} + FE1_{\l} = [\l] 1_{\l},
\]
where here $[\l]$ denotes the usual quantum integer.
\end{definition}

Since the action of $E$ and $F$ does not change the weight space $\l$, there is clearly a decomposition of algebras
\[
 \U(\mf{sl}(1|1))  = \bigoplus_{\l \in \Z} \U(\mf{sl}(1|1))1_{\l}.
\]
The algebra $\U(\mf{sl}(1|1))$ admits an integral form $\UA(\mf{sl}(1|1))$ defined over $\cal{A}=\Z[q,q^{-1}]$.

The relations in $\U(\mf{sl}(1|1))$ are very similar to the relations in $\U_{\cal{R}}$ at parameters $(q,\pi)=(-1,-1)$. However,  there isn't a specialization of $q$ in the usual quantum integers ($\pi=1$) that agree with the $(q,\pi)=(-1,-1)$ covering integers $[n]_{\cal{R}}$.  Instead, we see from \eqref{eq:Rint} that at $q=-1$, the integers $[\l]_{\cal{R}}$ are either 0 or 1.

\begin{proposition} \label{prop:gl}
There are $\Z$-algebra isomorphisms
\begin{equation}\begin{split}
  &\U_{\cal{R}}^{\text{even}}|_{\pi=-1} = \U_{q,\pi}^{\text{even}}|_{(q=-1, \pi=-1)} \cong \U(\mf{sl}(1|1))1_0
\\
  &\U_{\cal{R}}^{\text{odd}}|_{\pi=-1}= \U_{q,\pi}^{\text{odd}}|_{(q=-1, \pi=-1)} \cong \U(\mf{sl}(1|1))1_1
\end{split}
\end{equation}
determined by sending $E1_{\l}, F1_{\l} \in \U_{\cal{R}}$ to the corresponding element in $\U(\mf{sl}(1|1))$.
\end{proposition}

\begin{proof}
 By \eqref{eq:Rint} the quantum integer $[\l]_{\cal{R}}$ at $q=-1$ is either 0 or 1.  The result follows immediately from Proposition~\ref{prop:pi-special} and \ref{prop:piq-special}.
\end{proof}

\begin{remark}
 In Kauffman and Saleur's work constructing the Alexander-Conway polynomial from ${\bf U}_q(\mf{sl}(1|1))$ they restrict their attention to a specialization ($\mathit{\lambda}=1$ in their notation, see \cite[Equation (2.1)]{KaufS}), that corresponds in our notation to restricting to $\U(\mf{sl}(1|1))1_1$.  As noted above, the entire algebra $\U(\mf{sl}(1|1))1_{1}$ has a presentation over $\Z$, rather than $\Q(q)$.  The quantum parameter enters the Alexander story in the work of Kauffman and Saleur via the coproduct on ${\bf U}_q(\mf{sl}(1|1))$.
\end{remark}

Recall the modified form of quantum $\mf{gl}(1|1)$,  defined for example in \cite[Definition 3.2]{TVW}.
\begin{definition}
The \emph{idempotented form} $\U(\mf{gl}(1|1))$ of quantum $\mf{gl}(1|1)$ is the (non-unital) $\Q(q)$-algebra generated by orthogonal idempotents $\lbrace 1_{(\l_1,\l_2)}: (\l_1,\l_2)\in\Z^2\rbrace$ so that
$$
1_{(\l_1,\l_2)} 1_{(\l_1',\l_2')} = \delta_{\l_1,\l_1'} \delta_{\l_2,\l_2'} 1_{(\l_1,\l_2)},$$ and elements
\begin{equation}
\begin{split}
1_{\l_1+1,\l_2-1} E 1_{(\l_1,\l_2)} =E1_{(\l_1,\l_2)}= 1_{(\l_1+1,\l_2-1)} E, \\
1_{\l_1-1,\l_2+1} F 1_{(\l_1,\l_2)} =F1_{(\l_1,\l_2)}= 1_{(\l_1-1,\l_2+1)} F,
\end{split}
\end{equation}
for $(\l_1,\l_2)\in\Z^2$,
subject to the relation,
\begin{equation}
EF1_{(\l_1,\l_2)} +FE 1_{(\l_1,\l_2)} =[\l_1-\l_2]1_{(\l_1,\l_2)} .
\end{equation}
\end{definition}

Note that   the action of $E$ and $F$ preserves the lines in $\Z^2$ of slope $(\l_1-\l_2)$.  In particular, if we restrict to weights $(\l_1, \l_2)$ such that $\l_1-\l_2 = \mu$, then this subalgebra of  $\U(\mf{gl}(1|1))$ is isomorphic to $\U(\mf{sl}(1|1))1_{\mu}$. Hence, we have shown that the covering algebra $\U_{q,\pi}$ specializes at $(q,\pi)=(\sqrt{-1},1)$ to the small quantum group for $\mf{sl}_2$ and to a ``$q$-less subalgebra" of modified $\mf{gl}(1|1)$ at parameters $(-1,-1)$.

%
\section{Categorification results}
%

%
\subsection{Divided power modules }
%

%

Recall the graded superalgebra $\ONH_n$ from section~\ref{sec:oddnil} and the graded superalgebra isomorphism
\begin{equation} \label{eq:ONH2OPOL}
  \ONH_n \cong {\rm END}_{{\rm O}\Lambda_n}(\OPol_n),
\end{equation}
where $\OPol_n$ is the unique (up to isomorphism and grading shift)  graded indecomposable projective $\ONH_n$-supermodule.
 Taking gradings and parity into account, this isomorphism gives rise to a graded supermodule isomorphism
\[
 \ONH_n \cong \bigoplus_{[n]^!_{q,\pi}}(\OPol_n)
\]
where the direct sum over $[n]^!_{q,\pi}$ indicates the direct sum of copies of $\OPol_n$ with appropriate parity and degree shifts, see for example \cite[Lemma 11.1]{BE2}.

In \cite{EllisQi} they equip the superalgebra $\OPol_n$ with a   dg-structure defined by
\begin{equation}
  \partial(x_i)  = x_i^2.
\end{equation}
Denote the resulting $({\rm OPol}_n, {\rm O}\Lambda_n)$-dg-bimodule by $Z_n$,
\[
Z_n := \OPol_n \mathbf{z}
\]
where $Z_n$ is a rank-one free left module with cyclic vector $\mathbf{z}$.
Any dg-supermodule structure on the rank-one free left $\OPol_n$-module is determined by the value of $\partial$ on the cyclic vector.  These are parameterized by $\alpha \in \{0,1\}^a$ \cite[Proposition 3.1]{EllisQi}. In light of \eqref{eq:ONH2OPOL}, a dg-module structure on this module induces a compatible differentials on $\ONH_n$.

\begin{theorem} \hfill \label{thm:dg-onh}
\begin{enumerate}
  \item There is an equivalence of dg-superalgebras (Corollary 3.9 \cite{EllisQi})
\begin{equation} \label{eq:ONH}
 (\ONH_n, \partial) \longrightarrow \END_{{\rm O}\Lambda_n^{{\rm op}}}(Z_n) .
\end{equation}

\item For any $n\geq 0$, $Z_n$ is a finite-cell right dg-supermodule over ${\rm O}\Lambda_n$ (\cite[Proposition 3.16]{EllisQi}).

\item  If $n \geq 2$, then $\ONH_n$ is an acyclic dg-superalgebra.  Consequently, the derived category $\cal{D}(\ONH_n)$ is equivalent to the zero category (\cite{EllisQi} Proposition 3.16).

\item As a left $\ONH_n$ dg module, $Z_n$ is only cofibrant if $n=0,1$ and is acyclic otherwise \cite[Proposition 3.17]{EllisQi}.
\end{enumerate}
\end{theorem}

In light of the above theorem, we denote the dg-module $Z_n$ by $\cal{E}^{(n)}_+$ as \eqref{eq:ONH} gives a dg-categorification of the divided power relation $E^n = [n]^!_{q,\pi}E^{(n)}$ in the covering algebra $\mathbf{U}_{q,\pi}$.  Likewise, one has the dg-module $\cal{E}^{(n)}_-$  which can be realized as the dg $\ONH_n$-module with a conjugate action
\[
 \cal{E}^{(n)}_- := (\cal{E}^{(n)}_+)^{\omega}
\]
where $\omega$ is defined in section~\ref{sec:oddnil}.

%
%
\subsection{The Grothendieck ring of $\mf{U}(\mf{sl}_2)$}
%

This section closely follows Section 5 of \cite{EQ2}.
In what follows we consider the graded dg 2-supercategory $\mf{U} = \mf{U}(\mf{sl}_2)$ and use the notation of the dg-supercategory $_{\mu}\mf{U}_{\l}$ and its corresponding graded dg-superalgebra $_{\mu}\mathbb{A}_{\l}$ interchangeably.  Denote the abelian category of dg-supermodules over $(\mf{U},\partial)$ by $\mf{U}_{\partial}\dmod$. It decomposes into a direct sum of dg-supercategories
\begin{equation}
  \mf{U}_{\partial}\dmod = \bigoplus_{\l,\mu}(_{\mu}\mf{U}_{\l})_{\partial}\dmod.
\end{equation}
Horizontal composition induces induction functors
\begin{align}
 (_{\l_3}\mf{U}_{\l_2} \otimes {}_{\l_2} \mf{U}_{\l_1})_{\partial}\dmod
 &\longrightarrow  ({}_{\l_3} \mf{U}_{\l_1})_{\partial}\dmod
 \\
 \cal{M} \boxtimes \cal{N} & \mapsto \Ind(\cal{M} \boxtimes \cal{N}) \nn
\end{align}
for any $\l_1,\l_2,\l_3,\l_4 \in \Z$.  At the level of derived categories, the induction functor gives rise to an exact functor
\begin{equation}
  \Ind \maps \cal{D}(\mf{U} \otimes \mf{U}, \partial) \longrightarrow \cal{D}(\mf{U},\partial)
\end{equation}
and $\cal{R}$-linear maps
\begin{equation}
  [\Ind] \maps K_0(\cal{D}(\mf{U} \otimes \mf{U}, \partial)) \to K_0(\mf{U},\partial),
\end{equation}
where the Grothendieck group of dg-2-supercateogory is defined in Section~\ref{section:Groth-dg-2cat}.

 %

 For fixed $\l \in \Z$ define
\begin{equation}
{}_{\l}\mf{U} := \bigoplus_{\mu \in \Z} {}_{\mu}\mf{U}_{\l},
\qquad \qquad
\mf{U}_{\l} := \bigoplus_{\mu \in \Z} {}_{\l}\mf{U}_{\mu}.
\end{equation}

%
%

\begin{definition}
  Fix $n\in \N$.
\begin{enumerate}
 \item The left dg-supermodule $\onel\cal{E}^{(n)}$ over $(\oUcat,\partial)$ is the module
 \[
  \onel \cal{E}^{(n)} := \Ind_{\ONH_n}^{{}_{\l}\mf{U}}(\cal{E}^{(n)}_+),
 \]
where the induction comes from the composition of inclusions of superalgebras
\[
 \ONH_n \longrightarrow \END_{{}_{\l}\mf{U}_{\l-2n}}(\onel \cal{E}^n)  \longrightarrow {}_{\l}\mf{U}_{\l-2n}
\]
for each $\l \in \Z$ given by mapping $x_i$ to an upward oriented dot on the $i$-th copy of $\cal{E}$ and $\partial_i$ to an upward oriented crossing of the $i$th and $(i+1)$st term.
 \item The left dg-supermodule $\cal{F}^{(n)}\onel$ over $(\oUcat,\partial)$ is the induced module
 \[
  \cal{F}^{(n)}\onel := \Ind_{\ONH_n}^{\mf{U}{}_{\l}}(\cal{E}^{(n)}_-),
 \]
where the induction comes from the composition of inclusions of superalgebras
\[
\ONH_n \stackrel{\omega} {\longrightarrow} \ONH_n^{{\rm sop}} \longrightarrow \END_{{}_{\l-2n}\mf{U}_{\l}}(\onel \cal{F}^n)  \longrightarrow {}_{\l-2n}\mf{U}_{\l},
\]
 where $x_i^{{\rm sop}}$  and $\partial_i^{{\rm sop}}$ correspond to a downward oriented dot and crossing, respectively.
\end{enumerate}
\end{definition}

The dg-supermodules  $\onel \cal{E}^{(n)}$ and $ \cal{F}^{(n)}\onel$ are  representable $(\mf{U},\partial)$-modules in the sense of Definition~\ref{def:rep}.

\begin{corollary} \label{cor:517}
  Fix $\l \in \Z$ and $n \in \N$.
\begin{enumerate}
  \item The  representable module $\onel \cal{E}^n$ (resp. $\cal{F}^n\onel$) admits an $n!$-step filtration whose subquotients are isomorphic to grading and parity shifts of the divided power module  $\onel \cal{E}^{(n)}$ (resp. $\cal{F}^{(n)}\onel$).

  \item The divided power modules are acyclic whenever $n \geq 2$.

  \item The dg-supermodule $\onel \cal{E}^{(n)}$ (resp. $\cal{F}^{(n)}\onel$) is cofibrant over the dg-supercategory $(_{\l}\mf{U},\partial)$ (resp. $(\mf{U}_{\l})$) for $n=0,1$, and its image in the derived category $\cal{D}(_{\l}\mf{U}, \partial)$  (resp. $\cal{D}(\mf{U}_{\l}, \partial)$) is compact.
\end{enumerate}
\end{corollary}

\begin{proof}
This follows from the corresponding properties of $\cal{E}^{(n)}_+$ and $\cal{E}^{(n)}_-$ from Theorem~\ref{thm:dg-onh}.
\end{proof}

\begin{definition}
For any $a,b \geq 0$ and $\l \in \Z$, define $\cal{E}^{(a)}\cal{F}^{(b)}\onel$ to be the induced dg-module
\[
\cal{E}^{(a)}\cal{F}^{(b)}\onel := \Ind_{\mf{U}_{\l-2b} \otimes \mf{U}_{\l}}^{\mf{U}_{\l}}
 \left(
\cal{E}^{(a)}\onenn{\l-2b} \boxtimes \cal{F}^{(b)}\onel
\right),
\]
with induction defined along the inclusion
\[
\mf{U}_{\l-2b} \otimes \mf{U}_{\l} \longrightarrow  \mf{U}_{\l},
\qquad
\zeta_1 \onenn{\l-2b} \otimes \onenn{\mu}\zeta_2\onel \mapsto \delta_{\l-2b,\mu} \zeta_1 \zeta_2 \onel.
\]
The dg-supermodule $\cal{F}^{(b)}\cal{E}^{(a)}\onel$ is defined similarly.  Following \cite{EQ2} we refer to these modules as {\em canonical modules} over $\mf{U}_{\l}$.
\end{definition}

The fantastic filtrations on $\cal{E}\cal{F}\onel$ and $\cal{F}\cal{E}\onel$ established in section~\ref{sec:fantastic-EF-FE} give rise to a filtration on an arbitrary reprentable module of the form  $Q^r\Pi^{\bar{s}}\cal{E}_{\und{\epsilon}}\onel \in \mf{U}_{\l}$ by dg modules of the form $Q^{u}\Pi^{\bar{v}}\cal{E}^{a}\cal{F}^{b}\onel$ or $Q^{u}\Pi^{\bar{v}}\cal{F}^{b}\cal{E}^{a}\onel$ for $a,b \in \N$ and $u\in \Z$ and $v \in \Z_2$.
Define
\begin{align}
&\mathbb{X}_{\l} := \left\{ \cal{E}^{(a)}\cal{F}^{(b)}\onel \mid a,b \in \{0,1\}, \; \l \leq b-a\right\} \cup
  \left\{ \cal{F}^{(b)}\cal{E}^{(a)}\onel \mid a,b \in \{0,1\}, \; \l \geq b-a\right\}.
\end{align}

\begin{proposition} \label{prop:endo-equiv}
There is a derived equivalence of graded dg-supercategories
\begin{equation}
 \cal{D}(\mf{U}_{\l}) \cong \cal{D}(\END_{\mf{U}_{\l}}(\mathbb{X}_{\l}))
\end{equation}
\end{proposition}

\begin{proof}
The statements in Corollary~\ref{cor:517} apply to the modules $\cal{E}^{(a)}\cal{F}^{(b)}\onel$ and $\cal{F}^{(b)}\cal{E}^{(a)}\onel$; in particular, $\mathbb{X}_{\l}$ consists of compact and cofibrant modules.  Hence, \cite[Proposition 2.10]{EQ2} adapted to the super setting provides the equivalence.
\end{proof}

The cofibrance of the modules in $\mathbb{X}_{\l}$ enables us to compute the derived endomorphism ring $\cal{D}(\END_{\mf{U}_{\l}}(\mathbb{X}_{\l}))$ in the usual manner avoiding cofibrant replacement (see \eqref{eq:cofib}).   The following lemma then follows as a direct consequence of \cite[Proposition 8.2]{Lau-odd}, which shows that the for any $x,y$ in $\mathbb{X}_{\l}$
\[
 \dim \left( \Hom^k_{\mf{U}}(x,y) \right)
\;\; = \;\;
 \left\{
   \begin{array}{ll}
     0 , & \hbox{if $k<0$;} \\
     1, & \hbox{if $x=y$ and $k=0$} \\
     0, & \hbox{ if $x\neq y$ and $k=0$.}
   \end{array}
 \right.
 \]

\begin{lemma} \label{eq:strong_pos}
The endomorphism algebra $\END_{\mf{U}_\l}(\mathbb{X}_{\l})$ is a strongly positive dg-superalgebra.
\end{lemma}


%

\begin{corollary} \label{cor:k0-canonical}
For any weight $\l \in \Z$, the Grothendieck group $K_0(\mf{U}_{\l},\partial)$ of the graded dg-supercategory $\mf{U}_{\l}$ is isomorphic to the corresponding $R$-span of canonical basis elements
\[
 K_0(\mf{U}_{\l}) \cong \cal{R}  \la \dot{\mathbb{B}}_{\cal{R}} 1_{\l} \ra
\]
where
\[
 \dot{\mathbb{B}}_{\cal{R}} 1_{\l}:= \left\{ E^{(a)}F^{(b)}1_{\l} \mid a,b \in \{0,1\}, \; \l \leq b-a\right\} \cup
  \left\{\pi^{ab}F^{(b)}E^{(a)}1_{\l} \mid a,b \in \{0,1\}, \; \l \geq b-a\right\}.
\]
The isomorphism sends the class $\left[Q^0\Pi^{\bar{0}}\cal{E}^{(a)}\cal{F}^{(b)}\onel\right]$ or $\left[Q^0\Pi^{\overline{ab}}\cal{F}^{(b)}\cal{E}^{(a)}\onel\right]$ from $\mathbb{X}_{\l}$ to the corresponding element in $\dot{\mathbb{B}}_{\cal{R}} 1_{\l}$.
\end{corollary}

\begin{proof}
Proposition~\ref{prop:endo-equiv} and  Lemma~\ref{eq:strong_pos} imply that  $\cal{D}(\mf{U}_{\l})$ is equivalent to a positively graded dg-endomorphism algebra.  The result then follows by Theorem~\ref{thm:Ko-positive}.
\end{proof}

As a consequence of strong positivity we also have the following result.

\begin{corollary}
For any weights $\l_1, \l_2, \l_3, \l_4 \in \Z$, the dg-supercategories ${}_{\l_4}\mf{U}_{\l_3}$, and ${}_{\l_2}\mf{U}_{\l_1}$ have the Kunneth property
\[
 K_0 \left(( {}_{\l_4}\mf{U}_{\l_3}) \right) \otimes_{\cal{R}} K_0\left( {}_{\l_2}\mf{U}_{\l_1} \right)
 \cong
 K_0\left( {}_{\l_4}\mf{U}_{\l_3} \otimes {}_{\l_2}\mf{U}_{\l_1}\right).
\]
\end{corollary}

\begin{proof}
Using Lemma~\ref{eq:strong_pos} and \cite[Corollary  2.22]{EQ2} at $p=2$ the result follows.
\end{proof}

It follows that $K_0(\mf{U},\partial)$ is idempotented $\cal{R}$-algebra, with multiplication given by the induction functor:
\[
[Ind] \maps K_0(\mf{U}) \otimes_{\cal{R}} K_0(\mf{U}) \longrightarrow K_0(\mf{U}).
\]

\begin{theorem} \label{thm:main}
  There is an isomorphism of $\cal{R}$-algebras
\begin{equation}
  \U_{\cal{R}}    \longrightarrow K_0(\mf{U},\partial)
\end{equation}
that sends $E1_{\l} \mapsto [Q^0\Pi^{\bar{0}}\cal{E}\onel]$ and $1_{\l}F \mapsto [Q^0\Pi^{\bar{0}}\onel\cal{F}]$ for any weights $\l \in \Z$.
\end{theorem}

\begin{proof}
We first must show that the defining relations for $\U_{\cal{R}}$ hold in $K_0(\mf{U},\partial)$.  The nontrivial relations from Proposition~\ref{prop:piq-special} to check are \eqref{eq:cov-quot-rel3} and \eqref{eq:cov-quot-rel4}. The fantastic filtrations on $\cal{E}\cal{F}\onel$ and $\cal{F}\cal{E}\onel$  from Proposition~\ref{prop:filtration} give rise to convolution diagrams establishing $\eqref{eq:cov-quot-rel3}$ in $\cal{D}(\mf{U},\partial)$ on the corresponding representable modules, see \cite[Remark 2.7, Theorem 6.11]{EQ2}.    Relation \eqref{eq:cov-quot-rel4} follows from the acyclicity results in Corollary~\ref{cor:517}.  The resulting homomorphism of algebras is an isomorphism because it sends $ \dot{\mathbb{B}}_{\cal{R}}1_{\l}$ to the symbols of modules in $\mathbb{X}_{\l}$ which form a basis for  $ K_0(\mf{U},\partial)$ by Corollary~\ref{cor:k0-canonical}.
\end{proof}

\begin{corollary}
 \label{cor:main}
The map sending $E1_{\l} \mapsto [Q^0\Pi^{\bar{0}}\cal{E}\onel]$ and $1_{\l}F \mapsto [Q^0\Pi^{\bar{0}}\onel\cal{F}]$ for any weights $\l \in \Z$ defines
\begin{enumerate}[(i)]
  \item  an isomorphism of $\Z[\sqrt{-1}]$-algebras
\begin{equation}
  \dot{u}_{\Z[\sqrt{-1}]}(\mf{sl}(2)) \longrightarrow K_0(\mf{U},\partial)|_{\pi=1}
\end{equation}
at $\pi=1$, and
\item  an isomorphism of $\Z$-algebras
\begin{equation}
  \U_{\cal{R}}|_{\pi=-1} \longrightarrow K_0(\mf{U},\partial)|_{\pi=-1}
\end{equation}
at $\pi=-1$, where $\U_{\cal{R}}|_{\pi=-1}$ is a $\Z$-subalgebra of $\U(\mf{sl}(1|1))$ by Proposition~\ref{prop:gl}.
\end{enumerate}
\end{corollary}

%

\bibliographystyle{plain}
\bibliographystyle{amsalpha}
\bibliography{bib_odd-diff-new}
 %
%

\end{document}